\documentclass[hidelinks,12pt,reqno]{amsart}
\usepackage[dvipsnames]{xcolor}
\usepackage{xcolor,amssymb,enumerate,tikz-cd,hyperref,amsmath,mathrsfs,mathtools,enumitem,changepage,
mathdots,eucal,xspace,amsthm,parskip,cancel,
amsfonts,fge,xpatch,array,pgfplots,mathabx}
\usetikzlibrary{3d,knots,patterns}
\usepackage[T1]{fontenc}
\usepackage{crimson}

\usepackage[margin=1in]{geometry}
\renewcommand{\baselinestretch}{1.2}
\usetikzlibrary{decorations.markings}
\newcommand{\myarrowlength}{10pt}
\tikzset{mytip/.tip={Butt Cap[black, length=\myarrowlength, sep=-1.6pt]>[black]},
    std/.style={white, text=black, #1, decoration={transform={xshift=.5*\myarrowlength}, markings, mark=at position .5 with {\arrow{mytip}}}, postaction=decorate},
    myarrow/.default={}}
\ExplSyntaxOff

\theoremstyle{definition}

\newtheorem{defn}[subsubsection]{Definition}
\newtheorem{lemdefn}[subsubsection]{Lemma-Definition}

\theoremstyle{theorem}
\newtheorem{prop}[subsubsection]{Proposition}
\newtheorem{theorem}[subsubsection]{Theorem}
\newtheorem{lem}[subsubsection]{Lemma}
\newtheorem{cor}[subsubsection]{Corollary}

\newcommand{\pseudoSmallCaps}[1]{\scalebox{0.8}{\MakeUppercase{#1}}}

\newcommand{\Gm}{\textbf{\textup{G}}_m\xspace}
\newcommand{\Ga}{\textbf{\textup{G}}_a\xspace}
\newcommand{\BGm}{\textup{B}\textbf{\textup{G}}_m\xspace}
\newcommand{\BG}{\textup{B}G\xspace}

\newcommand{\Endl}{\mathcal{E}nd\xspace}
\newcommand{\Homl}{\mathcal{H}om\xspace}
\newcommand{\Md}{\text{-}\textup{Mod}\xspace}

\newcommand{\CoMd}{\text{-}\textup{CoMod}\xspace}
\newcommand{\FactMd}{\text{-}\textup{FactMod}\xspace}
\newcommand{\rank}{\textup{rk}\xspace}

\newcommand{\Ag}{\text{-}\textup{Alg}\xspace}
\newcommand{\CoAg}{\text{-}\textup{CoAlg}\xspace}

\newcommand{\FactCt}{\text{-}\textup{FactCat}\xspace}

\newcommand{\FactAg}{\text{-}\textup{FactAlg}\xspace}
\newcommand{\FactCoAg}{\text{-}\textup{FactCoAlg}\xspace}

\newcommand{\defeq}{\vcentcolon=}

\usepackage[numbers]{natbib}
\newcommand{\protosbt}{\,\begin{picture}(-0.5,1)(-0.5,-2.2)\circle*{2.5}\end{picture}\ }
\newcommand{\sbt}{{\protosbt}}

\newcommand{\overbar}[1]{\mkern 1.5mu\overline{\mkern-1.5mu#1\mkern-1.5mu}\mkern 1.5mu}
\usepackage{xparse,subcaption,hyperref,cleveref}

\ExplSyntaxOn

\NewDocumentCommand{\massdefine}{m}
 {
  \clist_map_inline:nn { #1 }
   {
    \cs_new_protected:cpn { ##1 } { \operatorname{##1}\xspace }
   }
 }

\ExplSyntaxOff

\massdefine{Higgs,Bun,Set,Fun,Maps,loc,Hilb,QCoh,Coh,triv,
BM,Obj,Ext,Hom,Mor,Aut,Spec,pt,gr,Vect,Frac,im,Rep,Gr,corank,
coker,ad,cl,Res,act,Ind,End,ch,ev,tr,Perf,Pic,BGL,
BSL,GL,SL,Ran,Sh,id,Zhu,Mod,coMod,Alg,Art,Supp,colim,Stk,Fuk,
Cob,LocSys,Stab,Perv,IC,IH,cSh,fib,cofib,oblv,cone,cocone,Top,Loc,
Ann,Sym,Gal,IndCoh,LGr,HF,PreStk,Sp,fin,Whit,Dix,Vir,Cat,Div,ComCoAlg,
fact,Lie,CoAlg,HH,HN,Jac,CoHA,Hitch,BD,BV,vol,crit,Cone,Hol,Isom,Sch,reg,
PSL,PGL,PSp,PSO,Proj,ShvCat,Groth,dgCat,Conf,Fin,SES,FinSet,Prim,dR,Spf,Nat,
OrdnPreStk,Aff,OrdnGroth,Spc,aff}
\ExplSyntaxOn

\NewDocumentCommand{\massdefinemathcal}{m}
 {
  \clist_map_inline:nn { #1 }
   {
    \cs_new_protected:cpx { ##1 } { \exp_not:N \mathcal { \tl_range:nnn { ##1 } { 1 } { -2 } } }
   }
 }

\NewDocumentCommand{\massdefinemathfrak}{m}
 {
  \clist_map_inline:nn { #1 }
   {
    \cs_new_protected:cpx { ##1 } { \exp_not:N \mathfrak { \tl_range:nnn { ##1 } { 1 } { -2 } } }
   }
 }
 
 \NewDocumentCommand{\massdefinetext}{m}
 {
  \clist_map_inline:nn { #1 }
   {
    \cs_new_protected:cpx { ##1 } { \exp_not:N \textup { \tl_range:nnn { ##1 } { 1 } { -2 } } }
   }
 }

\NewDocumentCommand{\textbftextup}{m}{%
  \textbf{\textup{#1}}%
}

 \NewDocumentCommand{\massdefinetextbf}{m}
 {
  \clist_map_inline:nn { #1 }
   {
    \cs_new_protected:cpx { ##1 } { \exp_not:N \textbftextup { \tl_range:nnn { ##1 } { 1 } { -2 } } }
   }
 }
 
\ExplSyntaxOff

\massdefinemathcal{Al,Bl,Cl,Dl,El,Fl,Gl,Hl,Il,Jl,Kl,Ll,Ml,Nl,Ol,Pl,Ql,Rl,Sl,Tl,Ul,Vl,Wl,Xl,Yl,Zl}
\massdefinemathfrak{Ak,Bk,Ck,Dk,Ek,Fk,Gk,Hk,Ik,Jk,Kk,Lk,Mk,Nk,Ok,Pk,Qk,Rk,Sk,Tk,Uk,Vk,Wk,Xk,Yk,Zk,
ak,bk,ck,dk,ek,fk,gk,hk,ik,jk,kk,lk,mk,nk,ok,pk,qk,rk,sk,tk,uk,vk,wk,xk,yk,zk,
slk,glk,sok,spk}
\massdefinetext{At,Bt,Ct,Dt,Et,Ft,Gt,Ht,It,Jt,Kt,Lt,Mt,Nt,Ot,Pt,Qt,Rt,St,Tt,Ut,Vt,Wt,Xt,Yt,Zt,
at,bt,ct,dt,et,ft,gt,jt,kt,lt,mt,nt,ot,qt,rt,st,ut,vt,wt,xt,yt,zt,sst,SSt,SUt,SLt,GLt,Prt}
\massdefinetextbf{Ab,Bb,Cb,Db,Eb,Fb,Gb,Hb,Ib,Jb,Kb,Lb,Mb,Nb,Ob,Pb,Qb,Rb,Sb,Tb,Ub,Vb,Wb,Xb,Yb,Zb,
ab,bb,cb,db,eb,fb,gb,hb,ib,jb,kb,lb,mb,nb,ob,pb,qb,rb,tb,ub,vb,wb,xb,yb,zb}
\makeatletter   
\xpatchcmd{\@tocline}
{\hfil\hbox to\@pnumwidth{\@tocpagenum{#7}}\par}
{\ifnum#1<0\hfill\else\dotfill\fi\hbox to\@pnumwidth{\@tocpagenum{#7}}\par}
{}{}
\makeatother    
\makeatletter
 \def\l@subsection{\@tocline{2}{0pt}{4pc}{6pc}{}}
\def\l@subsubsection{\@tocline{3}{0pt}{8pc}{8pc}{}}
 \makeatother
\usetikzlibrary{decorations.pathmorphing,shapes}
\usetikzlibrary{arrows.meta}
\tikzcdset{arrow style=tikz,
    squigarrow/.style={
        decoration={
        snake, 
        amplitude=.4mm,
        segment length=2mm
        }, 
        rounded corners=.2pt,
        decorate
        }
    }
\begin{document}

\setcounter{section}{-1}
\setcounter{subsection}{0}

\author{Alexei Latyntsev}
\title{Factorisation Quantum Groups}
\maketitle

\begin{adjustwidth}{30pt}{30pt}
  \small{A\pseudoSmallCaps{bstract}: We develop vertex and factorisation algebra analogues of the theory of quasitriangular bialgebras. Analogously to the classical theory, we prove their categories of representations are controlled by spectral $R$-matrices. In the vertex algebra case this generalises previous notions due to Etingof-Kazhdan and Frenkel-Reshetikhin. Finally we give examples, including Borcherds twists and homology vertex algebras.}
\end{adjustwidth}
  
\vspace{10mm}

\renewcommand{\baselinestretch}{1.0} % Set the desired line spacing for the table of contents
\tableofcontents
\renewcommand{\baselinestretch}{1.2}

\newpage

\section{Introduction}

\subsection{Preamble} 
Quantum groups are algebraic structures whose category of representations is \textit{braided monoidal}: we may take the iterated tensor product of any representations and there is an isomorphism to the tensor product taken in a different order. Equivalently, its category of representations forms a cosheaf of categories on the space of unordered tuples of points on $\Rb^2$:
\begin{center} \centering 
  \begin{minipage}{.5\textwidth} \centering
    \begin{center}
    \begin{tikzpicture}[every node/.style={font=\small}]
      \node at (-3cm, -0.825cm) {\large{$\Rb$}};
      \node at (0.75cm, -3cm) {\large{$\Rb$}};
      \node at (0.75cm, -3.3cm) {$\ $};
      \node at (0, 0) {\includegraphics[width=6cm]{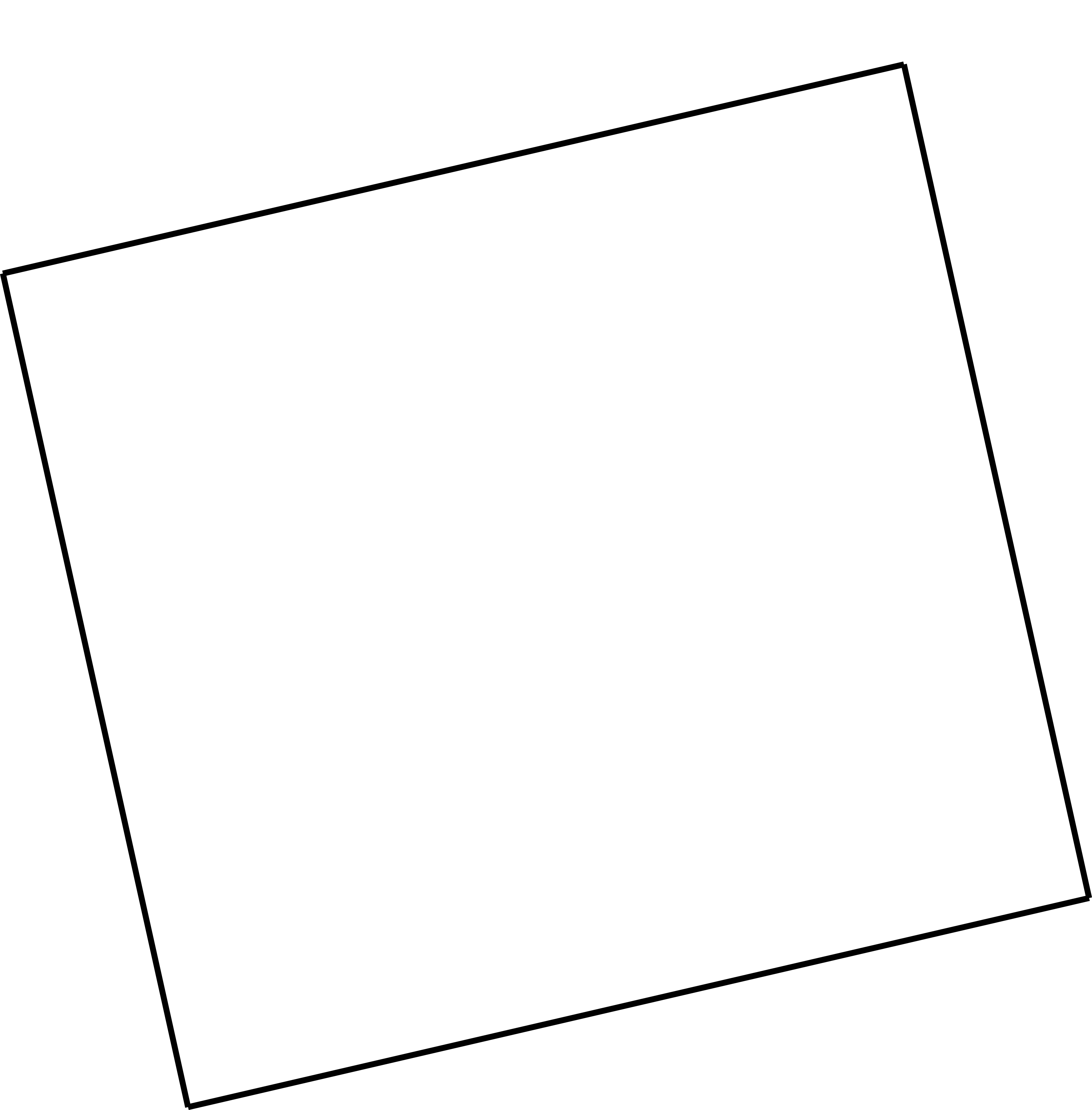}};

      \draw [fill] (1cm,-1cm) circle [radius=2pt] node[above] {$T$};
      \draw [fill] (-1cm,-0.4cm) circle [radius=2pt] node[above] {$U$};
      \draw [fill] (-0.5cm,0.6cm) circle [radius=2pt] node[above] {$V$};
      \draw [fill] (0.7cm,0.4cm) circle [radius=2pt] node[above] {$W$};
      \end{tikzpicture}
  \end{center}
   \end{minipage}% 
  \begin{minipage}{.5\textwidth} \centering 
  
    \begin{tikzpicture}[every node/.style={font=\small}]
      \begin{knot}[
        clip width=5,
      % flip crossing=1,
        flip crossing=2,
      %  flip crossing=3,
      %  flip crossing=4,
        flip crossing=5,
      %  flip crossing=6,
        flip crossing=7,
        flip crossing=8,
      ]
      \strand[black,ultra thick] (2,0) node[above,yshift=-16pt] {$T$} .. controls +(0,1) and +(0,-1) .. (0,3) node[above] {$T$};
      \strand[black,ultra thick] (1,0) node[above,yshift=-16pt] {$U$} .. controls +(0,1) and +(0,-1) .. (0,1.5);
      \strand[black,ultra thick] (0,1.5) .. controls +(0,1) and +(0,-1) .. (1,3) node[above] {$U$};
      \strand[black,ultra thick] (3,0) node[above,yshift=-16pt] {$V$} .. controls +(0,1.5) and +(0,-1.5) .. (2,3) node[above] {$V$};
      \strand[black,ultra thick] (0,0) node[above,yshift=-16pt] {$W$} .. controls +(0,1.5) and +(0,-1.5) .. (3,3) node[above] {$W$};
      \end{knot}
    \end{tikzpicture}
\end{minipage} 
\end{center}
The most well-known examples of quantum groups are deformations $U_q(\gk)$ of universal enveloping algebras of finite dimensional Lie algebras, the Drinfeld-Jimbo quantum groups of \cite{Dr2}.

This paper develops the theory of \textit{factorisation quantum groups}: algebraic structures whose categories of representations ``live over $\Sigma\times \Rb$", where now $\Sigma$ is a Riemann surface and the dependence on the $\Sigma$-direction is holomorphic:
\begin{center} \centering 
  \begin{minipage}{.5\textwidth} \centering
    \begin{center}
    \begin{tikzpicture}[every node/.style={font=\small}]
      \node at (-3cm, -0.825cm) {\large{$\Rb$}};
      \node at (0.75cm, -3cm) {\large{$\Sigma$}};
      \node at (0, 0) {\includegraphics[width=6cm]{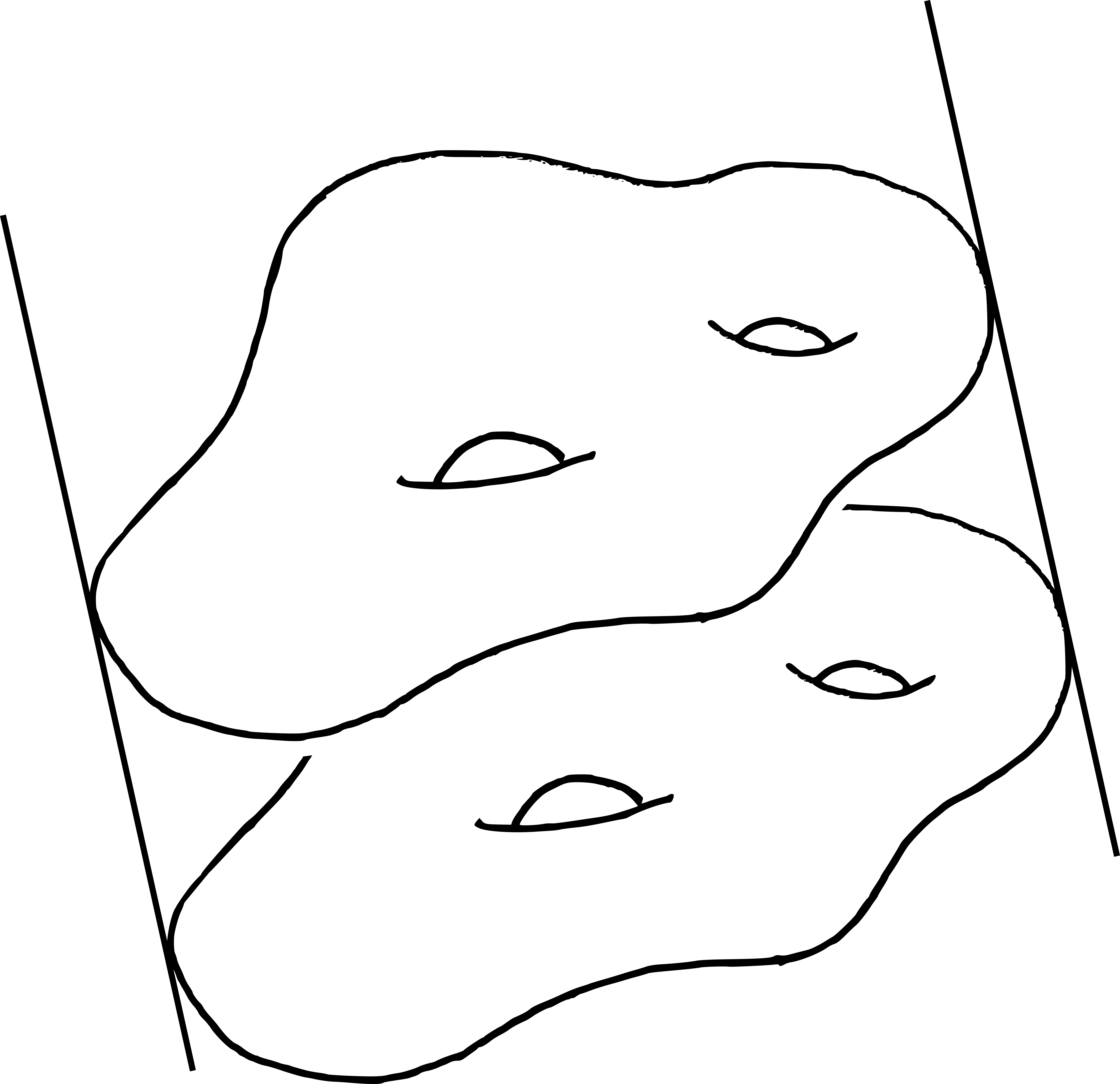}};

      \draw [fill] (1.9cm,-0.35cm) circle [radius=2pt] node[above] {$T$};
      \draw [fill] (-1.5cm,-0.2cm) circle [radius=2pt] node[above] {$U$};
      \draw [fill] (-0.8cm,1.2cm) circle [radius=2pt] node[above] {$V$};
      \draw [fill] (0.7cm,0.4cm) circle [radius=2pt] node[above] {$W$};
      \end{tikzpicture}
  \end{center}
   \end{minipage}% 
  \begin{minipage}{.5\textwidth} \centering 
  
    \begin{tikzpicture}[every node/.style={font=\small}]
      \begin{knot}[
        clip width=5,
      % flip crossing=1,
        flip crossing=2,
      %  flip crossing=3,
      %  flip crossing=4,
        flip crossing=5,
      %  flip crossing=6,
        flip crossing=7,
        flip crossing=8,
      ]
      \strand[black,ultra thick] (2,0) node[above,yshift=-16pt] {$T$} .. controls +(0,1) and +(0,-1) .. (0,3) node[above] {$T$};
      \strand[black,ultra thick] (1,0) node[above,yshift=-16pt] {$U$} .. controls +(0,1) and +(0,-1) .. (0,1.5);
      \strand[black,ultra thick] (0,1.5) .. controls +(0,1) and +(0,-1) .. (1,3) node[above] {$U$};
      \strand[black,ultra thick] (3,0) node[above,yshift=-16pt] {$V$} .. controls +(0,1.5) and +(0,-1.5) .. (2,3) node[above] {$V$};
      \strand[black,ultra thick] (0,0) node[above,yshift=-16pt] {$W$} .. controls +(0,1.5) and +(0,-1.5) .. (3,3) node[above] {$W$};
      \end{knot}
      \node at (-0.3, 1.5) {$z_1$};
      \node at (0.9, 1.95) {$z_2$};
      \node at (1.9, 1.40) {$z_3$};
      \node at (2.8, 1.5) {$z_4$};
    \end{tikzpicture}
\end{minipage} 
\end{center}
We now motivate this, first from a mathematics then from a physics point of view, then describe the contents of this paper in detail.

\newpage 

\subsubsection*{Mathematical motivation} Firstly, a theory of factorisation quantum groups $A$ should give a framework for studying quantum groups, $q$-vertex algebras, affine quantum groups, Maulik-Okounkov Yangians, elliptic quantum groups, $\ldots$ , using the language  of \textit{factorisation} objects due to \cite{AF,BD}, i.e. functorial assignments 
$$U\ \subseteq\ X\hspace{10mm}\rightsquigarrow\hspace{10mm}\Al(U)$$
to each open set a vector space, category, etc.  satisfying a \textit{factorisation} condition 
$$\Al(U_1\sqcup\cdots\sqcup U_n)\ \simeq\  \Al(U_1)\otimes\cdots\otimes\Al(U_n).$$
The relevant $\Al$ will be the category of representations of the above listed objects $A$, or (the Koszul dual of) $A$ itself. 

To be more precise, the category of representations is controlled by an element $R(z)\in A\otimes A((z))$, multiplying by which allows us to \textit{braid} representation of $A$:
$$(M\otimes N)((z))\ \stackrel{\sim}{\to}\ (N\otimes M)((z)).$$
Moreover, the object $A$ itself may be reconstructed using partial information about $R(z)$ and a small amount of extra data $T(z)$, via the \textit{RTT construction} of \cite{MO}. Thus, understanding these \textit{spectral $R$-matrices} conceptually is an important task to which a lot of this paper is devoted.

To illustrate the sort of answer we will give, consider the case of quantum groups as in \cite{Dr2}, where there is no $z$-dependence in the above. The category of representations is thus braided monoidal, i.e. is a locally constant factorisation category over $X=\Rb^2$, where choosing any line gives the category $\Al(D^2)=\Rep A$ a monoidal structure and half-rotating the disks gives a braiding:
\begin{center}
  \begin{tikzpicture}[scale=0.7]
    % Draw square
    \draw (-2,-2) rectangle (2,2);
    
    % Draw the line
    \draw [very thick] (-2,-1.7) -- (2,1.7) node[above] at (0.6*2,0.6*1.7+0.1){};

    %Draw the hatched lines
    \draw[black,pattern=north west lines] (0.2*2,0.2*1.7) circle [radius=0.3];
    \draw[black, pattern=north west lines] (-0.2*2,-0.2*1.7) circle [radius=0.3];
    \draw[thick] (0,0) circle [radius=1.2];

    \node at (0,-2.7) {$\otimes_\Al\ :\ \Al(D^2)\otimes \Al(D^2)\ \to\ \Al(D^2)$};

    \begin{scope}[shift={(10,0)}]
    % Draw square
    \draw (-2,-2) rectangle (2,2);
    
    % Draw the line
    \draw [very thick] (-2,-1.7) -- (2,1.7) node[above] at (0.6*2,0.6*1.7+0.1){};

    %Draw the hatched lines
    \draw[black,pattern=north west lines] (0.2*2,0.2*1.7) circle [radius=0.3];
    \draw[black, pattern=north west lines] (-0.2*2,-0.2*1.7) circle [radius=0.3];
    \draw[thick] (0,0) circle [radius=1.2];
    \draw[->] (0,0) ++(-130:1.5) arc (-130:30:1.5);
    \draw[->] (0,0) ++(50:1.5) arc (50:210:1.5);

    \node at (0,-2.7) {$\beta\ :\ \otimes_\Al \ \stackrel{\sim}{\to} \ \otimes_\Al\cdot\sigma$};
    \end{scope}
\end{tikzpicture} 
\end{center}
One then shows that a module category is braided monoidal if and only if $A$ is a bialgebra with an element $R\in A\otimes A$ satisfying among other things the \textit{Yang-Baxter equation}:
$$R_{12}R_{13}R_{23}\ =\ R_{23}R_{13}R_{12}.$$
Thus clearly the most important structure here is the braided monoidal structure on $\Rep A$.

To continue to a non-locally-constant example, we consider the \textit{affine quantum group} $U_q(\hat{\gk})$ from \cite{Dr1}, which contains the Drinfeld-Jimbo quantum group $U_q(\gk)$ as a subalgebra:
\begin{center}
\begin{tikzcd}[row sep = {20pt}, column sep = {20pt}]
    U(\hat{\gk}) \ar[r,hook]& U_q(\hat{\gk}) \\ 
   U(\gk) \ar[r,hook] \ar[u,hook]& U_q(\gk)\ar[u,hook]
\end{tikzcd}
\end{center}
To begin, assume $\gk$ is the Lie algebra of simple algebraic group $G$. Then the affine Lie algebra $\hat{\gk}$ is a central extension of the Lie algebra of the group of \textit{loops}
$$S^1\ \to\ G$$
where as a model for $S^1$ we take the punctured formal disk $\Spec \Cb((z))$, as in \cite{KV}, and its universal enveloping algebra $U(\hat{\gk})$ is a subalgebra of modes of a vertex algebra $V(\gk)$. 

Vertex algebras were introduced by Borcherds \cite{Bo1}, but later discovered by Beilinson-Drinfeld \cite{BD} to be holomorphic factorisation algebras over $X=\Cb$, thus connecting them to the geometry of $\Ab^1=\Spec \Cb[z]$. This is defined as a holomorphic analogue of Lurie's equivalence \cite{Lu1} between locally constant factorisation algebras on $X$ and constructible cosheaves on the space $\Ran X$ of finite subsets of $X$.

It is thus reasonable to expect $U_q(\hat{\gk})$ to be a subalgebra of modes of a vertex quantum group $A$.\footnote{We leave the construction of $A$ to future work, instead most  examples in this paper will be Borcherds twists. Also, we expect $U_q(\gk)$ should come from a multiplicative vertex quantum group, but we ignore this distinction in the introduction, e.g. we only write the additive spectral Yang-Baxter equation.} The structure on its category of representations will be defined in the main text, but it is loosely speaking a ``holomorphic-topological factorisation category over $\Sigma\times \Rb$'', so restricting to the Riemann surface $\Sigma$ we get a tensor product depending holomorphically on a pair of points, and a braiding induced by the topological direction, which we \textit{heuristically} sketch using the diagrams in $\Sigma\times\Rb$:
\begin{center}
  \begin{tikzpicture}[scale=0.7]
    % Draw square
    \draw (-2,-2) rectangle (2,2);
    
    % Draw the line
    \draw [very thick] (-2,-0) -- (2,0) node[above] at (0.6*2,0.6*1.7+0.1){};

    %Draw the hatched lines
    \draw[black,pattern=north west lines] (0.5257,0) circle [radius=0.3];
    \draw[black, pattern=north west lines] (-0.5257,-0) circle [radius=0.3];
    \draw[thick] (0,0) circle [radius=1.2];

   \node at (0,-2.7) {$\otimes_{\Al,z,w}\ :\ \Al(D^2)\otimes \Al(D^2)\ \to\ \Al(D^2)$};
    \node at (-1.7,0.4) {$\Sigma$};

    \begin{scope}[shift={(10,0)}]
    % Draw square
    \draw (-2,-2) rectangle (2,2);
    
    % Draw the line
    \draw [very thick] (-2,-0) -- (2,0) node[above] at (0.6*2,0.6*1.7+0.1){};

    %Draw the hatched lines
    \draw[black,pattern=north west lines] (0.5257,0) circle [radius=0.3];
    \draw[black, pattern=north west lines] (-0.5257,-0) circle [radius=0.3];
    \draw[thick] (0,0) circle [radius=1.2];
    \draw[->] (0,0) ++(-165:1.5) arc (-165:-15:1.5);
    \draw[->] (0,0) ++(15:1.5) arc (15:150:1.5);

   \node at (0,-2.7) {$\beta\ :\ \otimes_{\Al,z,w} \ \stackrel{\sim}{\to} \ \otimes_{\Al,w,z}\cdot\sigma$};
    \node at (-1.7,0.4) {$\Sigma$};
    \end{scope}
\end{tikzpicture} 
\end{center} 
We then show that endowing $\Rep A$ with this structure is equivalent to $A$ being a vertex bialgebra and picking an element $R\in A\otimes A ((z))$ satisfying among other things the spectral Yang-Baxter equation, which when $\Sigma=\Ab^1$ is
$$R_{12}(z)R_{13}(z+w)R_{23}(w)\ =\ R_{23}(w)R_{13}(z+w)R_{12}(z).$$
Note that quantum groups are not factorisation algebras over $\Rb^2$ themselves, but rather their $\Eb_2$-Koszul duals are. Likewise, we expect that there is a notion of holomorphic-topological Koszul dual of vertex quantum groups recovering the \textit{raviolo vertex algebras} of Garner and Williams \cite{GW}.

Aspects of factorisation quantum groups have appeared elsewhere in mathematics. When Borcherds first defined lattice vertex algebras, it was twisting by a spectral $R$-matrix, a generalisation of which we treat in the main text. In a recent example, Joyce \cite{Jo} has used this to construct vertex algebra structures on the homology of moduli stacks, which together with its cohomological Hall algebra forms a (cocommutative) vertex quantum group by \cite{La,Li}.

Analogously to how the theory of braided monoidal categories contains $\Rep U_q(\gk)$ as an example, so too we expect the different definitions and examples of $q$-vertex algebra in \cite{EK,FR,Jo} to relate to the theory of factorisation quantum groups. One good property of the theory is that, like Beilinson-Drinfeld factorisation algebras, there are no extraneous choices made when defining factorisation quantum  groups. In the text we only take $\Eb_n$-algebras in certain categories, and then extract the vector-space-and-data definition as a consequence. 

Finally, we note that factorisation category type structures on categories of vertex algebra modules have appeared in the literature. For instance, the notion of vertex tensor product due to Creutzig, Kanade and McRae in \cite{CKM}, and Soibelman's meromorphic tensor categories in \cite{So}. Finally, there is a braided monoidal structure on the category of modules for a special class of vertex algebras, due to Huang,  Lepowsky, and Zhang \cite{HLZ}.

\subsubsection*{Physics motivation} The quantum field theories relevant to this paper are \textit{$3d$ topological-holomorphic quantum field theories} (without boundary). 

To begin, \textit{$2d$ conformal field theories} attach data to cobordisms given by Riemann surfaces $\Sigma$, and give rise to vertex algebras by evaluating on a ``small'' circle $S^1$ then taking holomorphic part.

Likewise, $3d$ topological-holomorphic theories attach data to $3$-topological manifolds like $\Sigma\times \Rb$ equipped with a foliation by Riemann surface leaves, evaluating on a ``small'' sphere $S^2$ should give rise to a raviolo vertex algebra \cite{GW}, and its category of line operators should be representations of a quantum vertex group. We may dimensionally reduce to obtain a $2d$ CFT:
\begin{center}
  \begin{tikzcd}[row sep = {20pt}, column sep = {20pt}]
    \Sigma\times \Rb \ar[r,hook]\ar[d,rightsquigarrow]&\Sigma\times \Rb\times\Rb_{\ge 0} \ar[d,rightsquigarrow] \\ 
   \Sigma \ar[r,hook] & \Sigma\times\Rb_{\ge 0}
  \end{tikzcd}
  \end{center}

One can then consider theories admitting the above as boundaries. First, there are certain $3d$ TQFTs with holomorphic boundary condition. The basic example is $3d$ Chern-Simons theory with WZW (affine vertex algebra) or oper (W vertex algebra) boundary conditions. There is a large body of work about other ``boundary vertex algebras'' arising as boundaries of $A$- or $B$-topological twists of $3d$ $\Nl=2$ gauge theories, e.g. \cite{CDG,BCDN}. Second are the \textit{four}-dimensional topological-holomorphic theories. For instance, in \cite{CWY1,CWY2} Costello, Witten and Yamazaki study categories of line operators in $4d$ Chern-Simons theory, whose category of line operators they relate to $\Rep U_q(\hat{\gk})$. One expects that there are boundary conditions giving the $3d$ theories referred to above. Note that in $3d$, by the Kazhdan-Lusztig Theorem there is an braided monoidal equivalence between the category $\Rep U_q(\gk)$ and the category of $V(\gk)$-modules. By analogy it is reasonable to expect that the category of line operators in appropriate $4d$ topological-holomorphic theories have the structure of braided factorisation categories of the form $\Rep A$ for $A$ a vertex quantum group.

\subsection{Summary of sections}This paper is divided into four parts.

In section \ref{sec:HolEnAlgs} we briefly make all our definitions in a homotopy-coherent way. We define factorisation spaces, categories, algebras and modules. 

In section \ref{sec:SpectralRMatrices} we carry over the main results in the basic theory of quantum groups to live ``over'' factorisation spaces. We first give more explicit forms of our main definitions (sections \ref{sec:ClassicalFactorisationSpaces} and \ref{sec:ClassicalFactorisationCategories}). Then, we prove Theorem \ref{thm:SpectralRMatrixBraidedMonoidal} classifying factorisation braidings on $A\Md$, which we use as our definition of factorisation quantum group, spectral $R$-matrix, etc. The definition of lax braiding is more subtle in the factorisation setting, which is covered in section \ref{sec:ClassicalFactorisationCategories}.

In section \ref{sec:QuantumVertexGroups} we write down the theory of quantum groups over the Ran space explicitly. We show  in Theorem \ref{thm:TannakaFactorisation} that our theory recovers the classical notion of vertex quantum groups and vertex $R$-matrices. We develop the theory of braided-commutative vertex algebras in section \ref{sec:ReconstructionTheorem} and Appendix \ref{sec:RationalMultiplicativeElliptic}, where we also review known facts about vertex algebras.

Finally, in section \ref{sec:BorcherdsTwists} we show how to produce new examples of factorisation quantum groups by Borcherds twisting. In section \ref{sec:Examples} we give some basic examples of all relevant concepts in the paper; for instance in section \ref{sec:ModuliSpaces} we give examples coming from the homology of moduli stacks. There are many other examples of factorisation quantum groups one expects to exist, whose construction we leave to future work.

\subsection{Background and conventions} 

\subsubsection{$\Eb_n$-algebras} \label{sec:EnAlgs}$\Eb_n$-algebras are algebraic structures which generalise associative or commutative algebras, corresponding to $n=1$ and $n=\infty$. For an account of $\Eb_n$ as an $\infty$-operad, as well as an introduction to $\infty$-operads, their algebras and modules over them, see sections $\S$ 2-4 and 5 of \cite{Lu2}. We sketch the one-categorical analogue here. 

 $\El_n$ was first described in \cite{May}, given by a sequence of topological spaces
 $$\El_n(k)\ =\ \left\{ \textup{rectilinear embeddings }\sqcup_{i=1}^k \Box\ \hookrightarrow\ \Box\right\},$$
 where $k\ge 0$ and $\Box$ is the unit $n$-cube. This forms a symmetric topological operad. Thus if $\Cl$ is a topologically enriched $1$-category, an $\El_n$-algebra is an object $A\in \Cl$ together with continuous maps
 $$\El_n(k)\ \to\ \Hom(A^{\otimes k},A),$$
which are compatible as we vary $k\ge 0$ or act by the symmetric group $\Sk_k$ on both sides. One can also define the set-valued symmetric operad, which we also denote by
$$\pi_0\El_n(k)\ =\ \pi_0\left\{ \textup{rectilinear embeddings }\sqcup_{i=1}^k \Box\ \hookrightarrow\ \Box\right\}.$$
Then for any symmetric monoidal $1$-category $\Cl$, a $\pi_0\El_n$-algebra corresponds to an object when $n=0$, an associative algebra object when $n=1$, and a commutative algebra object when $n\ge 2$. We may view $\Cl$ as an $\infty$-category by taking its nerve, and the analogues of these statements in the language of $\Eb_n$-algebras are \cite[4.1.6.18, 5.1.1.7]{Lu2}.

\subsubsection{Categories} \label{sec:Cat} By default all categories considered will be $(\infty,1)$-categories, although after section \ref{sec:HolEnAlgs} we will usually consider only examples which are ordinary $1$-categories or $(2,1)$-categories. In that case all definitions will be translated into a more concrete equivalent form, which readers unfamiliar with higher category theory can safely take as the definition.

\subsubsection{Prestacks} \label{sec:PreStk} When we refer to ``space'' in this paper we will mean a \textit{prestack}. One should think of a prestack $X$ as a black box which takes as input a test scheme $T$ and produces a topological space $X(T)$, with no extra conditions other than functoriality. To be precise, it is an element of the $(\infty,1)$-category
$$\PreStk\ =\ \Fun\left((\Sch^{\textup{aff}})^{\textup{op}},\textup{Spc}\right)$$
where $\text{Sch}^{\textup{aff}}$ and $\textup{Spc}$ are the $(\infty,1)$-categories of derived affine schemes and topological spaces, respectively. These definitions are from \cite[$\S$ 1.8]{GR}. We also have the $(k,1)$-category 
$$\PreStk^{\textup{cl}}_{\le k}\ \hookrightarrow\ \PreStk$$
of \textit{classical $k$-prestacks} where $k\ge 0$, defined  as
$$\PreStk^{\textup{cl}}_{\le k}\ =\ \Fun\left((\Sch^{\textup{aff}}_{\le 0})^{\textup{op}},\textup{Spc}_{\le k}\right)$$
where we use the full subcategories $\Sch^{\textup{aff}}_{\le 0}\subseteq \Sch^{\textup{aff}}$ of \textit{classical} affine schemes and $\Spc_{\le k}\subseteq \Spc$ the spaces with vanishing homotopy groups in degrees $>k$. Thus classical $0$-prestacks incude schemes, and classical $1$-prestacks include Artin stacks. When $k=1$ we omit it from the notation.The reason we consider prestacks is that our main example of a factorisation space is the Ran space $\Ran X$, and this is only defined as a prestack, even when $X$ is a scheme.

\subsubsection{Sheaves of categories} We now categorify the notion of $\Ol$-module, 
$$\Ol_X\Md\hspace{10mm}\rightsquigarrow\hspace{10mm}\ShvCat(X).$$
 The guiding principle is to replace the algebra $\Ol(X)$ with the monoidal dg-category $\QCoh(X)$. Thus when $X$ is a derived affine scheme, we set 
$$\Ol_X\Md = \Ol(X)\Md, \hspace{15mm}\ShvCat(X)\ =\ \QCoh(X)\Md$$
where the latter is the $(\infty,2)$-category of module categories over $\QCoh(X)$. We define $\QCoh(X)$ and $\ShvCat(X)$ for general prestacks $X$ by Kan extending, see \cite{Ga1}.

\subsubsection{Correspondences} \label{sec:Correspondences} We recall \cite[Chap. 7]{GR}. If $\Cl$ is an $(\infty,1)$-category with arbitrary pullbacks, we may consider the $(\infty,2)$-category $\textup{Corr}(\Cl)^{\textup{all}}_{\textup{all};\textup{all}}$ constructed in \cite[$\S$ 7.1.2]{GR}. The objects are objects of $\Cl$, one-morphisms are diagrams of one-morphims in $\Cl$
\begin{center}
\begin{tikzcd}[row sep = {30pt,between origins}, column sep = {45pt,between origins}]
 & c'\ar[rd]\ar[ld]& \\ 
c_1 & & c_2
\end{tikzcd}
\end{center}
 two-morphisms are given by diagrams of two-morphisms in $\Cl$
\begin{equation}\label{fig:CommMon}
\begin{tikzcd}[row sep = {30pt,between origins}, column sep = {45pt,between origins}]
 &[-5pt]c'\ar[rdd, bend left = 30,""{name=UR,inner sep=1pt,{xshift=-5pt},{yshift=5pt}}]\ar[ldd,swap, bend right = 30,""{name=UL,inner sep=1pt,{xshift=5pt},{yshift=5pt}}]\ar[d] &[6pt] \\
 &[-5pt]c''\ar[rd,""{name=DR,inner sep=1pt}]\ar[ld,""{name=DL,inner sep=1pt}] & \\
c_1 & & c_2
\arrow[Rightarrow, from=UR, to=DR, "", shorten <= 7pt, shorten >= 7pt] 
\arrow[Rightarrow, from=UL, to=DL, "", shorten <= 7pt, shorten >= 7pt] 
\end{tikzcd}
\end{equation}
and so on. When $\Cl$ is a $1$-category, the above is a $2$-category. We will denote by $\Cl^{\textup{corr}}$ the $(\infty,1)$-category formed by discarding all non-invertible two-morphisms, which is necessary for reasons discussed in section \ref{sec:Straightening}. Moreover, when working with factorisation algebras on prestacks, we need to in addition assume that the vertical morphisms are quasi-compact and quasi-separated, so will use the notation $\PreStk^{\textup{corr}}$ to refer to $\textup{Corr}(\PreStk)^{\textup{all}}_{\textup{qcqs};\textup{all}}$ with non-invertible $2$-morphisms discarded. The reason is that we need to be able to push forward sections of sheaves of categories, which by \cite[Prop A.9.1]{Ras} we may do so along such morphisms. 

\subsubsection{Straightening and unstraightening} \label{sec:Straightening} In mathematics, the (un)straightening or \textit{Grothendieck} construction refers to phenomenona where data $d\in \Dl$ living over a space $X$ may be encoded as a map $X\to \Bt\Dl$ into a ``classifying'' space for that data, by pulling back universal data over that space. For example, every principal $G$-bundle is obtained as the pullback along a (unique) map into $\BG\simeq \pt/G$:
\begin{center}
\begin{tikzcd}[row sep = {30pt,between origins}, column sep = {45pt,between origins}]
P\ar[d]&[20pt] P\ar[r]\ar[d] & \pt\ar[d]&[20pt]&\\ 
X&X\ar[r] & \BG&X\ar[r]&\BG
\end{tikzcd}
\end{center}
We will use various Grothendieck constructions, where $X$ is a category and the data $\Dl$ consists of a certain type of categories living over it. It will be important for us these constructions be compatible with varying the choice of $\Dl$: by \cite[Thm. 3.2.0.1]{Lu1} or \cite{HHR}, if $X$ is any $(\infty,1)$-category, there is an  equivalence 
  \begin{equation}\label{eqn:UnStraightening}
    \textup{CoCart}(X)\ \simeq\ \Fun(X,\Cat_\infty)
  \end{equation}
  between the $(\infty,1)$-category of functors $X\to \Cat_\infty$ and the category of \textit{cocartesian fibrations} $P\to X$ as defined in \cite[$\S$ 2.4]{Lu1}. This equivalence (\ref{eqn:UnStraightening}) is upgraded to an equivalence of symmetric monoidal $(\infty,1)$-categories by \cite{GL,Ram}. In particular, if $\Ol$ is any $\infty$-operad and $X$ has an $\Ol$-monoidal structure then there is an equivalence 
  \begin{equation}\label{eqn:OperadUnStraightening}
    \textup{CoCart}_\Ol(X)\ \simeq\ \Fun_\Ol(X,\Cat_\infty)
  \end{equation}
  between the $(\infty,1)$-category of lax $\Ol$-monoidal functors $X\to\Cat_\infty$ and \textit{$\Ol$-monoidal cocartesian fibrations} $P\to X$, as defined in \cite[1.11]{Ram}. This is compatible with (\ref{eqn:UnStraightening}) via the forgetful functor on both sides.

There is a version of this for when $X$ is an $(\infty,2)$-category. In this case \cite[Thm. 1.1.8]{GR} or \cite{Ab} gives an equivalence 
  \begin{equation}\label{eqn:UnStraightening2}
    \textup{CoCart}_2(X)\ \simeq\ \Fun(X,\Cat_{(\infty,2)})
  \end{equation}
  between the $(\infty,2)$-category of (right lax) functors $X\to \Cat_{(\infty,2)}$ and the category of (\textit{$2$-})\textit{strict cocartesian fibrations $P\to X$} as defined in \cite[$\S$ 11.1]{GR}. Here $\Cat_{(\infty,2)}$ is the $(\infty,2)$-category of $(\infty,2)$-categories.

  Finally, we would like a symmetric monoidal upgrade of (\ref{eqn:UnStraightening2}), but unfortunately as far as we are able to tell the language of $\Eb_n$-algebras and operads inside $(\infty,2)$-categories have not yet been fully developed yet. Thus in section \ref{sec:SpectralRMatrices} we will make our concrete definitions of lax classical factorisation spaces and categories by hand, which just as in the strong case should be straightforwardly shown to be equivalent to $\Eb_n$-algebra type definitions after the language is developed enough to make them.

\subsubsection{Vertex algebra nomenclature} \label{sec:VertexAlgNomenclature} For an introduction to vertex algebras, see \cite{FBZ} or Appendix \ref{sec:RationalMultiplicativeElliptic}. There we will define an \textit{associative} or \textit{commutative vertex algebra} as a translation-equivariant associative or commutative factorisation algebra in $\Dl\Md_{\Ran^{ch} \Ab^1}$. Unfortunately, this departs from the usual naming conventions:
\begin{center}
  \begin{tabular}{c|c}
    Usual name, e.g. \cite{FBZ,Jo,EK} & Our name \\ \hline
    \textit{Vertex algebra}& \textit{Commutative vertex algebra} \\
    \textit{Nonlocal vertex algebra}& (\textit{Associative}) \textit{vertex algebra} \\
    \textit{Holomorphic vertex algebra}& \textit{Holomorphic commutative vertex algebra} \\
  \end{tabular}
\end{center}
As usual, we use the term \textit{holomorphic} for factorisation $\Eb_n$-algebras in $\Dl\Md_{\Ran^*\Ab^1}$. 

\subsection{Acknowledgements} We would like to thank Yuchen Fu for patiently explaining many facts about $\Eb_n$-algebras, Kevin Lin for reading an early draft of this paper and useful conversations, Dominic Joyce for suggesting a related question as my PhD advisor which eventually led to this paper, as well as Emile Bouaziz, Lukas Brantner, Henry Liu, and Vivek Shende for helpful discussions. This work was completed at the Centre for Quantum Mathematics at SDU, supported by the Danish National Research Foundation (DNRF157) and Villum Fonden (37814). Early stages of this work were at Oxford University, supported by the EPSRC.

\newpage

\section{Basic Definitions} \label{sec:HolEnAlgs} 

\noindent
In this section we define factorisation spaces, categories, algebras, modules, and their basic properties. See also \cite{Ras} for a related reference.  

\subsubsection*{Motivating example} Let $\Ll$ be a complex line bundle over a group $G$. A \textit{multiplicative} structure on it is a pair of maps
$$\Ll\boxtimes\Ll\ \stackrel{}{\to} \ m^*\Ll, \hspace{15mm} \triv\ \stackrel{}{\to}\ e^*\Ll$$
such that the first satisfies an associativity condition and both maps are compatible in the obvious way. Here $m:G\times G\to G$ and $e:\pt\to G$ are the multiplication map and unit.   Thus, for all $g_1,g_2\in G$ we can take fibres of the above to get linear maps
$$\Ll_{g_1}\otimes \Ll_{g_2}\ \to\ \Ll_{g_1\cdot g_2}, \hspace{15mm} \Cb\ \to\ \Ll_e.$$
One may likewise define multiplicative sheaves, sheaves of categories, spaces, algebras, functions, etc.

\noindent Generalising further, we may define this structure when the product $(g_1,g_2)\mapsto g_1\cdot g_2$ is only be partially defined, or more generally depends on a choice of some auxiliary data $c_{g_1,g_2}$. In other words, we replace $m$ with a pair of maps
\begin{center}
\begin{tikzcd}[row sep = 10pt, column sep = 15pt]
 &C\ar[rd]\ar[ld] & \\
G\times G & & G 
\end{tikzcd}
\end{center}
sending $c_{g_1,g_2}\in C$ to $(g_1,g_2)$ and an element denoted $g_1\cdot_cg_2$, respectively, and satisfying an associativity condition and compatibly with the unit. 
%\begin{center}
%  \begin{tikzcd}[row sep = 10pt, column sep = 15pt]
%   &C\ar[rd]\ar[ld] & &[30pt] &c_{g_1,g_2}\ar[ld,|->]\ar[rd,|->]& \\ 
%  G\times G & & G & (g_1,g_2)&& g_1\cdot_c g_2
%  \end{tikzcd}
%  \end{center}
%\begin{center}
%\begin{tikzcd}[row sep = 10pt, column sep = 15pt]
% &c_{g_1,g_2}\ar[ld,|->]\ar[rd,|->]& \\ 
% (g_1,g_2)&& g_1\cdot_c g_2
%\end{tikzcd}
%\end{center}
The analogue of a multiplicative structure on line bundle $\Ll$ is a \textit{factorisation} structure, which on fibres give for all elements $c_{g_1,g_2}\in C$ linear maps
$$\Ll_{g_1}\otimes \Ll_{g_2}\ \to\ \Ll_{g_1\cdot_c\, g_2},\hspace{15mm}k\ \to\ \Ll_e.$$
Two basic examples are the topological space $G=\Ran T$ of finite subsets in a topological space $ T$, the multiplication map given by taking unions on the open locus $C\subseteq G\times G$ of \textit{disjoint} pairs of finite subsets, 
  \begin{center}
  \begin{tikzcd}[row sep = {30pt,between origins}, column sep = {40pt,between origins}]
    &(\Ran T\times \Ran T)_\circ\ar[ld,|->]\ar[rd,|->]& &[35pt]&(S_1,S_2)\ar[ld,|->]\ar[rd,|->]& \\ 
    \Ran T \times \Ran T&&\Ran T & (S_1,S_2)&& S_1\cup S_2
  \end{tikzcd}
  \end{center}
   and $G=\Ml_\Al$ the moduli stack of objects in an abelian category $\Al$, with $C=\Ml_{\textup{SES}(\Al)}$ so that the extra data corresponds to choosing a short exact sequence in $\Al$:
  \begin{center}
  \begin{tikzcd}[row sep = {30pt,between origins}, column sep = {45pt,between origins}]
   &\Ml_{\textup{SES}(\Al)}\ar[rd]\ar[ld] & &[30pt] & a_1\to e\to a_2\ar[ld,|->]\ar[rd,|->]& \\ 
  \Ml_\Al\times \Ml_\Al & & \Ml_\Al & (a_1,a_2)&& e
  \end{tikzcd}
  \end{center}
  The units correspond to the empty set $\varnothing$ and the zero object $0$ respectively, and the associativity to associativity property of disjointness or of iterated extensions.\footnote{To be precise, we have 
  $$S_1\cap S_2\ =\ \varnothing\ =\  (S_1\cup S_2)\cap S_3\hspace{5mm}\Leftrightarrow\hspace{5mm} S_2\cap S_3\ =\ \varnothing\ =\ S_1\cap (S_2\cup S_3)$$
  and an equivalence between the data of iterated extensions
  $$a_1\to e_{12}\to a_2, \ \ e_{12}\to e_{123}\to e_3 \hspace{5mm}\leftsquigarrow \hspace{-4.5mm} \rightsquigarrow \hspace{5mm} a_2\to e_{23}\to a_3, \ \ a_1\to e_{123}\to e_{23}.$$}

\subsection{Factorisation spaces and categories} \label{sec:FactSpaces}

Recall the $(\infty,1)$-category $\PreStk^{\textup{corr}}$ introduced in section \ref{sec:Correspondences}, whose objects are prestacks and whose morphisms are correspondences correspondence
\begin{center}
\begin{tikzcd}[row sep = 10pt, column sep = 15pt]
 &[-5pt]C\ar[rd]\ar[ld,swap] & \\
X & & Y
\end{tikzcd}
\end{center}
Taking the cartesian product of spaces endows it with a symmetric monoidal structure. 

\begin{defn}
  An (\textit{associative}) \textit{factorisation space} is an associative algebra object in $\PreStk^{\textup{corr}}$. 
\end{defn}

Likewise, we call the $(\infty,1)$-category $\Eb_n\Ag\left(\PreStk^{\textup{corr}}\right)$ the category of \textit{factorisation $\Eb_n$-spaces}. Thus, a factorisation algebra consists of a prestack $Y$ together with correspondences
\begin{center}
\begin{tikzcd}[row sep = 10pt, column sep = 15pt]
 &[-5pt]C\ar[rd,"p"]\ar[ld,swap,"q"] &&[20pt]&E\ar[rd,"e"]\ar[ld] & \\
Y\times Y & & Y & \pt& &Y 
\end{tikzcd}
\end{center}
which satisfies an associativity and unit condition, plus higher compatibilities. On $1$-morphisms, these give isomorphisms
$$(C\times Y)_{Y\times Y}C\ \simeq\ (Y\times C)\times_{Y\times Y}C, \hspace{10mm} (E\times Y)\times_{Y\times Y}C\ \simeq\ (Y\times E)\times_{Y\times Y}C\ \simeq\ Y$$
of correspondences.

\subsubsection{Factorisation categories} We now define a category $\Groth$ whose objects are pairs $(X,\Cl)$ of a prestack and a sheaf of categories over it, and $1$-morphisms $(X,\Cl)\to (Y,\Dl)$ consist of a correspondence 
\begin{center}
  \begin{tikzcd}[row sep = 10pt, column sep = 15pt]
   &[-5pt]C\ar[rd,"p"]\ar[ld,"q"'] & \\
  X & & Y
  \end{tikzcd}
  \end{center}
and functor $\varphi:q^*\Cl\to p^*\Dl$ of sheaves of categories on $C$. 

To define it, we consider the symmetric monoidal functor of $(\infty,1)$-categories
$$\PreStk^{\textup{corr}}\ \to\ \Cat \hspace{15mm}X\ \mapsto \ \ShvCat(X).$$
On the level of $1$-morphisms, given a correspondence $C$ as above we get 
$$p_*q^*\ :\ \ShvCat(Y)\ \to\ \ShvCat(X)$$
which is a functor because of base change for sheaves of categories, see \cite[Sec. A.9]{Ras}. This functor is symmetric monoidal, so $\ShvCat(X)$ is an $\Eb_n$-category whenever $X$ is a factorisation $\Eb_n$-space. We now apply the Grothendieck construction from section \ref{sec:Straightening} to give a symmetric monoidal cocartesian fibration 
\begin{equation} \label{eqn:GrothCoCartFib}
   \Groth\ \to\ \PreStk^{\textup{corr}} \hspace{15mm} (X,\Cl)\ \mapsto\ X.
\end{equation}
Then, 
\begin{defn}
  An (\textit{associative}) \textit{factorisation category} is an associative algebra object in $\Groth$.  
\end{defn}

Often we denote to a factorisation category $(X,\Cl)$ by $\Cl$ and say that it \textit{lives over} $X$. Likewise, a \textit{factorisation $\Eb_n$-category} is an element of $\Eb_n\Ag(\Groth)$. The space it lives over is automatically a factorisation $\Eb_n$-space because (\ref{eqn:GrothCoCartFib}) is a symmetric monoidal functor. Thus, a factorisation category over factorisation space $Y$ is a sheaf of categories $\Cl$ together with maps 
$$\otimes_\Cl\ :\ q^*(\Cl\boxtimes\Cl)\ \to\ p^*\Cl, \hspace{15mm} 1_\Cl\ :\ \triv \ \to\ e^*\Cl$$
of sheaves of categories on $C$ and $E$, along with associativity and compatiblity data.

\subsubsection{} If $M_1,M_2$ are sections of $\Cl$ we will write\footnote{When $Y=\pt$ this agrees with the usual notation for the tensor product of objects in a monoidal category.}
$$M_1\otimes_\Cl M_2\ =\ \otimes_\Cl \left(q^*(M_1\boxtimes M_2)\right).$$
Likewise, we will write 
$$M_1\otimes_\Cl\cdots\otimes_\Cl M_n$$
for the section of $p_n^*\Cl$ given by applying $\otimes_\Cl$ iteratively and where we have the correspondence given by pulling back by $C$ iteratively,
\begin{center}
\begin{tikzcd}[row sep = 10pt, column sep = 15pt]
 &C_n\ar[rd,"p_n"]\ar[ld,"q_n"']  & \\ 
Y^n & & Y
\end{tikzcd}
\end{center}
which by associativity independent of the order taken.

\subsubsection{Factorisation algebras} Recall that one may define associative algebras $A$ in an arbitrary monoidal category, and more generally $\Eb_n$-algebras in $\Eb_m$-categories for $m\ge n$. One does not need to make a special definition for this, one just considers associative algebra objects $(\Cl,A)$ in the category of pairs of a category and an object.

We apply this construction twice, first to 
$$\ShvCat(X)\ \to\ \Cat_\infty\hspace{15mm} \Cl\ \mapsto\ \Gamma(X,\Cl)$$
where $X$ is any prestack, giving the cocartesian fibration 
$$\ShvCat(X)_{\QCoh_X/}\ \to\ \ShvCat(X)$$
classifying pairs $(\Cl,A)$ of a sheaf of categories and section. Second, we apply it to the symmetric monoidal functor
$$\PreStk^{\textup{corr}}\ \to\ \textup{CommAlg}(\Cat_\infty) \hspace{15mm} X\ \mapsto \ \ShvCat(X)_{\QCoh_X/}$$
and apply the Grothendieck construction to give a symmetric monoidal cocartesian fibration
$$\Groth_{\QCoh/}\ \to\ \PreStk^{\textup{corr}}$$
Its objects are triples $(X,\Cl,A)$ of a prestack, sheaf of categories, and a section thereof. A $1$-morphism $(X,\Cl,A)\to (Y,\Dl,B)$ is a $1$-morphism $(X,\Cl)\to (Y,\Dl)$ as discussed above, and a map $\varphi(q^*A)\to p^*B$ inside $p^*\Dl$. 

\begin{defn}
   An (\textit{associative}) \textit{factorisation algebra} is an associative algebra object in $\Groth_{\QCoh/}$.
\end{defn}

Likewise, a \textit{factorisation $\Eb_n$-category} is an element of $\Eb_n\Ag(\Groth_{\QCoh/})$. Forgetting the section induces a symmetric monoidal functor $\Groth_{\QCoh/}\to\Groth$, so it follows that if $A$ is a factorisation $\Eb_n$-algebra, the sheaf of categories $\Cl$ it is a section of is also a factoriation $\Eb_n$-category. Thus, a factorisation algebra is a section $A\in \Gamma(Y,\Cl)$ of a factorisation category $\Cl$, equipped with structure maps
$$m_A\ :\ A\otimes_\Cl A\ \to \ p^*A,\hspace{15mm} 1_A\ :\ 1_\Cl \ \to\ e^*A$$
of sections of $p^*\Cl$ and $e^*\Cl$, satisfying the associativity condition and with a unit.

The most trivial example of a factorisation space is $X=\pt$ with the trivial correspondence $C=\pt$, factorisation category over it $\Cl=\QCoh_{\pt}\simeq\Vect$ and factorisation algebra $A=\Ol_{\pt}\simeq k$. Slightly more generally, if $X$ is any factorisation $\Eb_n$-space, then $\Cl=\QCoh_{X}$ and $A=\Ol_X$ are automatically factorisation $\Eb_n$-categories and algebras, because they are the units in $\ShvCat(X)$ and in $\Cl(X)$.

\subsubsection{Spaces of sections} For any factorisation category $\Cl$ over $X$, we have by adjunction an equivalence of dg-categories
$$\Gamma(X,\Cl)\ =\ \Hom_{\ShvCat(X)}(\QCoh_X,\Cl)$$ 
where by definition the left side is the value of $\Cl$ on $X$, and on the right we use that $\QCoh_X$ is the unit for the symmetric monoidal structure on $\ShvCat(X)$. 

One category level down, if $(\Cl,1_\Cl)$ is a \textit{pointed} sheaf of categories over $X$, i.e. a factorisation $\Eb_0$-category, then for any $A\in \Gamma(X,\Cl)$ we define the dg-vector space of its \textit{sections} by 
$$\Gamma(X,A)\ \defeq\ \Hom_{\Gamma(E,e^*\Cl)}(1_\Cl(\Ol),e^*A).$$
For instance, if the point of $X$ is given by the correspondence $E=X$ and $e$ the identity map, then this is $\Hom_{\Gamma(X,\Cl)}(1_\Cl(\Ol),A)$. 

In the trivial example, for any $A\in \Gamma(X,\QCoh_X)=\QCoh(X)$ we have that $\Gamma(X,A)$ is the usual vector space of sections as a quasicoherent sheaf.

\subsubsection{Fibrewise description} Let $Y$ be a factorisation space and pick a point $c:\pt\to C$ of the correspondence defining the product. Denote its images by
\begin{equation}\label{fig:ImageOfc}
  \begin{tikzcd}[row sep = 10pt, column sep = 15pt]
  &[-5pt]c\ar[rd,"p",|->]\ar[ld,swap,"q",|->] &  \\
  (y_1,y_2) & & y 
  \end{tikzcd}
  \end{equation}
Then if $\Cl$ is a factorisation category, we get a collection of dg-functors 
$$\otimes_{\Cl,c}\ :\ \Cl_{y_1}\otimes \Cl_{y_2}\ \to\ \Cl_{y}.$$
Similarly if $A$ is a factorisation algebra, its product gives a a collection of maps in the category $\Cl_y$
$$m_{A,c} \ :\ A_{y_1}\otimes_{\Cl,c}A_{y_2}\ \to\ A_y.$$
Here we denote $\Cl_y$ and $A_y$ the pullbacks as sheaf of categories and sections along the map $y:\pt\to C$. There is a similar fibrewise description for $\Cl$-module categories and $A$-modules. 

Similarly, we can pick a point $d:\pt \to E$ of the unit, with images 
\begin{equation}\label{fig:ImageOfe}
  \begin{tikzcd}[row sep = 10pt, column sep = 15pt]
  &[-5pt]d\ar[rd,"e",|->]\ar[ld,swap,|->] &  \\
  \pt & & y
  \end{tikzcd}
  \end{equation}
Then if $\Cl$ is a factorisation category and $A$ a factorisation algebra inside it, we get structure maps 
$$1_{\Cl,d}\ :\ \Vect\ \to \ \Cl_d, \hspace{15mm} 1_{A,d}\ :\ 1_{\Cl,d}(k)\ \to\ A_d.$$

We write down associativity explicitly. If $\Cl$ is a factorisation category, then attached to any point 
$$c_{1,2,3}\ \in\ (Y\times C)\times_{Y\times Y}C\ \simeq\ (C\times Y)\times_{Y\times Y}C$$
with images 
\begin{center}
  \begin{tikzcd}[row sep = {30pt,between origins}, column sep = {45pt,between origins}]
    &&c_{1,2,3}\ar[rd,|->]\ar[ld,|->]&&\\
    &[-5pt](c_{12},y_3)\ar[rd,"p\times\id",|->]\ar[ld,swap,"q\times\id",|->] &&[-5pt]c_{12,3}\ar[rd,"p",|->]\ar[ld,swap,"q",|->] &  \\
    (y_1,y_2,y_3) & & (y_{12},y_3) & & y 
    \end{tikzcd} 
\end{center}
and likewise for the other bracketing order, we have for any sections $A_{y_i}\in \Cl_{y_i}$ an equivalence in $\Cl_y$:
$$(A_{y_1}\otimes_{\Cl,c_{12}}A_{y_2})\otimes_{\Cl,c_{12,3}}A_{y_3}\ \simeq\ A_{y_1}\otimes_{\Cl,c_{1,23}}(A_{y_2}\otimes_{\Cl,c_{2,3}}A_{y_3}).$$
Moreover, if $A$ is a factorisation algebra, then under this isomorphism, for any three sections $a_i\in \Gamma(Y,A_{y_i})$ we have an identification 
$$(a_1\cdot_{A,c_{1,2}}a_2)\cdot_{A,c_{12,3}}a_3\ =\ a\cdot_{A,c_{1,23}}(a_2\cdot_{A,c_{2,3}}a_3)$$
in the vector space $\Gamma(Y,A_y)$. One can write similar descriptions for the unit map.

\subsection{Modules} If $A$ is an associative algebra, there are two types of module categories one can consider: bimodules, and left (or right) modules. The former comes with a natural monoidal structure, and the latter does not. 

More generally, if $A$ be a an $\Eb_n$-algebra inside a symmetric monoidal $(\infty,1)$-category $\Cl$, we may consider the categories
$$ A\Md_{\Eb_n}(\Cl), \hspace{15mm} A\Md\ \defeq\ A\textup{-LMod}_{\Eb_n}(\Cl)$$
of its $\Eb_n$-modules and left $\Eb_n$-modules. These have the structure of an $\Eb_n$-category and an $\Eb_{n-1}$-category, respectively. When $\Cl$ is an $(\infty,1)$-category this is proven in \cite[3.4.4.2]{Lu2} and \cite[4.8.5.20]{Lu2} respectively. We will use the convention that unless otherwise stated an ``$A$-module'' means a left module.

\subsubsection{Modules over factorisation spaces} A module over a factorisation $\Eb_n$-space $Y$ is called a \textit{$Y$-module space}. For instance, if $Y$ is an associative factorisation space this consists of a prestack $M$ along with 
\begin{center}
  \begin{tikzcd}[row sep = 10pt, column sep = 15pt]
   &[-5pt]C_X\ar[rd,"p_X"]\ar[ld,swap,"q_X"] & \\
  Y\times X & & Y 
  \end{tikzcd}
  \end{center}
which is compatible with the unit and the compatiblity conditions of $Y$.

\subsubsection{Modules over factorisation categories}  A module over a factorisation $\Eb_n$-category $\Cl$ is called a \textit{$\Cl$-module category}. If $\Cl$ is an associative factorisation category over $X$ in the above example this consists of a map 
$$\otimes_\Ml\ :\ q_M^*(\Cl\boxtimes\Ml)\ \to\ p_M^*\Ml$$
of sheaves of categories on $C_X$ compatible with the factorisation structure $\otimes_\Cl$ and unit $1_\Cl$.

\subsubsection{Modules over factorisation algebras} A module over a factorisation $\Eb_n$-algebra $A$ is called a \textit{$\Cl$-module category}. If $A$ is an associative factorisation algebra in $\Ml$ above, this consists of a section $M\in \Gamma(Y,\Ml)$ together with a map 
$$m_M\ :\ A\otimes_\Ml M\ \to \ p^*M$$
compatible with the multiplication $m_A$ and unit $1_A$. 

\begin{lem}
  The category $A\FactMd(X,\Ml)$ of $A$-modules in $\Ml$ is a factorisation $\Eb_{n-1}$-category.
\end{lem} 
\begin{proof}
  This follows by a relative version of  \cite[Cor. 4.8.5.20]{Lu2} applied to the symmetric monoidal cocartesian fibration $\Groth_{\QCoh/}\to \Groth$. 
\end{proof}

Likewise, the category of (both-sided) $A$-modules is a factorisation $\Eb_n$-category.

\subsection{Functoriality and variants} 
\label{sec:FunctorialityFactCat} 

\subsubsection{Factorisation categories} We have a functor 
$$\Eb_n\Ag(\Groth)\ \to\ \Eb_n\Ag(\PreStk^{\textup{corr}})$$
Since Grothendieck constructions commute with taking $\Eb_n$-algebras, it follows that this is a cocartesian fibration. In particular, for any map of factorisation spaces $X\to Y$ represented by a correspondence 
\begin{center}
  \begin{tikzcd}[row sep = 10pt, column sep = 15pt]
   &[-5pt]C\ar[rd,"p"]\ar[ld,"q"'] & \\
  X & & Y
  \end{tikzcd}
  \end{center}
with compatibilities with the factorisation $\Eb_n$-structures on $X$ and $Y$, 
\begin{lem}
  We have a functor 
  $$p_*q^*\ :\ \Eb_n\FactCt(X)\ \to\ \Eb_n\FactCt(Y).$$
  This acts on the underlying sheaf of categories by the functor sending $\Cl\mapsto p_*q^*\Cl$.
\end{lem}

For instance, let $X$ and $Y$ be associative factorisation spaces and the correspondence between them be given by a genuine map of prestacks
\begin{center}
  \begin{tikzcd}[row sep = 10pt, column sep = 15pt]
   &[-5pt]X\ar[rd,"f"]\ar[ld,equals] & \\
  X & & Y
  \end{tikzcd}
  \end{center}
Whenever this is a map of factorisation spaces the above says that $f_*$ takes factorisation categories to factorisation categories. Thus, we spell out this condition. To give a two-isomorphism in $\PreStk^{\textup{corr}}$
\begin{center}
  \begin{tikzcd}[row sep = {30pt,between origins}, column sep = {20pt}]
   X\times X\ar[r,"C_X"]\ar[d,"f\times f",swap,""{name=U,inner sep=1pt}]& X\ar[d,"f",""{name=D,inner sep=1pt}] \\
  Y\times Y \ar[r,"C_Y"] & Y  
  \arrow[Rightarrow, from=U, to=D, "\sim", shorten <= 18pt, shorten >= 18pt] 
   \end{tikzcd}
  \end{center}
is equivalent to giving an isomorphism of prestacks
$$\varphi_L\ :\ C_X\ \simeq\ C_Y\times_{Y\times Y}(X\times X)$$
inducing commuting diagram
\begin{equation}\label{fig:RespectsFactStructures}
\begin{tikzcd}[row sep = {30pt,between origins}, column sep = {20pt}]
 &[-5pt] C_X\ar[d]\ar[rd]\ar[ld,swap]&[5pt] \\
 X\times X\ar[d,swap,"f\times f"]&C_Y\ar[rd]\ar[ld] & X\ar[d,"f"]\\
 Y\times Y&& Y 
\end{tikzcd}
\end{equation}
where the left hand square is a pullback square. 

Likewise, consider $X$ and $Y$ factorisation spaces and the map between them being given by a function 
\begin{center}
  \begin{tikzcd}[row sep = 10pt, column sep = 15pt]
   &[-5pt]Y\ar[ld,"g"']\ar[rd,equals] & \\
  X & & Y
  \end{tikzcd}
  \end{center}
so that if this is a map of factorisation spaces then $g^*$ takes factorisation categories to factorisation categories. Denoting the above correspondence by $g$, to supply a two-isomorphism 
\begin{center}
  \begin{tikzcd}[row sep = {30pt,between origins}, column sep = {20pt}]
   X\times X\ar[r,"C_X"]\ar[d,"g\times g",swap,""{name=U,inner sep=1pt}]& X\ar[d,"g",""{name=D,inner sep=1pt}] \\
  Y\times Y \ar[r,"C_Y"] & Y  
  \arrow[Rightarrow, from=U, to=D, "\sim", shorten <= 18pt, shorten >= 18pt] 
   \end{tikzcd}
  \end{center}
it is equivalent to giving an isomorphism of prestacks 
$$\varphi_R\ :\  C_Y\ \simeq\ C_X\times_X Y$$
inducing the commuting diagram
\begin{equation}\label{fig:RespectsFactStructures2}
  \begin{tikzcd}[row sep = {30pt,between origins}, column sep = {20pt}]
   &[-5pt] C_Y\ar[d]\ar[rd]\ar[ld,swap]&[5pt] \\
   Y\times Y\ar[d,swap,"g\times g"]&C_X\ar[rd]\ar[ld] & Y\ar[d,"g"]\\
   X\times X&& X 
  \end{tikzcd}
  \end{equation}
whre the right hand square is a pullback square. 

\subsubsection{} Note that any $X$ is a prestack with respect to the diagonal factorisation structure with $C_X\simeq X$, see section \ref{sec:DiagonalFactStruct}. Thus when $Y=\pt$ the correspondence 
\begin{center}
  \begin{tikzcd}[row sep = 10pt, column sep = 15pt]
   &[-5pt]X\ar[rd,equals]\ar[ld,"\pi_X"'] & \\
  \pt & & X
  \end{tikzcd}
  \end{center}
is a map of commutative factorisation spaces, and $\pi_X^*\Cl$ is a factorisation $\Eb_n$-category on $X$ for any ordinary $\Eb_n$-category $\Cl$. However, the correspondence 
\begin{center}
  \begin{tikzcd}[row sep = 10pt, column sep = 15pt]
   &[-5pt]X\ar[ld,equals]\ar[rd,"\pi_X"] & \\
  X & & \pt
  \end{tikzcd}
  \end{center}
is not a map of factorisation spaces because $\varphi_L$ is not an isomorphism, and so for arbitrary factorisation $\Eb_n$-category $\Cl$ on $X$, its space $\Gamma(X,\Cl)=\pi_{X,*}\Cl$ does not have the structure of an ordinary $\Eb_n$-category.

However, if $G$ is a group or monoid then $C_G\simeq G\times G$ and 
\begin{center}
  \begin{tikzcd}[row sep = 10pt, column sep = 15pt]
   &[-5pt]G\ar[ld,equals]\ar[rd,"\pi_X"] & \\
  G & & \pt
  \end{tikzcd}
  \end{center}
is a map of factorisation spaces, so $\Gamma(G,\Cl)$ has the structure of a monoidal category, or $\Eb_n$-category if we consider factorisation $\Eb_n$-categories on $\Eb_n$-groups.

\subsubsection{}  \label{sec:FactBiAlg} We can likewise define cofactorisation $\Eb_n$-spaces, bifactorisation $\Eb_{n,m}$-spaces, etc. by taking $\Eb_n$-coalgebra objects, $\Eb_{n,m}$-bialgebra objects, etc. in the appropriate category.  We may define factorisation $\Eb_{n_1}$-algebras in factorisation $\Eb_{n_2}$-categories over factorisation $\Eb_{n_3}$-spaces, for $n_1\le n_2\le n_3$.

\newpage

\section{Spectral $R$-Matrices} 
\label{sec:SpectralRMatrices}

\noindent
In this section we develop a factorisable analogue of the theory of quantum groups.

\subsection{Quantum groups and $R$-matrices}\label{sec:BraidedMonoidalCats} \label{sec:QuantumGroups} Recall that a \textit{monoidal} structure on one-category $\Cl$ is a functor
$$\otimes_{\Cl}\ :\ \Cl\otimes\Cl\ \to\ \Cl$$
with a unit $1\in \Cl$ and an associativity natural isomorphism $\otimes_{\Cl}\cdot (\id\otimes\, \otimes_{\Cl})\simeq \otimes_{\Cl}\cdot (\otimes_{\Cl}\otimes\id)$ satisfying the pentagon identity, and a \textit{braided monoidal} structure is an additional natural isomorphism
\begin{center}
\begin{tikzcd}[row sep = 20pt, column sep = 30pt]
\Cl\otimes\Cl\ar[r,"\otimes_{\Cl}"{xshift=-10pt},bend left = 35, ""{name=U,inner sep=0pt,below,xshift=0pt}]\ar[r,"\otimes_{\Cl}\cdot\sigma_{\Cat}"{xshift=-12pt},bend right = 35,swap, ""{name=D,inner sep=1pt}]&\Cl
  \arrow[Rightarrow, from=U, to=D, "\beta", shorten <= 5pt, shorten >= 5pt] 
  \arrow[Rightarrow, from=U, to=D, swap,"\wr", shorten <= 5pt, shorten >= 5pt]
\end{tikzcd}
\end{center}
satisfying the hexagon identities,\footnote{The hexagon relations state 
$$\beta_{A,B\otimes_\Cl C}\ =\ \beta_{A,C}\cdot \beta_{A,B},\hspace{15mm}\beta_{A,B\otimes_\Cl C}\ =\ \beta_{A,C}\cdot \beta_{B,C}$$
for all objects $A,B,C$ of $\Cl$.} where the endofunctor $\sigma_{\Cat}$ swaps the two factors. It is \textit{symmetric monoidal} if $\beta\cdot (\beta\sigma_{\Cat})=\id$.  

In practice, most examples one is interested in are of the form $A\Md$ for $A$ an associative algebra in a background symmetric monoidal category $(\El,\sigma_\El)$. In this section we will recall the well-known necessary and sufficient conditions for $A\Md$ to be braided monoidal. First, the module category must be monoidal:
\begin{lem} \label{lem:MonoidalLift}
  A lift of the monoidal structure along $A\Md\to \El$ is equivalent to a bialgebra structure on $A$.
\end{lem}
\begin{proof}
   If $A$ is a bialgebra, then $M\otimes_\El N$ inherits an $A$-module structure for any $M,N\in A\Md$, where $a\in A$ acts as multiplication by $\Delta(a)$. Conversely, given a lift $\otimes$ of the monoidal structure we get an $A$-module structure on $A\otimes A$, and acting on the identity  $\Delta : A \stackrel{\id\otimes 1\otimes 1}{\to} A\otimes_\El (A\otimes A) \to (A\otimes A)$ gives a coproduct.
\end{proof}

Then, it must have a braiding:
\begin{prop} \emph{\cite[Prop. 14.2]{ES}} \label{prop:RMatrixBraidedMonoidal}
  Let $A$ be a bialgebra in $\El$. Then the monoidal structure on $A\Md$ has a braiding if and only if there is an invertible element $R\in A\otimes A$ satisfying
  \begin{equation} \label{eqn:RMatrixHexagon}
    (\Delta\otimes\id)R\ =\ R_{13}R_{23}, \hspace{10mm} (\id\otimes\Delta)R\ =\ R_{13}R_{12},
   \end{equation}
    exhibiting $A$'s \emph{almost cocommutativity}\footnote{Note that this is equivalent to it being a braided-commutative ($\Eb_2$) coalgebra in the category of $A$-bimodules, endowed with the braiding $R^{-1}\sigma(-)R$.}
   \begin{equation}\label{eqn:AlmostCommutativity}
    \sigma_\El\cdot \Delta(-)\ =\ R\Delta(-)R^{-1}
   \end{equation}
   and compatibly with the counit, $(\epsilon\otimes\id)(R)=(\id\otimes\epsilon)(R)=1$. Likewise for lax braidings if we drop the invertibility condition, and symmetric braidings if we require that $R=\sigma(R)^{-1}$.
\end{prop}
\begin{proof}
   Given such an element, $\beta = \sigma_\El\cdot (R\cdot (-))$ defines a braided monoidal structure on $A\Md$. Indeed, for every $A$-modules $M,N$ it defines a map $M\otimes N \stackrel{\sim}{\to} N\otimes M$ in $\El$, which is $A$-linear because
   \begin{equation}
     \label{fig:ALinearityBraiding} 
     \begin{tikzcd}[row sep = 20pt, column sep = 20pt]
       M\otimes N\ar[r,"R\cdot "]\ar[d,"\Delta(a)\cdot"] & M\otimes N\ar[r,"\sigma_\El "]\ar[d,"\sigma_\El(\Delta(a))\cdot"]  & N\otimes M\ar[d,"\Delta(a)\cdot"] \\ 
       M\otimes N\ar[r,"R\cdot "] & M\otimes N\ar[r,"\sigma_\El"] & N\otimes M
       \end{tikzcd}
   \end{equation}
     commutes by (\ref{eqn:AlmostCommutativity}). It satisfies the hexagon relations by (\ref{eqn:RMatrixHexagon}). Conversely, given a braiding $\beta$ of the monoidal structure on $A\Md$, the element $R=\beta_{A,A}(1_A\otimes 1_A)$ defines an $R$-matrix. 
\end{proof}

As a consequence, it is easy to show that $R$ satisfies the \textit{Yang-Baxter equation}:
\begin{center}
\begin{tikzpicture}[scale=0.7]
\begin{knot}[
  clip width=5,
  flip crossing=1,
  flip crossing=2,
  flip crossing=3,
  flip crossing=4,
  flip crossing=5,
  flip crossing=6,
]
\strand[black,ultra thick] (0,0) .. controls +(0,1.5) and +(0,-1.5) .. (2,3);
\strand[black,ultra thick] (1,0) .. controls +(0,1) and +(0,-1) .. (0,1.5);
\strand[black,ultra thick] (0,1.5) .. controls +(0,1) and +(0,-1) .. (1,3);
\strand[black,ultra thick] (2,0) .. controls +(0,1) and +(0,-1) .. (0,3);
\node[] at (3.5,1.5) {$=$};
\node[] at (14,1.5) {$R_{12}R_{13}R_{23}\ =\ R_{23}R_{13}R_{12}.$};

\strand[black,ultra thick] (5,0) .. controls +(0,1.5) and +(0,-1.5) .. (7,3);
\strand[black,ultra thick] (6,0) .. controls +(0,1) and +(0,-1) .. (5+2,1.5);
\strand[black,ultra thick] (5+2,1.5) .. controls +(0,1) and +(0,-1) .. (6,3);
\strand[black,ultra thick] (7,0) .. controls +(0,1) and +(0,-1) .. (5,3);

\end{knot}
\end{tikzpicture}
\end{center}
In Proposition \ref{prop:RMatrixBraidedMonoidal}, we get a symmetric monoidal structure if in addition $R$ is \textit{symmetric}:
\begin{center}
  \begin{tikzpicture}[scale=0.7]
  \begin{knot}[
    clip width=5,
    flip crossing=1,
  ]
  \strand[black,ultra thick] (1,0) .. controls +(0,1) and +(0,-1) .. (0,1.5);
  \strand[black,ultra thick] (0,1.5) .. controls +(0,1) and +(0,-1) .. (1,3);
  \strand[black,ultra thick] (0,0) .. controls +(0,1) and +(0,-1) .. (1,1.5);
  \strand[black,ultra thick] (1,1.5) .. controls +(0,1) and +(0,-1) .. (0,3);
  \node[] at (2.5,1.5) {$=$};
  \node[] at (13,1.5) {$R_{21}R_{12}\ =\ 1\otimes 1.$};
  
  \strand[black,ultra thick] (4,0) .. controls +(0,1.5) and +(0,-1.5) .. (4,3);
  \strand[black,ultra thick] (5,0) .. controls +(0,1.5) and +(0,-1.5) .. (5,3);
  \end{knot}
  \end{tikzpicture}
  \end{center}
  Let $\El$ be a braided monoidal category; one can show that the above Proposition is still true. Then 
  \begin{defn} \label{defn:RMatrix}
    A bialgebra $A\in \El$ with invertible element $R\in A\otimes A$ satisfying the conditions of Proposition \ref{prop:RMatrixBraidedMonoidal} is called a \emph{quantum group} or \emph{quasitriangular bialgebra}. $R$ is called an \textit{$R$-matrix}.
  \end{defn}

Notice that Proposition \ref{prop:RMatrixBraidedMonoidal} does \textit{not} say that there is an equivalence between braided monoidal structures and $R$-matrices, there could potentially be more than one braiding giving the same $R$-matrix, although the Proposition identifies a canonical one. Equivalence statements do exist in some contexts: if $\Cl$ is a braided monoidal category equipped with a monoidal functor $\Cl\to \El$ into another braided monoidal category, in some cases we may apply the Barr-Beck monadicity Theorem to identify it with the category of modules for $A=\End_{\Eb_1 \textup{-Cat}}(\Cl\to \El)$. For instance, this is true under mild conditions when $\El=\Vect$, see \cite[9.4.2]{Maj}.

\subsubsection{Bicharacters}  Instead of working directly with an $R$-matrix, some sources work instead with its \textit{bicharacter}:
$$r\ :\ A\otimes A\ \stackrel{R\cdot}{\to}\ A\otimes A\ \stackrel{\epsilon\otimes\epsilon}{\to}\ k.$$
The $R$-matrix is recovered by taking cap product, $\cap r=R\cdot$, so no information is lost by working with bicharacters. See (e.g. \cite{An,Bo2}), where it is assumed $A$ is commutative and cocommutative, so the almost commutativity axiom becomes vacuous.

\subsubsection{Lax versions} We are going to also need a lax version of the above. We define a \textit{lax braided monoidal} category $\El$ to be the data of two monoidal structures $\otimes,\otimes'$ endowed with a natural transformation
$$\sigma_\El\ :\ (\, \otimes'\, )\otimes(\, \otimes'\, )\ \stackrel{}{\to}\ (\, \otimes\, )\otimes'(\, \otimes\, )\cdot \sigma_{\Cat,23}$$
making $\otimes'$ into a lax $\otimes$-monoidal functor, i.e. are compatible with the unit and we have the following commuting diagram, after supressing associativity maps from the notation,
\begin{equation}\label{fig:LaxHexagonRelations}
  \begin{tikzcd}[row sep = {25pt,between origins}, column sep = {125pt, between origins}]
    &\left((a_1\otimes b_1)\otimes' (a_2\otimes b_2)\right)\otimes (c_1\otimes' c_2) \ar[rd]& \\ 
    (a_1\otimes' a_2)\otimes (b_1\otimes' b_2)\otimes (c_1\otimes' c_2) \ar[ru]\ar[rd]& &(a_1\otimes b_1\otimes c_1)\otimes' (a_2\otimes b_2 \otimes c_2)  \\
   &(a_1\otimes' a_2)\otimes \left( (b_1\otimes' c_1)\otimes (b_2\otimes' c_2)\right)\ar[ru]&
   \end{tikzcd}
\end{equation}
Setting $a_1=b_1=c_2=1$ or $a_1=b_2=c_2=1$ gives the hexagon relations, thus we call (\ref{fig:LaxHexagonRelations}) the (\textit{lax}) \textit{hexagon relation}. Moreover, if we take $\sigma_{\El}$ and let the inner two factors be the unit, we get the natural transformation
$$\sigma_{\El,23}\ :\ \otimes\ \to\ \otimes'$$
which does not determine $\sigma_{\El}$ in general. This notion is called a \textit{duoidal category} in the literature, see for instance \cite{To}.

We again consider examples of lax braided monoidal categories the form $A\Md$ for an $A$ an associative $\otimes$-algebra in $\El$, whose product we denote by $m:A\otimes A\to A$. As before, it first must be monoidal:  

\begin{lem}
  A lift of the monoidal structure $\otimes'$ along $A\Md\to \El$ is equivalent to a bialgebra structure on $A$,  with coproduct denoted $\Delta' : A\ \to\ A\otimes'A$.
\end{lem}

Next, assuming that the monoidal structure $\otimes$ on $\El$ is strongly braided, compatiblity with $\sigma_\El$: 
$$\beta_\otimes\ :\ \otimes \ \simeq\ \otimes\cdot \sigma_{\Cat},$$
we have the anaologue of Proposition \ref{prop:RMatrixBraidedMonoidal}:
\begin{prop} \label{prop:LaxRMatrixBraidedMonoidal}
  Let the algebra $A$ have two bialgebra structures, with coproducts 
  $$\Delta\ :\ A\ \to\ A\otimes A, \hspace{15mm} \Delta'\ :\ A\ \to\ A\otimes' A.$$
  Then there is a lax braiding of the monoidal structures $\otimes$ and $\otimes'$ on $A\Md$ if and only if there is an element  
  $$\tilde{R}\ \in\ (A\otimes A)\otimes'(A\otimes A)$$
  satisfying the lax hexagon relation\footnote{In full, it is 
  $$((\Delta\otimes \id)\otimes'(\Delta\otimes \id))(\tilde{R})\cdot \sigma_{A^{\otimes 2},A^{\otimes 2},A,A}(\tilde{R}\otimes (1\otimes'1))\ =\ ((\id\otimes\Delta)\otimes'(\id\otimes\Delta)(\tilde{R}))\cdot \sigma_{A,A,A^{\otimes 2},A^{\otimes 2}}((1\otimes'1)\otimes \tilde{R}).$$} and almost cocommutativity
  \begin{equation}\label{eqn:LaxRMatrixHexagon}
   ((\Delta\otimes \id)\otimes'(\Delta\otimes \id))(\tilde{R})\cdot \tilde{R}_{12,45}\ =\ ((\id\otimes\Delta)\otimes'(\id\otimes\Delta)(\tilde{R}))\cdot \tilde{R}_{23,56},
    \end{equation}
  \begin{equation}\label{eqn:LaxRMatrixAlmostCocommutativity}
   \tilde{R}\cdot \left( \sigma_{\El}\cdot (\Delta'\otimes\Delta')\Delta(-) \right)\ =\ \left(  (\Delta\otimes\Delta)\Delta'(-) \right)\cdot \tilde{R},
  \end{equation}
  as elements of $A^{\otimes 3}\otimes' A^{\otimes 3}$ and as maps $A\to A^{\otimes 2}\otimes'A^{\otimes 2}$, compatibly with the counit
  \begin{equation}
   ((\epsilon\otimes \id)\otimes'(\epsilon\otimes\id))(\tilde{R})\ =\  ((\id\otimes\epsilon)\otimes'(\id\otimes\epsilon ))(\tilde{R})\ =\ 1\otimes'1.
  \end{equation} 
\end{prop}
\begin{proof}
  The lax braided monoidal structure is $\beta=\tilde{R}\cdot \sigma_\El$. Here $R$ is acting by left multiplication according to the $\otimes$-algebra structure on $(A\otimes A)\otimes' (A\otimes A)$. The proof proceeds exactly as in Proposition \ref{prop:RMatrixBraidedMonoidal}; almost cocommutativity ensures that this is $A$-linear\footnote{Explicitly, we have
  \begin{equation}
    \begin{tikzcd}[ampersand replacement=\&,row sep = 20pt, column sep = 20pt]
      (M_1\otimes' M_2)\otimes(M_3\otimes' M_4)\ar[r,"\sigma_\El"]\ar[d,"(\Delta'\otimes\Delta')\Delta(a)\cdot"] \& (M_1\otimes M_3)\otimes' (M_2\otimes M_4)\ar[d," \sigma_\El (\Delta'\otimes\Delta')\Delta(a)\cdot "] \ar[r,"R"] \& (M_1\otimes M_3)\otimes' (M_2\otimes M_4)\ar[d," (\Delta\otimes\Delta)\Delta'(a)\cdot "] \\ 
      (M_1\otimes' M_2)\otimes(M_3\otimes' M_4)\ar[r,"\sigma_\El"] \& (M_1\otimes M_3)\otimes' (M_2\otimes M_4)\ar[r,"R"]\& (M_1\otimes M_3)\otimes' (M_2\otimes M_4)
      \end{tikzcd}
  \end{equation}} 
  the hexagon and unit relations for $\tilde{R}$ are equivalent to the hexagon and unit relations for $\beta$.  
\end{proof}

In the context of Proposition \ref{prop:LaxRMatrixBraidedMonoidal}, we call $A$ in the lax braided monoidal category $\El$  a \textit{lax quantum group}, or \textit{lax quasitriangular bialgebra}, and $R$ as a \textit{lax $R$-matrix}. Note that applying the counit $\epsilon$ in the second or fifth factors of (\ref{eqn:LaxRMatrixHexagonSimple}) gives the two ordinary hexagon relations (\ref{eqn:RMatrixHexagon}). Applying $\epsilon$ to the first and fourth factors of (\ref{eqn:LaxRMatrixAlmostCocommutativitySimple}) recovers ordinary almost cocommutativity (\ref{eqn:LaxRMatrixAlmostCocommutativity}).

In particular, given any element $R\in A\otimes' A$ we can take its image $R_{23}$ under
$$A\otimes' A\ \stackrel{(1\otimes\id)\otimes'(\id\otimes 1)}{\to}\ (A\otimes A)\otimes'(A\otimes A).$$ 
We then have an example of the above Proposition:

\begin{cor}\label{cor:LaxRMatrixBraidedMonoidal}
  In the context of Proposition \ref{prop:LaxRMatrixBraidedMonoidal}, an element $R\in A\otimes' A$ induces a lax braided monoidal structure on $A\Md$ if 
  \begin{equation}\label{eqn:LaxRMatrixHexagonSimple}
    ((\id\otimes' \Delta)(R))_{345} \cdot R_{24}\ =\ ((\Delta\otimes'\id)(R))_{234}\cdot R_{35},
  \end{equation}
  \begin{equation}\label{eqn:LaxRMatrixAlmostCocommutativitySimple}
    R_{23}\cdot \left( \sigma_{\El}\cdot (\Delta'\otimes\Delta')\Delta(-) \right)\ =\ \left(  (\Delta\otimes\Delta)\Delta'(-) \right)\cdot R_{23},
  \end{equation}
  and we have compatiblity with the counit, $(\id\otimes'\epsilon)(R)=(\epsilon\otimes' \id)(R)= 1$.
\end{cor}

\subsubsection{Comodule versions} If $A$ is a \textit{co}algebra in (lax) braided monoidal category $\El$, with coproduct denoted $\Delta':A\to A\otimes' A$, we can dualise the previous sections. As before we assume that in addition $\otimes$ has a strong braiding compatibly with the lax braiding between $\otimes$ and $\otimes'$. 

\begin{prop}
   A lift of the monoidal structure along $A\CoMd\to \El$ is equivalent to a bialgebra structure on $A$. Let $A$ have two bialgebra structures with products
   $$m \ :\ A\otimes A\ \to\ A, \hspace{15mm} m' \ :\ A\otimes' A\ \to\ A.$$
  Then there is a lax braiding of the monoidal structures $\otimes$ and $\otimes'$ on $A\CoMd$ if and only if there is a map 
  $$\tilde{R}^*\ :\ (A\otimes A)\otimes'(A\otimes A)\ \to\ 1_\El$$
  satisfying the (duals of) the relations (\ref{eqn:LaxRMatrixHexagon}) and (\ref{eqn:LaxRMatrixAlmostCocommutativity}), and compatibly with the unit. 
\end{prop}

In the above case, we also call $A$ a \textit{lax op-quantum group}, and $R^*$ a \textit{lax op-$R$-matrix}. As in Corollary \ref{cor:LaxRMatrixBraidedMonoidal}, we have a simpler version of the above:

\begin{cor}
  For any map 
  $$R^*\ :\ A\otimes'A\ \to\ 1_\El$$
  satisfying the dual conditions to (\ref{eqn:LaxRMatrixHexagonSimple}) and (\ref{eqn:LaxRMatrixAlmostCocommutativitySimple}), and compatibility with the unit, the composition $\tilde{R}^*=(\epsilon\otimes \id)\otimes'(\id\otimes\epsilon)(R^*)$ defines an op-$R$-matrix. 
\end{cor}

\subsection{Classical factorisation spaces}  \label{sec:ClassicalFactorisationSpaces} To begin,

\begin{defn}
  A \textit{classical factorisation space} is an associative algebra object in $\PreStk^{\textup{cl},\textup{corr}}$. 
\end{defn}

We recalled the category $\PreStk^{\textup{cl}}$ of classical prestacks in section \ref{sec:PreStk}. Likewise, a \textit{classical factorisation $\Eb_n$-space} is an $\Eb_n$-algebra in the $(2,1)$-category $\PreStk^{\textup{cl},\textup{corr}}$, and we use the name (\textit{braided}) \textit{commutative classical factorisation space} when $n=\infty$ (resp. $n=2$). In particular, 
\begin{lemdefn} \label{lem:ClassicalAssocFactSpace}
   If $Y$ is a classical prestack, the structure of a \textit{classical factorisation space} on $Y$ is equivalent to a compatible product and unit correspondence 
\begin{center}
  \begin{tikzcd}[row sep = 10pt, column sep = 15pt]
   &[-5pt]C\ar[rd,"p"]\ar[ld,swap,"q"] & &[25pt] &E\ar[rd]\ar[ld]& \\
  Y\times Y & & Y & \pt && Y
  \end{tikzcd}
  \end{center}
and associativity isomorphism of correspondences
\begin{center}
  \begin{tikzcd}[row sep = {30pt,between origins}, column sep = {45pt,between origins}]
    & &(Y\times C)\times_{Y\times Y}C\ar[rd]\ar[ld] & & &[-15pt] &[-10pt] & &(C\times Y)\times_{Y\times Y}C\ar[rd]\ar[ld] & & \\
    &Y\times C \ar[rd]\ar[ld,"q\times q"']& &C\ar[rd]\ar[ld] & &\stackrel{\alpha}{\stackrel{\sim}{\to}} & &C\times Y \ar[rd]\ar[ld]& &C\ar[rd]\ar[ld] &  \\
    Y\times Y\times Y& &Y\times Y & &Y & &Y\times Y\times Y& &Y\times Y & &Y 
  \end{tikzcd} 
\end{center}
 satisfying the pentagon relation. 
\end{lemdefn}
\begin{proof}
  This follows from the description of associative algebra objects inside $(2,1)$-categories due to \cite[4.1.6.17]{Lu2}, the pentagon identity corresponding to the condition that the image of $\beta_A$ be null-homotopic.
\end{proof}

Likewise, since taking algebra objects commutes with Duskin nerves, we may use Dunn additivity to get a description of classical factorisation $\Eb_2$-spaces: 

\begin{lemdefn}\label{lem:ClassicalFactSpace}
A (\textit{strong}) \textit{classical braided factorisation structure} on $Y$ is equivalent to having two classical factorisation space structures 
\begin{center}
\begin{tikzcd}[row sep = {30pt,between origins}, column sep = {45pt,between origins}]
   &[-5pt]C\ar[rd,"p"]\ar[ld,swap,"q"] & &[15pt] &C'\ar[rd,"p'"]\ar[ld,"q'"'] &  \\  
  Y\times Y & & Y& Y\times Y& & Y 
\end{tikzcd}
\end{center}
which are compatible, i.e. there is an isomorphism between the correspondences
\begin{equation} \label{fig:OrdBraid}
  \begin{tikzcd}[row sep = {30pt,between origins}, column sep = {45pt,between origins}]
    & &(C\times C)\times_{Y\times Y}C'\ar[rd]\ar[ld] & & &[-15pt] &[-10pt] & &(C'\times C')\times_{Y\times Y}C\ar[rd]\ar[ld] & & \\
    &C\times C \ar[rd,"p\times p"]\ar[ld,"q\times q"']& &C'\ar[rd,"p'"]\ar[ld,"q'"'] & &\stackrel{\gamma}{\stackrel{\sim}{\to}} & &C'\times C' \ar[rd,"p'\times p'"]\ar[ld,"\sigma_{23}\cdot \, q'\times q'"']& &C\ar[rd,"p"]\ar[ld,"q"'] &  \\
    Y^2\times Y^2& &Y\times Y & &Y & &Y^2\times Y^2& &Y\times Y & &Y 
  \end{tikzcd}
\end{equation}
  and such that both factorisation structures share the same unit $1:\pt \to Y$.
\end{lemdefn}

We define a \textit{lax} or \textit{oplax classical braided factorisation space} as above, but weakening the assumptiuon that $\gamma$ be an isomorphism to just being a map rightwards or leftwards, respectively. Given a (lax) classical braided factorisation space $Y$, if we pull back the correspondences in (\ref{fig:OrdBraid}) along
$$(\pt\times Y)\times (Y\times \pt)\ \to\ Y^2\times Y^2$$
we get a map of correspondences 
\begin{center}
  \begin{tikzcd}[row sep = {30pt,between origins}, column sep = {45pt,between origins}]
     &[-5pt]C'\ar[rd,"p'"]\ar[ld,swap,"q'"] &&[5pt]\stackrel{\gamma_{23}}{\to} &[5pt] &C\ar[rd,"p"]\ar[ld,"q"'] &  \\  
    Y\times Y & & Y&& Y\times Y& & Y 
  \end{tikzcd}
  \end{center}
  which is an isomporphism if $\gamma$ is. By analogy with lax braided monoidal categories, we might expect $\gamma_{23}$ to determine $\gamma$. This is not true in general, but is still true for strong classical braided factorisation spaces:

\begin{prop} \label{prop:OrdBraidFactSpace}
  The data of a (strong) classical braided factorisation space is the same an isomorphism $\tilde{\sigma}$ fitting into a commuting diagram 
  \begin{equation}\label{fig:BraidingFactSpace}
    \begin{tikzcd}[row sep = {30pt,between origins}, column sep = {45pt,between origins}]
     &C\ar[rdd,"p", bend left = 30]\ar[ldd,swap,"\sigma\cdot q", bend right = 30]\ar[d,swap,"\wr","\tilde{\sigma}"'] & \\[-5pt]
     &C\ar[rd,"p"]\ar[ld,swap,"q"] & \\
    Y\times Y & & Y
    \end{tikzcd}
    \end{equation}
  satisfying the hexagon identities, where $\sigma:Y\times Y \stackrel{\sim}{\to} Y\times Y$ is the swap map. $Y$ is a \emph{commutative} classical factorisation space if $\tilde{\sigma}$ is an involution. 
\end{prop}
\begin{proof}
  We construct $\tilde{\sigma}=\gamma_{23}$ as above, and the hexagon relations are easily checked. Conversely, from (\ref{fig:BraidingFactSpace}) and having fixed an identification $C=C'$, we get a pair of isomorphisms between the middle rows of (\ref{fig:OrdBraid}), which in turn induces an isomorphism $\gamma$ between the two pullbacks. 
\end{proof}

 Likewise, Proposition \ref{prop:OrdBraidFactSpace} is true if we have a lax classical braided factorisation structure and identification $C=C'$, but not for general lax classical braided factorisation spaces. Finally, we need two Lemmas.

\begin{lem} \label{lem:C'IsFactSpace}
  Let $Y$ be a (strong or oplax) classical braided factorisation space. Then the map $C'\to Y$ is a map of classical factorisation spaces. 
\end{lem}
\begin{proof}
  We have a commuting diagram whose left square is a pullback 
  \begin{equation}\label{fig:C'FactStructure}
    \begin{tikzcd}[row sep = {15pt}, column sep = {45pt,between origins}]
    C'\times C'\ar[d] &[50pt](C'\times C')\times_{Y\times Y}C\ar[r,dashed]\ar[d]\ar[l] &[30pt] C'\ar[d] \\ 
    Y\times Y &C\ar[l]\ar[r] &Y 
    \end{tikzcd}
  \end{equation}
  where the dashed arrow is induced by $\gamma$ and $(C\times C)\times_{Y\times Y}C'\to C'$. The right square commutes because $\gamma$ is a map over $Y$. 

  It remains to show that the top row of (\ref{fig:C'FactStructure}) defines a classical factorisation structure on $C'$. To supply an associativity isomorphism
  \begin{center}
    \begin{tikzcd}[row sep = {30pt,between origins}, column sep = {45pt,between origins}]
      & &Z^-\ar[rd]\ar[ld] & & &[-15pt]  &[-5pt] & &Z^+\ar[rd]\ar[ld] & & \\
      & ({C'}^2\times_{Y^2}C) \times C' \ar[rd]\ar[ld]& &{C'}^2\times_{Y^2}C\ar[ld]\ar[rd] & &\stackrel{\sim}{\to} & &C'\times ({C'}^2\times_{Y^2}C) \ar[ld]\ar[rd]& &{C'}^2\times_{Y^2}C\ar[ld]\ar[rd] &  \\
      C'\times C'\times C'& &C'\times C'& &C' & &C'\times C'\times C'& &C'\times C' & &C'
    \end{tikzcd} 
  \end{center}
  note
  $$Z^-\ \simeq\ (({C'}^2\times_{Y^2}C) \times C')\times_{Y^2}C\ \simeq\ (({C'}^2\times_{Y^2}C) \times C')\times_{C\times Y}\left((C\times Y)\times_{Y^2}C\right)$$
  and likewise for $Z^+$, the isomorphism is then induced by the associativity isomorphism for $C$
  $$(C\times Y)\times_{Y^2}C\ \simeq\ (Y\times C)\times_{Y^2}C.$$
  The pentagon and unit axioms are checked similarly. 
\end{proof}

Similarly, 

\begin{lem}\label{lem:C'ToYxYIsFact}
  Let $Y$ be a (strong or lax) classical braided factorisation space. Then the map $C'\to Y\times Y$ is a map of classical factorisation spaces. 
\end{lem}
\begin{proof}
  We have a commuting diagram whose right square is a pullback
  \begin{equation}\label{fig:C'ToYxYIsFact}
    \begin{tikzcd}[row sep = {15pt}, column sep = {45pt,between origins}]
    C'\times C'\ar[d] &[50pt](C\times C)\times_{Y\times Y}C'\ar[r]\ar[d]\ar[l,dashed] &[30pt]C'\ar[d] \\ 
    Y^2\times Y^2 &C\times C \ar[r,"p\times p"]\ar[l,"\sigma_{23}\cdot (q\times q)"'] & Y\times Y 
    \end{tikzcd}
  \end{equation}
where the dashed arrow is the map $\gamma:(C\times C)\times_{Y\times Y}C'\to (C'\times C')\times_{Y\times Y}C$ composed with the projection to $C'\times C'$. It suffices to check that both rows define factorisation structures. For the bottom row this is immediate, and for the top row the associativity and unit axioms are easy to check and proceed similarly to Lemma \ref{lem:C'IsFactSpace}.
\end{proof}

\subsection{Classical factorisation categories} \label{sec:ClassicalFactorisationCategories}  Next we need a theory of sheaves of categories on classical prestacks. Loosely speaking, we will construct a sheaves of categories version of \cite[Prop 9.0.2]{CF}.

To begin, notice that if $X$ is a classical prestack, then $\QCoh(X)$ is naturally endowed with a $t$-structure, and so we define the functor 
$$\ShvCat^{t \textup{-ex}}\ :\ \PreStk^{\textup{cl},\textup{corr}}\ \to\ \Cat_\infty$$
on affine prestacks by $X\mapsto \QCoh(X)\Md({\Pr}^{t \textup{-ex}})$, then and Kan extending as in \cite{Ga1}. Here $\Pr^{t \textup{-ex}}$ denotes the category of stable presentable categories equipped with an accesible $t$-structure, which by is equivalent via
$$(-)^\heartsuit\ :\ {\Pr}^{t \textup{-ex}} \stackrel{\sim}{\leftrightarrow}\ \textup{Ab}\ :\ D(-)$$
to the category $\textup{Ab}$ of Grothendieck abelian categories with exact functors between them, see the proof of \cite[Prop 9.0.2]{CF}.  The forgetful functor ${\Pr}^{t \textup{-ex}}\to \Cat_\infty$ induces a natural transformation $\ShvCat^{t \textup{-ex}}\to \ShvCat$ and hence a map of the associated map  of the associated symmetric monoidal cocartesian fibrations over $\PreStk^{\textup{cl},\textup{corr}}$, denoted 
$$\Groth^{t \textup{-ex}}\ =\ \Groth^{\textup{cl}}\to \Groth.$$
We call an object  $\El\in \Groth^{\textup{cl}}$ a \textit{classical sheaf of categories}, which may be viewed as taking values in Grothendieck abelian categories or equivalently in stable presentable categories  equipped with a presentable $t$-structure.

\begin{defn}
  A \textit{classical factorisation category} is an associative algebra $\El$ in $\Groth^{\textup{cl}}$. 
\end{defn}

In particular, it lies over a classical factorisation space $X$. As before, a \textit{classical factorisation $\Eb_n$-category} is an $\Eb_n$-algebra in $\Groth^{\textup{cl}}$, and we use the name ($\varnothing$, \textit{braided}, \textit{commutative}) \textit{classical factorisation category} when $n=1,2,\infty$. Recall by \cite[C.5.4.3]{Lu3} that $\textup{Ab}$ is a $(2,1)$-category, and so if $X$ is a classical affine prestack, then
   $$\QCoh(X)\Md({\Pr}^{t \textup{-ex}})\ \simeq\ \QCoh(X)^\heartsuit\Md(\textup{Ab})$$
   is also a $(2,1)$-category. It follows that $\ShvCat^{t \textup{-ex}}$ factors through the Duskin nerve functor $\Nt :\Cat_2\to \Cat_\infty$, and so it follows that
$\Groth^{\textup{cl}}$ is a $(2,1)$-category.

\begin{defn}
  A \textit{classical factorisation category} is an associative algebra object in $\Groth^{\textup{cl}}$. 
\end{defn}

Likewise \textit{classical factorisation $\Eb_n$-categories} are $\Eb_n$-algebra objects in $\Groth^{\textup{cl}}$, and we use the name ($\varnothing$, \textit{braided}, \textit{commutative}) \textit{classical factorisation category} when $n=1,2,\infty$. As in the previous section, we have
\begin{lemdefn}\label{lemdefn:BraidedFactCat}
  A \textit{classical factorisation category} is equivalent to a classical factorisation space $Y$ along with $\El\in \Groth^{\textup{cl}}$ equipped a map 
  $$\otimes_\El\ :\ q^*(\El\boxtimes\El)\ \to\ p^*\El$$
  satisfying an associativity condition, and a compatible unit $1_\El:\triv\to \El$. Likewise, a \textit{classical braided factorisation category} is a classical braided factorisation space $Y$ along with $\El\in \Groth^{\textup{cl}}$ equipped with with two classical factorisation structures 
  $$\otimes_\El\ :\ q^*(\El\boxtimes \El)\ \to\ p^*\El, \hspace{15mm} \otimes_\El'\ :\ {q'}^*(\El\boxtimes\El) \ \to\ {p'}^*\El$$
  which are compatible in the sense that we have a \textit{braiding} natural transformation\footnote{We use the following as notation for $\beta\ :\ \gamma^*\left(\otimes_\El(\otimes_\El'\boxtimes\otimes_\El')\right)\ \stackrel{\sim}{\to}\ \otimes_\El'(\otimes_\El\boxtimes \otimes_\El)$.}
  \begin{equation}\label{eqn:PreBraidingFactCat}
   \beta\ :\ \gamma^*\left((\, \otimes_\El'\,)\otimes_\El(\, \otimes_\El'\, )\right) \ \stackrel{\sim}{\to}\ \left((\, \otimes_\El\,)\otimes_\El'(\, \otimes_\El\, )\right)
  \end{equation}
  are identified under the isomorphism $\gamma$ between the correspondences in (\ref{fig:OrdBraid}). 
\end{lemdefn}

  If $Y$ is a lax classical braided factorisation space, $\El$ above is a (\textit{op-})\textit{lax} classical braided factorisation category if there is just a map rightwards (leftwards). Next,

\begin{prop}
   If $Y$ is a strong classical braided factorisation space,  then to give a classical braided factorisation category structure on $\El$ is the same as to give a factorisation category structure $\otimes_\El$ and a natural transformation also denoted $\beta :\otimes_\El \to  \tilde{\sigma}^*\otimes_\El$, i.e.
   \begin{equation}\label{fig:FactBraidingNonCommutativeDiagram}
  \begin{tikzcd}[row sep = 10pt, column sep = 30pt]
    q^*(\Cl\boxtimes\Cl)\ar[d,"\wr"]\ar[r,"\otimes_\Cl"{xshift=-0pt},""{name=U,inner sep=1pt,below}, bend left = 35] & p^*\Cl\ar[d,"\wr"]\\ 
    \tilde{\sigma}^*q^*(\Cl\boxtimes\Cl)\ar[r,"\tilde{\sigma}^*\otimes_\Cl"'{xshift=-0pt}, ""{name=D,inner sep=1pt}, bend right = 35] & \tilde{\sigma}^*p^*\Cl
    \arrow[Rightarrow, from=U, to=D, "\beta", shorten <= 5pt, shorten >= 5pt] 
    \arrow[Rightarrow, from=U, to=D, swap,"", shorten <= 5pt, shorten >= 5pt]
    \end{tikzcd}
  \end{equation}
   which satisfies the hexagon relations. It is strong if $\beta$ is a natural isomorphism. 
\end{prop}
\begin{proof}
   Given a classical braided factorisation category $\El$, pull back the correspondences in (\ref{fig:OrdBraid}) along
   $$(\pt\times Y)\times (Y\times \pt)\ \to\ Y^2\times Y^2$$
   gives an identification $\otimes_\El\simeq \otimes_\El'$, and the identification of (\ref{eqn:PreBraidingFactCat}) gives the braiding $\beta$. The converse is proven as in the proof of Proposition \ref{prop:OrdBraidFactSpace}.
\end{proof}

In the above, if $Y$ is in addition a commutative factorisation space, this is called \textit{symmetric} if  $\beta\cdot (\tilde{\sigma}^*\beta)=\id$.

\subsubsection{Module categories} Let $\El$ be a lax classical braided factorisation category, and let $A\in \Gamma(Y,\El)$ be an algebra with respect to $\otimes_\El$. Denote its product by 
$$\otimes_\El q^*(A\boxtimes A)\ \stackrel{m}{\to}\ q^*A$$
and its category of factorisation modules by $\Cl=A\FactMd$.
\begin{prop}
  There is a bijection between 
  \begin{enumerate}
    \item equivalence classes of lifts along $\Cl=A\FactMd\to \El$ of the classical factorisation category structure $\otimes_\El'$, and
    \item equivalence classes of factorisation bialgebra structures on $A$ extending the factorisation algebra structure. 
  \end{enumerate}
\end{prop}
\begin{proof}
    First choose a lift of $\otimes_\El'$, which we denote by the same symbol. Then $A\otimes_\El' A$ defines an element of $\Gamma(C',{p'}^*\Cl)$. By Lemma \ref{lem:C'IsFactSpace} the map $C'\to Y$ is a map of factorisation spaces, and so the pullback of $A$ acts on it, i.e. 
  $${p'}^*A\otimes_\El(A\otimes_\El' A)\ \to\ {\overbar{p}}^*(A\otimes_\El' A).$$
  This is a map of sections of the pullback of $\El$ to $(C'\times C')\times_{Y\times Y}C$, 
    \begin{equation}
    \begin{tikzcd}[row sep = {15pt}, column sep = {45pt,between origins}]
      C'\times C'\ar[d] &[50pt] (C'\times C')\times_{Y\times Y}C\ar[r,"\overbar{p}"]\ar[d]\ar[l] &[30pt]C'\ar[d] \\ 
    Y\times Y &C\ar[l]\ar[r] &Y 
      \end{tikzcd}
  \end{equation}
  and we have denoted by $\otimes_\El$ the pullback of the factorisation structure to $(Y\times C')\times_{Y\times Y}C$. Finally, consider the pullback
  \begin{equation}
    \begin{tikzcd}[row sep = {15pt}, column sep = {45pt,between origins}]
      C'\times \pt\ar[d] &[50pt] C'\ar[d,"a"]\ar[l] &[30pt]   \\ 
      C'\times C'&[50pt](Y\times C')\times_{Y\times Y}C\ar[r,"\overbar{p}"]\ar[l] &[30pt] C'
      \end{tikzcd}
  \end{equation}
  We pull back by $a$ to get a map 
  $${p'}^*A\ =\ A\otimes_\El(1_\El\otimes_\El' 1_\El)\ \to\ a^*(A\otimes_\El(A\otimes_\El' A))\ \to \ a^*{\overbar{p}}^*(A\otimes_\El'A)\ =\ A\otimes_\El'A$$
  of sections of ${p'}^*\El$, which defines a factorisation coproduct with respect to $\otimes'_\El$.  Convesely, assume that $A$ is a factorisation bialgebra. To lift of $\otimes_\El'$ to a classical factorisation structure on $\Cl$ we need a functorial assignment to every $M\in \Gamma(C',{q'}^*(\Cl\boxtimes\Cl))$ a lift of $\otimes_\El'M$ to an element of $\Gamma(C',{p'}^*\Cl)$. It is sufficient to provide a factorisation module action 
  $${p'}^*A\otimes_{\El}(\otimes_\El'M)\ \to\ \otimes_\El'M.$$
  To begin with, by Lemma \ref{lem:C'ToYxYIsFact} the map $q':C'\to Y\times Y$ is a map of a factorisation spaces, so we have a factorisation algebra action  
  \begin{equation}\label{eqn:FactActionAxA}
    {q'}^*(A\boxtimes A)\otimes_{\El\boxtimes \El} M\ \to\ (\overbar{p}\times\overbar{p})^*M
  \end{equation}
  as a map in $(\overbar{p}\times\overbar{p})^*{q'}^*(\El\boxtimes\El)$, as in diagram (\ref{fig:C'ToYxYIsFact}):
  \begin{center}
    \begin{tikzcd}[row sep = {15pt}, column sep = {45pt,between origins}]
      C'\times C'\ar[d,"q'\times q'"] &[70pt](C'\times C')\times_{Y\times Y}C\ar[r,"\overbar{p}\times\overbar{p}"]\ar[d,"b"]\ar[l,"\sigma_{23}\cdot (\overbar{q}\times \overbar{q})"'] &[45pt]C'\ar[d,"q'"] \\ 
      Y^2\times Y^2 &C\times C \ar[r,"p\times p"]\ar[l,"\sigma_{23}\cdot (q\times q)"'] & Y\times Y 
      \end{tikzcd} 
  \end{center}
  Here we use the factorisation category structure of $\El\boxtimes \El$ on $Y^2$. We will now take (\ref{eqn:FactActionAxA}) and apply the factorisation category structure 
  $$(\overbar{p}\times\overbar{p})^*\otimes_\El'\ :\ (\overbar{p}\times\overbar{p})^*{q'}^*(\El\boxtimes\El)\ \to\ (\overbar{p}\times\overbar{p})^*{p'}^*\El,$$ 
  after using the compatibility (\ref{eqn:PreBraidingFactCat}) between $\otimes_\El$ and $\otimes_\El'$ to simplify the right hand side:
  $$(A\otimes_\El' A)\otimes_\El(\otimes_\El'M)\ \to\ \otimes'_\El M.$$
  Then, composing with the factorisation coproduct ${p'}^*A\to A\otimes_\El' A$ gives the desired action.
\end{proof}

\subsubsection{Fibrewise description} We spell out consequences of the definition of classical braided factorisation category in a more concrete way. We have an isomorphism 
$$(A\otimes_\El B)\otimes_\El'(C\otimes_\El D)\ \simeq \ (A\otimes_\El' C)\otimes_\El(B\otimes_\El' D)$$
for all $A,B,C,D\in \Gamma(\El)$, coming from the compatibility condition in (\ref{lemdefn:BraidedFactCat}).\footnote{We have abused notation by omitting pullback along the isomorphism (\ref{fig:OrdBraid}) from the notation.} Next, on the fibre over a point  $c:\pt\to C$ this structure gives the following. If $c$ has images 
 \begin{center}
  \begin{tikzcd}[row sep = 10pt, column sep = 15pt]
    &[-5pt]c\ar[rd,"p",|->]\ar[ld,swap,"q",|->] &  \\
    (y_1,y_2) & & y 
    \end{tikzcd} 
\end{center}
Then pulling back the factorisation braiding (\ref{fig:FactBraidingNonCommutativeDiagram}) to $c$ gives
\begin{center}
\begin{tikzcd}[row sep = 15pt, column sep = 35pt]
  \El_{y_1}\otimes\El_{y_2}\ar[d,"\wr"]\ar[r,"\otimes_{\El,c}",""{name=U,inner sep=1pt,yshift=-5pt}] & \El_y\ar[d,equals]\\ 
  \El_{y_2}\otimes\El_{y_1}\ar[r,"\otimes_{\El,\tilde{\sigma}(c)}"',""{name=D,inner sep=1pt}]&\El_y
  \arrow[Rightarrow, from=U, to=D, "\beta_{c}", shorten <= 5pt, shorten >= 5pt] 
  \arrow[Rightarrow, from=U, to=D, swap,"\wr", shorten <= 5pt, shorten >= 5pt]
\end{tikzcd}
\end{center}
or in other words an isomorphism $A_1\otimes_{\El,c}A_2\simeq A_2\otimes_{\El,\tilde{\sigma}(c)}A_1$ in $\El_y$ for every pair of objects $A_i\in \El_{y_i}$. If $y=y_1=y_2$, which is the case if for instance $y$ is the unit of $Y$, this endows $\El_y$ with the structure of a braided monoidal category. So in general one might think of $\beta$ as being a braided monoidal structure ``spread out'' over $Y$.

\subsection{Factorisation quantum groups and spectral $R$-matrices} We can now define factorisation quantum groups. Just as usual $R$-matrices were defined, we define 
\begin{defn}
  A (\emph{strong}) \emph{spectral $R$-matrix} for a factorisation bialgebra $A$ in strong classical braided factorisation category $(Y,\El)$ is an element $R \in \Gamma(Y,A\otimes_\El' A)$ satisfying the hexagon relations\footnote{The hexagon relations are equalities of maps 
  $${p'}^*A\otimes_\El'A \ \to\ A\otimes_\El' A\otimes_\El'A,\hspace{15mm}A\otimes_\El' {p'}^*A\ \to\ A\otimes_\El' A\otimes_\El'A$$
  in the pullback of $\El$ to the associativity correspondence, and almost cocommutativity is an equality of maps ${p'}^*A\to A\otimes_\El'A$ in ${p'}^*\El$. } 
  \begin{equation} \label{eqn:SpectralRMatrixHexagon}
    (\id\otimes'_\El \Delta)\, R\ =\ R_{12}\cdot R_{13}, \hspace{15mm} (\Delta\otimes_\El' \id)\, R\ =\ R_{12}\cdot R_{23}
  \end{equation}
    and almost cocommutativity: 
  \begin{equation}\label{eqn:SpectralRMatrixAlmostCommutativity}
    \Delta\ =\ R\cdot \sigma(\Delta)\cdot R^{-1},
  \end{equation}
  where $\cdot$ means multiplication with respect to the factorisation product on $A$ and $\sigma$ is the braiding in $\El$. A \emph{lax spectral $R$-matrix} is the same but $R$ is allowed to not be an isomorphism.
\end{defn}

Next, we prove an analogue of Proposition \ref{prop:RMatrixBraidedMonoidal}:

\begin{theorem}\label{thm:SpectralRMatrixBraidedMonoidal} \label{thm:TannakaFactorisation}
  Let $A$ be a factorisation bialgebra in a strong classical braided factorisation category $\El$ over strong braided factorisation space $Y$. Then $\Cl=A\FactMd$ has a strong braiding if and only if $A$ has a spectral $R$-matrix. 
\end{theorem}
\begin{proof}
   If $\beta$ is a strong braiding for $\Cl$, then $R=\beta_{A,A}\sigma (1)$ defines a spectral $R$-matrix, where $\sigma$ is the braiding isomorphism for $\El$. It satisfies (\ref{eqn:SpectralRMatrixHexagon}) by the hexagon relations for $\beta$ and $\sigma$, and satisfies (\ref{eqn:SpectralRMatrixAlmostCommutativity}) by $A$-linearity of $\beta$. Conversely, if $R$ is a spectral $R$-matrix and $M\in \Gamma(C',{q'}^*(\Cl\boxtimes\Cl))$, then 
   $$\beta_M\ :\ \otimes_\El'M\ \stackrel{1_A\otimes_\El' 1_A}{\to}\ (A\otimes_\El'A)\otimes_\El(\otimes_\El'M)\ \stackrel{R}{\to}\ (A\otimes_\El'A)\otimes_\El(\otimes_\El'M)\ \stackrel{\act}{\to}\otimes_\El'M\ \stackrel{\sigma_M}{\to}\ \otimes_\El'M,$$
   and varying over all $M$, defines a braiding on the lift of $\otimes_\El'$ to $\Cl$.
\end{proof}
 It is easy to show that any spectral matrix $R$ satisfies the \textit{Spectral Yang-Baxter equation}:
\begin{center}
  \begin{tikzpicture}[scale=0.7]
  \begin{knot}[
    clip width=5,
    flip crossing=1,
    flip crossing=2,
    flip crossing=3,
    flip crossing=4,
    flip crossing=5,
    flip crossing=6,
  ]
  \strand[black,ultra thick] (0,0) .. controls +(0,1.5) and +(0,-1.5) .. (2,3);
  \strand[black,ultra thick] (1,0) .. controls +(0,1) and +(0,-1) .. (0,1.5);
  \strand[black,ultra thick] (0,1.5) .. controls +(0,1) and +(0,-1) .. (1,3);
  \strand[black,ultra thick] (2,0) .. controls +(0,1) and +(0,-1) .. (0,3);
  \node[] at (3.5,1.5) {$=$};
  \node[] at (14,2) {$R_{23}(y_2,y_3)\cdot R_{13}(y_1,y_3)\cdot R_{12}(y_1,y_2)$};
  \node[] at (17,1) {$\ =\ R_{12}(y_1,y_2)\cdot R_{23}(y_2,y_3)\cdot R_{13}(y_1,y_3).$};

  \strand[black,ultra thick] (5,0) .. controls +(0,1.5) and +(0,-1.5) .. (7,3);
  \strand[black,ultra thick] (6,0) .. controls +(0,1) and +(0,-1) .. (5+2,1.5);
  \strand[black,ultra thick] (5+2,1.5) .. controls +(0,1) and +(0,-1) .. (6,3);
  \strand[black,ultra thick] (7,0) .. controls +(0,1) and +(0,-1) .. (5,3);
  
  \end{knot}
  
  % Nodes on top
  \node[above] at (0,3) {$y_1$};
  \node[above] at (1,3) {$y_2$};
  \node[above] at (2,3) {$y_3$};
  
  % Nodes below
  \node[below] at (0,0) {$y_3$};
  \node[below] at (1,0) {$y_2$};
  \node[below] at (2,0) {$y_1$};

  % Nodes on top
  \node[above] at (5,3) {$y_1$};
  \node[above] at (6,3) {$y_2$};
  \node[above] at (7,3) {$y_3$};
  
  % Nodes below
  \node[below] at (5,0) {$y_3$};
  \node[below] at (6,0) {$y_2$};
  \node[below] at (7,0) {$y_1$};
  \end{tikzpicture}
  \end{center}
   In Theorem \ref{thm:SpectralRMatrixBraidedMonoidal}, we get a symmetric factorisation category structure if in addition $R$ is \textit{symmetric}:
\begin{center}
  \begin{tikzpicture}[scale=0.7]
  \begin{knot}[
    clip width=5,
    flip crossing=1,
  ]
  \strand[black,ultra thick] (1,0) .. controls +(0,1) and +(0,-1) .. (0,1.5);
  \strand[black,ultra thick] (0,1.5) .. controls +(0,1) and +(0,-1) .. (1,3);
  \strand[black,ultra thick] (0,0) .. controls +(0,1) and +(0,-1) .. (1,1.5);
  \strand[black,ultra thick] (1,1.5) .. controls +(0,1) and +(0,-1) .. (0,3);
  \node[] at (2.5,1.5) {$=$};
  \node[] at (13,1.5) {$R_{12}(y_1,y_2)R_{21}(y_2,y_1)\ =\ \id.$};
  
  \strand[black,ultra thick] (4,0) .. controls +(0,1.5) and +(0,-1.5) .. (4,3);
  \strand[black,ultra thick] (5,0) .. controls +(0,1.5) and +(0,-1.5) .. (5,3);
  \end{knot}
  
  % Nodes on top
  \node[above] at (0,3) {$y_1$};
  \node[above] at (1,3) {$y_2$};
  
  % Nodes below
  \node[below] at (0,0) {$y_1$};
  \node[below] at (1,0) {$y_2$};
  
  % Nodes on top
  \node[above] at (4,3) {$y_1$};
  \node[above] at (5,3) {$y_2$};
  
  % Nodes below
  \node[below] at (4,0) {$y_1$};
  \node[below] at (5,0) {$y_2$};
  \end{tikzpicture}
  \end{center}
  Here, $R(y_1,y_2)$ and $R(y_2,y_3)$ denote the pullbacks of $R\boxtimes 1_A$ and $1_A\boxtimes R$ along the leftwards maps in 
  \begin{center}
  \begin{tikzcd}[row sep = {30pt,between origins}, column sep = {45pt,between origins}]
    & (C\times Y)\times_{Y\times Y}C\ar[rd]\ar[ld] & &[25pt]& (Y\times C)\times_{Y\times Y}C\ar[rd]\ar[ld] &   \\ 
    C\times Y& & C & Y\times C& & C &
  \end{tikzcd}
  \end{center}
and likewise write $R(y_1y_2,y_3)$ and $R(y_1,y_2y_3)$ for the pullbacks of $R$ along the rightwards maps. Similarly, denote the application of the braiding $\sigma_\El R(y_1,y_2)$ by $R(y_1,y_2)$, and $R(y_1,y_3)$ denotes the pullback of $R\boxtimes 1_A$ along the dashed map
\begin{center}
\begin{tikzcd}[row sep = {30pt,between origins}, column sep = {45pt,between origins}]
  &&[-10pt]&& (C\times Y)\times_{Y\times Y}C\ar[ld]\ar[llld,dashed, bend right = 10] & \\ 
  &C\times Y\ar[ld,"q\times\id"']&&C\times Y\ar[ld,"q\times \id"]\ar[ll,"\sim"]& & \\
  Y\times Y\times Y&&Y\times Y\times Y\ar[ll,"\sigma_{23}"',"\sim"]&& & 
\end{tikzcd}
\end{center}
We also write $R(y_i,y_j)=R_{ij}(y_i,y_j)=R_{ij}$.
\begin{defn}
  A factorisation bialgebra $A\in \Gamma(Y,\El)$ equipped with a spectral $R$-matrix is called a \textit{factorisation quantum group} or \textit{quasitriangular factorisation bialgebra}. 
\end{defn}

Likewise, we have a lax version of the above, an analogue of Proposition \ref{prop:LaxRMatrixBraidedMonoidal}:

\begin{prop}
  Let $A$ be a factorisation bialgebra in a lax classical braided factorisation category $(Y,\El)$. Then every element 
  $$R\ \in\ (A\otimes_\El A)\otimes_\El'(A\otimes_\El A)$$
  satisfying the factorisable analogues of (\ref{eqn:LaxRMatrixHexagon}), (\ref{eqn:LaxRMatrixAlmostCocommutativity}) and compatiblity with the counit defines a lax braiding on the factorisation category structure on $\Cl=A\FactMd$.
\end{prop}

Such an element $R$ is called a \textit{lax spectral $R$-matrix}, and $A$ is called a \textit{lax factorisation quantum group}.

\subsubsection{Comodule versions} If $A$ is a factorisation \textit{co}algebra in (lax) classical braided factorisation category $\El$, with coproduct denoted $\Delta':A\to A\otimes'_\El A$, we can dualise the previous sections. As before we assume that in addition $\otimes_\El$ has a strong braiding compatibly with the lax braiding between $\otimes_\El$ and $\otimes_\El'$. 

\begin{prop}
   A lift of the monoidal structure along $A\CoMd\to \El$ is equivalent to a bialgebra structure on $A$. Let $A$ have two bialgebra structures with products
   $$m \ :\ A\otimes_\El A\ \to\ A, \hspace{15mm} m' \ :\ A\otimes'_\El A\ \to\ A.$$
  Then there is a lax braiding of the monoidal structures $\otimes_\El$ and $\otimes'_\El$ on $A\CoMd$ if and only if there is a map in $\Gamma((C\times C')\times_{Y\times Y}C,\tilde{p}^*\El)$,
  $$\tilde{R}^*\ :\ (A\otimes_\El A)\otimes'_\El(A\otimes_\El A)\ \to\ \tilde{p}^*1_\El$$
  satisfying the (duals of) the relations (\ref{eqn:LaxRMatrixHexagon}) and (\ref{eqn:LaxRMatrixAlmostCocommutativity}), and compatibly with the unit. Here $\tilde{p}$ denotes the projection to $Y$.
\end{prop}

In the above case, we also call $A$ a \textit{factorisation lax op-quantum group}, and $R^*$ a \textit{factorisation lax op-$R$-matrix}. As in Corollary \ref{cor:LaxRMatrixBraidedMonoidal}, we have a simpler version of the above:

\begin{cor}
  For any map in $\Gamma(C',{p'}^*\El)$
  $$R^*\ :\ A\otimes_\El'A\ \to\ {p'}^*1_\El$$
  satisfying the dual conditions to (\ref{eqn:LaxRMatrixHexagonSimple}) and (\ref{eqn:LaxRMatrixAlmostCocommutativitySimple}), and compatibility with the unit, the composition $\tilde{R}^*=(\epsilon\otimes_\El \id)\otimes'_\El(\id\otimes_\El\epsilon)(R^*)$ defines a factorisation op-$R$-matrix. 
\end{cor}

\subsubsection{Remark} Instead of working with spectral $R$-matrices, some sources work instead with the \textit{bicharacter} of an $R$-matrix:
$$r\ :\ A\otimes_\El' A\ \stackrel{R\cdot}{\to}\ A\otimes_\El' A\ \stackrel{\epsilon\otimes\epsilon}{\to}\ 1_\El \otimes_\El' 1_\El\ \simeq\ {p'}^*1_\El,$$
The cap product $\cap r=R\cdot$ recovers the $R$-matrix, no information is lost by considering $r$. We can take
$$\Gamma(r)\ :\ \Gamma(C,A\otimes_\El' A)\ \to\ \Gamma(C,{p'}^*1_\El).$$
When $\El=\Dl\Md_Y$ the right hand side consists of locally constant functions, and we get an element 
$$\Gamma(r)\ \in\ \Gamma(C,A\otimes_\El' A)^*.$$
In the vertex algebra case $Y=\Ran^{ch}(\Ab^1_{dR})/\Ga$ and the vertex bialgebra $A$ is commmutative and cocommutative, this corresponds to the notion of (\textit{vertex}) \textit{bicharacter} due to Anguelova \cite{An}, and what is there called $H_D\otimes H_D$-covariance corresponds to the fact that $R$ and $r$ are maps of $\Dl$-modules.

\subsubsection{Remark on twistings} \label{sec:Twistings}
 Given any braiding $\beta$ on one-category $\Cl$, we may precompose it with a $2$-automorphism of the identity 
\begin{equation}\label{fig:GammaBraiding}
  \begin{tikzcd}[row sep = 20pt, column sep = 30pt]
  \Cl\otimes\Cl \ar[r,"\id"{xshift=0pt},bend left = 35, ""{name=U,inner sep=1pt,below}]\ar[r,"\id"{xshift=0pt},bend right = 35,swap, ""{name=D,inner sep=1pt}]&  \Cl\otimes\Cl
    \arrow[Rightarrow, from=U, to=D, "\gamma", shorten <= 5pt, shorten >= 5pt] 
    \arrow[Rightarrow, from=U, to=D, swap,"\wr", shorten <= 5pt, shorten >= 5pt]
  \end{tikzcd}
  \end{equation}
  satisfying the hexagon identity to get a new braiding:
  \begin{prop}\label{prop:GammaBraidedMonoidal}
    $\beta\cdot\gamma$ defines a braided monoidal structure if and only if $\gamma\cdot (\id\otimes\, \otimes_\Cl) = \gamma_{13}\cdot \gamma_{12}$ and $\gamma\cdot (\otimes_\Cl\,\otimes\id) =\gamma_{13}\cdot \gamma_{23}$.\footnote{Spelling this out explicitly, 
    \begin{equation}\label{eqn:GammaRMatrix}
        \gamma_{A,B\otimes_\Cl C}\ =\ (\beta_{A,B}^{-1}\cdot \gamma_{A,C} \cdot \beta_{A,B})\cdot \gamma_{A,B}, \hspace{15mm}\gamma_{A\otimes_\Cl B,C}\ =\ (\beta_{B,C}^{-1}\cdot \gamma_{A,C} \cdot \beta_{B,C})\cdot \gamma_{B,C}
      \end{equation}
      for all objects $A,B,C$ of $\Cl$. }
  \end{prop}
\begin{proof}
  These conditions are equivalent to the first and second hexagon relations for $\beta\cdot\gamma$, respectively. The first hexagon relation reads
  $$\beta_{A,B\otimes_\Cl C}\cdot \gamma_{A,B\otimes_\Cl C}\ = \   \beta_{A,C}\cdot \gamma_{A,C}\cdot\beta_{A,B}\cdot\gamma_{A,B}$$
  as maps $A\otimes_{\Cl}B\otimes_{\Cl}C \stackrel{\sim}{\to} B\otimes_{\Cl} C\otimes_{\Cl} A$, and similarly for the second. 
\end{proof}

We can show a factorisable analogue of the above. For $\beta$ a strong braiding on classical sheaf of categories $\El$, we may precompose it with a $2$-automorphism of the identity 
\begin{equation}\label{fig:SpectralGammaTwist}
  \begin{tikzcd}[row sep = 20pt, column sep = 30pt]
  q^*(\Cl\boxtimes\Cl) \ar[r,"\id"{xshift=0pt},bend left = 30, ""{name=U,inner sep=1pt,below}]\ar[r,"\id"{xshift=0pt},bend right = 30,swap, ""{name=D,inner sep=1pt}]&  q^*(\Cl\boxtimes\Cl)
    \arrow[Rightarrow, from=U, to=D, "\gamma", shorten <= 5pt, shorten >= 5pt] 
    \arrow[Rightarrow, from=U, to=D, swap,"\wr", shorten <= 5pt, shorten >= 5pt]
  \end{tikzcd}
  \end{equation}
  to get a new factorisation braiding:
  \begin{prop}
    \label{prop:TwistedFactBraiding}
    $\beta\cdot\gamma$ defines a braided factorisation structure on $\Cl$ if and only if $\gamma\cdot (\id\otimes\, \otimes_\Cl) =\gamma_{13}\cdot \gamma_{12}$ and $\gamma\cdot (\otimes_\Cl\,\otimes\id)\ =\ \gamma_{13}\cdot \gamma_{23}$.\footnote{Spelling this out, we mean 
    \begin{equation}\label{eqn:SpectralGammaRMatrix}
      \gamma_{A,B\otimes_\Cl C}\ =\ (\beta_{A,B}^{-1}\cdot \gamma_{A,C} \cdot \beta_{A,B})\cdot \gamma_{A,B}, \hspace{15mm}\gamma_{A\otimes_\Cl B,C}\ =\ (\beta_{B,C}^{-1}\cdot \gamma_{A,C} \cdot \beta_{B,C})\cdot \gamma_{B,C}
    \end{equation}
   for all sections of the pullback of $\Cl$ to $(Y\times C)\times_{Y\times Y}C$, which we denote $A\boxtimes B\boxtimes C$.}
   \end{prop}

   It follows that 
   \begin{cor}
     Let $A$ be a (factorisation) quantum group with (spectral)$R$-matrix $R$. Then if $S$ is an invertible element of $A\otimes A$ (or $\Gamma(C,A\otimes_\Cl'A)$), the element $R\cdot S$ defines a (spectral) $R$-matrix if and only if 
     \begin{equation} \label{eqn:SMatrixHexagon}
       (\Delta\otimes\id)S\ =\ S_{13}S_{23}, \hspace{10mm} (\id\otimes\Delta)S\ =\ S_{13}S_{12}.
      \end{equation}
     and
      \begin{equation}\label{eqn:SAlmostCommutativity}
       \Delta(-)\ =\ S\Delta(-)S^{-1}, \hspace{15mm} R_{23}S_{13}\ =\ S_{13} R_{23}, \hspace{10mm} R_{12}S_{13}\ =\ S_{13} R_{12},
      \end{equation}
      which for instance holds if $S$ is central. 
   \end{cor}

\subsection{The diagonal factorisation structure}\label{sec:DiagonalMap} \label{sec:DiagonalFactStruct} Any prestack $Y$ has a tautological commutative factorisation structure coming from the diagonal map, whose multiplication and unit are 
\begin{center}
 \begin{tikzcd}[row sep = {30pt,between origins}, column sep = {45pt,between origins}]
 & Y\ar[ld,"\Delta"']\ar[rd,equals] & &[15pt] & Y\ar[rd,equals]\ar[ld] &  \\
 Y\times Y & & Y & \pt&& Y
 \end{tikzcd}
 \end{center} 
A factorisation $\Eb_n$-category on $Y$ with respect to this is just an $\Eb_n$-algebra object in $(\ShvCat(Y),\otimes)$. Moreover, this factorisation structure is automatically compatible with any other cofactorisation structure on $Y$:
\begin{prop} 
  \label{prop:BiAlgFactSpace}
  If $Y$ is a factorisation $\Eb_n$-space, the diagonal factorisation structure makes it into a lax factorisation $\Eb_n\otimes \Eb_\infty\simeq \Eb_\infty$-space, i.e. the two structures are compatible.
\end{prop}
\begin{proof}
  There is a symmetric monoidal functor
  \begin{equation}\label{eqn:DiagFactStr}
    \PreStk^{\textup{corr}}\ \to\ \Eb_\infty\Ag(\PreStk^{\textup{corr}})
  \end{equation}
  acting on objects by $Y\mapsto (Y,\Delta)$, and taking categories of $\Eb_n$-coalgebra objects of both sides gives the result.
\end{proof}

We describe how the functor (\ref{eqn:DiagFactStr}) acts on $1$-morphisms. Given any morphism in $\PreStk^{\textup{corr}}$, i.e. a correspondence $X\leftarrow C\to Y$, the same correspondence defines a map of factorisation spaces with respect to the diagonal factorisation structures. Indeed, get a two-morphism in $\PreStk^{\textup{corr}}$ 
\begin{center}
  \begin{tikzcd}[row sep = {30pt,between origins}, column sep = 20pt]
   X\ar[r]\ar[d,"\Delta_X",""{name=U,inner sep=2pt,xshift=20pt}] &Y\ar[d,"\Delta_Y",""{name=D,inner sep=2pt,xshift=-20pt}] \\
   X\times X \ar[r]&Y\times Y 
   \arrow[Rightarrow, from=U, to=D, "", shorten <= 2pt, shorten >= 2pt] 
  \end{tikzcd}
  \end{center}
induced by the map of correspondences $C\to X\times_{X\times X}(C\times C)$ from $X$ to $Y\times Y$.

\begin{cor}\label{prop:DiagBialg}
If $Y$ is any classical factorisation space, it forms a lax braided factorisation space with respect to the diagonal factorisation structure.
\end{cor}

To be explicit, the compatibility two-morphism is given by
\begin{center}
  \begin{tikzcd}[row sep = {30pt,between origins}, column sep = {45pt,between origins}]
    & &C\ar[rd,equals]\ar[ld] & & &[-15pt] &[-10pt] & &(C\times C)\times_{Y\times Y}Y\ar[rd]\ar[ld] & & \\
    &Y\times Y \ar[rd,equals]\ar[ld,"(\Delta\times\Delta)\cdot\sigma_{23}"']& &C\ar[rd]\ar[ld] & &\stackrel{\Delta_C}{\to} & &C\times C \ar[rd]\ar[ld]& &Y\ar[rd,equals]\ar[ld,"\Delta"'] &  \\
    Y^2\times Y^2& &Y\times Y & &Y & &Y^2\times Y^2& &Y\times Y & &Y 
  \end{tikzcd} 
\end{center}
formed by the diagonal map $C\to C\times C$ and the structure map $p:C\to Y$.

\newpage

\section{Vertex Quantum Groups}  \label{sec:QuantumVertexGroups}
\noindent
Vertex algebras are the prototypical example of (commutative) factorisation algebras, essentially for whom the theory was built. They were discovered by Borcherds in \cite{Bo1} and are an attempt to formalise the physics notion of two-dimensional conformal field theory.

We will recap the theory following \cite{FG} and \cite{Ras}, then explain what quantum factorisation algebras are in this setting: vertex quantum groups. The braided commutative objects in their categories of representations give as examples the structures considered by Etingof and Kazhdan \cite{EK} and Frenkel and Reshetikhin \cite{FR}, but have the advantage that their definition is more conceptual and without formulas (the formulas follow as a consequence).

\subsection{The associative Ran space} \label{sec:RanSpace} Let $X$ be any prestack. Its \textit{Ran space} is the colimit in $\PreStk$ of the diagram
\begin{center}
\begin{tikzcd}[row sep = 20pt, column sep = 20pt]
  \Ran X\ =\ \colim_{I \in \FinSet}X^I\ =\ \colim\bigg( X\ar[r,std] & X^2\ar[r,{yshift=3pt},std]\ar[r,{yshift=-3pt},std] \arrow[loop, distance=2em, in=50, out=130, looseness=5,"\Sk_2"] & X^3 \arrow[loop, distance=2em, in=50, out=130, looseness=5,"\Sk_3"] \ar[r,{yshift=6pt},std] \ar[r,{yshift=0pt},std]\ar[r,{yshift=-6pt},std] & \cdots\ \ \bigg)
\end{tikzcd}
\end{center}
where we have a symmetric group $\Sk_n$'s worth of arrows acting on $X^n$ by permuting the factors. Here $\FinSet$ denotes the category of nonempty finite sets with morphisms surjections. When $X$ is a scheme the Ran space parametrises its finite subsets, i.e. if $T$ is a scheme of finite type then $(\Ran X)(T)$ is equivalent to a set, the set of nonempty finite subsets of $X(T)$, see \cite{Ta}. 

It has two structures of a commutative factorisation space:
\begin{defn} \label{defn:RanFactStructure}
 $\Ran^*X$ and $\Ran^{ch}X$ are the commutative factorisation space structures
 \begin{center}
 \begin{tikzcd}[row sep = {35pt,between origins}, column sep = {45pt,between origins}]
  &\Ran X\times\Ran X\ar[ld,equals]\ar[rd,"\cup"] & &[30pt]  &(\Ran X\times\Ran X)_\circ\ar[ld,"j"']\ar[rd,"\cup\cdot j"] &  \\
  \Ran X\times \Ran X& & \Ran X&  \Ran X\times \Ran X& & \Ran X \end{tikzcd}
 \end{center}
 which is induced by 
\begin{center}
\begin{tikzcd}[row sep = {35pt,between origins}, column sep = {45pt,between origins}]
  &X^{I_1}\times X^{I_2}\ar[ld,equals]\ar[rd,"\cup_{I_1,I_2}","\sim"'] & &[30pt] &(X^{I_1}\times X^{I_2})_\circ\ar[ld,"j_{I_1,I_2}"']\ar[rd,"\cup\cdot j_{I_1,I_2}"] & \\
  X^{I_1}\times X^{I_2}& & X^{I_1\sqcup I_2} & X^{I_1}\times X^{I_2}& & X^{I_1\sqcup I_2}
\end{tikzcd}
\end{center}
over pairs of nonempty finite subsets $I_1,I_2$, where $(X^{I_1}\times X^{I_2})_\circ$ is the open locus on which entries of the $I_1$-tuple and $I_2$-tuple of points in $X$ are disjoint.
\end{defn}

One can also consider the \textit{associative Ran space}
\begin{center}
  \begin{tikzcd}[row sep = 20pt, column sep = 20pt]
    \Ran_a X\ =\ \colim_{I \in \FinSet^{\textup{ord}}}X^I\ =\ \colim\bigg( X\ar[r,std] & X^2\ar[r,{yshift=3pt},std]\ar[r,{yshift=-3pt},std] & X^3 \ar[r,{yshift=6pt},std] \ar[r,{yshift=0pt},std]\ar[r,{yshift=-6pt},std] & \cdots\ \ \bigg)
  \end{tikzcd}
  \end{center}
parametrising \textit{ordered} finite subsets of $X$, where $\FinSet^{\textup{ord}}$ is the category of ordered finite sets with morphisms order-preserving surjections. The above definition gives it two structures of a (non-commutative) factorisation space, $\Ran^*_aX$ and $\Ran^{ch}_aX$. In section \ref{sec:Examples} we also define a braided analogue $\Ran_{\textup{br}}X$, and there are maps 
$$\Ran_aX \ \to\ \Ran_{br} X\ \to\ \Ran X$$
which respect the associative and braided factorisation space structures, respectively.

\subsubsection{} For the remainder section \ref{sec:QuantumVertexGroups}, we will view $\Ran X$ and $\Ran_a X$ as a strong    classical braided factorisation space with respect to its chiral and diagonal factorisation structures. Concretely, the strong braiding  (\ref{fig:OrdBraid}) is
\begin{center}
  \begin{tikzcd}[row sep = {30pt,between origins}, column sep = {65pt,between origins}]
    & &(X^{I_1}\times X^{I_2})_\circ \ar[rd,equals]\ar[ld] & & & \\
    &X^{I_1}\times X^{I_2}\ar[rd,equals]\ar[ld]& &(X^{I_1}\times X^{I_2})_\circ\ar[rd]\ar[ld] & &  \hspace{30pt} \\
    X^{I_1}\times X^{I_2}\times X^{I_1}\times X^{I_2}& &X^{I_1}\times X^{I_2} & &X^{I_1\sqcup I_2} & 
  \end{tikzcd} 
\end{center}
\begin{center}
  \begin{tikzcd}[row sep = {30pt,between origins}, column sep = {65pt,between origins}]
     &[-15pt] &[-10pt] & &(X^{I_1}\times X^{I_2})_\circ  \ar[rd]\ar[ld] & & \\
   &\hspace{30pt}\stackrel{\id}{\stackrel{\sim}{\to}}& &(X^{I_1}\times X^{I_2})_\circ \times (X^{I_1}\times X^{I_2})_\circ \ar[rd]\ar[ld]& &X^{I_1\sqcup I_2}\ar[rd,equals]\ar[ld] &  \\
     & &X^{I_1}\times X^{I_2}\times X^{I_1}\times X^{I_2}& & X^{I_1\sqcup I_2} \times X^{I_1\sqcup I_2} & &X^{I_1\sqcup I_2}
  \end{tikzcd} 
\end{center}
As a result $\QCoh_{\Ran X}$ is a strong braided factorisation category, with braiding denoted
$$\beta\ :\ (\, \otimes^{ch}\, )\otimes(\, \otimes^{ch}\, )\ \stackrel{\sim}{\to}\ (\, \otimes\, )\otimes^{ch}(\, \otimes\, ).$$
Likewise, $\QCoh(\Ran X)$ is a lax braided monoidal category with monoidal structures $\otimes^{ch}=\cup_*\jmath_*\jmath^*$ and $\otimes$.\footnote{To be concrete, the braiding is
\begin{align*}
  \Gamma(\beta)\ :\ (A_1\otimes^{ch}A_2)\otimes (A_3\otimes^{ch}A_4) & \ =\ \cup_*\jmath_*\jmath^*(A_1\boxtimes A_2)\otimes \cup_*\jmath_*\jmath^*(A_3\boxtimes A_4) \\
  &\ \to\ \cup_*\jmath_*\jmath^*(A_1\boxtimes A_2\, \otimes\, A_3\boxtimes A_4) \\
  &\ =\ \cup_*\jmath_*\jmath^*(A_1\otimes A_3\, \boxtimes \, A_2\otimes A_4) \\
  &\ =\ (A_1\otimes A_3)\otimes^{ch} (A_2\otimes A_4)
\end{align*}
for any four sections $A_i\in \QCoh(\Ran X)$, where we have used that $*$-pullbacks (and pushforwards) are (lax) monoidal.}

\subsection{Categories over the Ran space} Notice that a sheaf of categories on the Ran space is an element in
$$\ShvCat(\Ran X)\ =\ \lim_{I\in \FinSet}(\ShvCat(X^I)),$$
i.e. a collection $\Cl=(\Cl_I)$ of sheaves of categories on  $X^I$ for every (ordered) finite set $I$, with compatible isomorphisms $\varphi_{I/J}:\Delta_{I/J}^*\Cl_{I}\stackrel{\sim}{\to}\Cl_J$. All the above remains true for equivariant Ran spaces as well, if we consider ordered finite sets.

\subsubsection{Factorisation categories} Let $\Cl$ be a classical sheaf of categories with respect to the chiral factorisation space structure on $\Ran X$ (or $\Ran_a X$). It is a classical factorisation category if in addition we have functors of sheaves of categories
$$\otimes_{I_1,I_2}^{ch}\ :\ (\Cl_{I_1}\boxtimes \Cl_{I_2})\vert_{(X^{I_1}\times X^{I_2})_\circ}\ \to\ (\cup_{I_1,I_2}^*\Cl_{I_1\cup I_2})\vert_{(X^{I_1}\times X^{I_2})_\circ}$$
compatible with the $\varphi_{I/J}$ and satisfying an associativity condition
$$\otimes_{I_1\cup I_2, I_3}^{ch}(\otimes_{I_1,I_2}\,\otimes\id )\ =\ \otimes_{I_1,I_2\cup I_3}^{ch}(\id\otimes\, \otimes_{I_2,I_3}^{ch})$$ 
as functors
$$(\Cl_{I_1}\boxtimes\Cl_{I_2}\boxtimes \Cl_{I_3})\vert_{(X^{I_1}\times X^{I_2}\times X^{I_3})_\circ}\ \to\ \left(\cup^*\Cl_{I_1\cup I_2\cup I_3}\right)\vert_{(X^{I_1}\times X^{I_2}\times X^{I_3})_\circ}.$$
Likewise, a factorisation category structure with respect to the diagonal factorisation space structure is a collection of functors 
$$\otimes_I\ :\ \Cl_{I}\otimes \Cl_{I}\ \to\ \Cl_{I}$$
compatible with the $\varphi_{I/J}$ and satisfying the associativity and unit condition. 

\subsubsection{} 

We now consider the chiral structure on the Ran space as a strong braided factorisation space. Let $\Cl$ be a classical braided factorisation category. In particular, we have maps 
$$\beta_{\{z_i\},\{w_j\}}\ :\ A_{\{z_1,...,z_n\}}\otimes^{ch}B_{\{w_1,...,w_m\}}\ \to\ B_{\{w_1,...,w_m\} } \otimes^{ch} A_{\{z_1,...,z_n\}}$$
for every pair of sections $A,B$ and every nonempty finite subsets $\{z_i\},\{w_j\}\subseteq X(k)$ which are disjoint. For instance, we get
$$\beta_{z,w}\ :\ A_z\otimes^{ch}B_w \ \to\ B_w\otimes^{ch}A_z$$
for all pairs of points $z\ne w$. We will show that under certain requirements $\beta_{z,w}$ determines the entire braiding. Fibrewise one can see this because the diagram
\begin{equation}\label{fig:SmallerMassiveDiagram}
  \begin{tikzcd}[row sep = {35pt,between origins}, column sep = {50pt}]
    \Cl_{z_1}\otimes \Cl_{z_2}\otimes \Cl_w \ar[d,"\beta_{z_2,w}"]\ar[r,"\otimes_{z_1,z_2}\otimes\id"]&\Cl_{\{z_1,z_2\}}\otimes\Cl_w\ar[dd,"\beta_{\{z_1,z_2\},w}"] \\ 
    \Cl_{z_1}\otimes \Cl_w\otimes \Cl_{z_2}\ar[d,"\beta_{z_1,w}"]&\\
    \Cl_w\otimes \Cl_{z_1}\otimes \Cl_{z_2}\ar[r,"\id\otimes \otimes_{z_1,z_2}"]&\Cl_w\otimes \Cl_{\{z_1,z_2\}}
    \end{tikzcd}
\end{equation}
commutes, and so if $\otimes_{z_1,z_2}$ is an equivalence, $\beta_{\{z_1,z_2\},w}$ is determined uniquely by the two-point braidings. Note that the spectral Yang-Baxter equation is now associated to triples of nonempty finite subsets 
$$R_{12}(\{z_i\})R_{13}(\{w_j\})R_{23}(\{u_k\})\ =\ R_{23}(\{u_k\})R_{13}(\{w_j\})R_{12}(\{z_i\})$$
and under the above assumptions, it will follow that this spectral $R$-matrix is uniquely determined by $R(z,w)$.

\begin{prop}\label{prop:BetaDeterminedBy11} Let $\Cl$ be a classical factorisation category which is strong, i.e. $\otimes_\Cl$ is an equivalence. Then to give a braiding $\beta$ is equivalent to giving an natural transformation 
  \begin{center}
    \begin{tikzcd}[row sep = 20pt, column sep = 30pt]
      q^*(\Cl\boxtimes\Cl)\vert_{(X\times X)_\circ}\ =\ &[-35pt] q^*(\Cl\vert_X\boxtimes\Cl\vert_X) \ar[r,"\otimes^{ch}"{xshift=0pt},bend left = 35, ""{name=U,inner sep=1pt,below}]\ar[r,"\sigma^*\otimes^{ch}"{xshift=0pt},bend right = 35,swap, ""{name=D,inner sep=1pt}]& \hspace{15pt} p^*\Cl\vert_{(X\times X)_\circ}
        \arrow[Rightarrow, from=U, to=D, "\beta_{1,1}", shorten <= 5pt, shorten >= 5pt] 
        \arrow[Rightarrow, from=U, to=D, swap,"", shorten <= 5pt, shorten >= 5pt]
      \end{tikzcd}
  \end{center}
satisfying the hexagon relations.
\end{prop}
\begin{proof}Given a braiding $\beta$, let $\beta_{1,1}$ be its restriction to $(X\times X)_0 \to (\Ran X\times \Ran X)_\circ$. This inherits the hexagon relations from $\beta$. To spell out the first hexagon relation, consider the diagram 
  \begin{equation}\label{fig:RanSpaceHexagon}
  \begin{tikzcd}[row sep = {30pt,between origins}, column sep = {65pt,between origins}]
 && (\Ran X\times \Ran X\times \Ran X)_\circ\ar[rd,"\tilde{\alpha}"] \ar[ld]& & \\ 
 &(\Ran X\times \Ran X)_\circ\times\Ran X\ar[rd,"\alpha"]\ar[ld] & & (\Ran X\times \Ran X)_\circ\ar[rd]\ar[ld] & \\ 
 \Ran X\times \Ran X\times\Ran X& & \Ran X\times \Ran X& & \Ran X
  \end{tikzcd}
\end{equation}
and the analogous diagram for the second hexagon relation. Since $\alpha$ and $\tilde{\alpha}$ are surjective on points, the $*$-pullbacks they induce on sheaves of categories is conservative. We then have the diagram of spaces over $\Ran X$
\begin{center}
\begin{tikzcd}[row sep = {15pt,between origins}, column sep = 30pt]
 (\Ran X\times \Ran X\times \Ran X)_\circ\ar[rd,"\tilde{\alpha}"] \ar[dd,"\tilde{\sigma}_{23}"]& \\ 
&(\Ran X\times \Ran X)_\circ \ar[dd,"\tilde{\sigma}"]\\
 (\Ran X\times \Ran X\times \Ran X)_\circ\ar[dd,"\tilde{\sigma}_{12}"]& \\ 
&(\Ran X\times \Ran X)_\circ\\
 (\Ran X\times \Ran X\times \Ran X)_\circ\ar[ru,"\tilde{\alpha}"] &
\end{tikzcd}
\end{center}
where for instance $\tilde{\sigma}_{12}$ is induced from diagram (\ref{fig:RanSpaceHexagon}) by applying $\tilde{\sigma}\times\id$ and $\id$ to the left and right middle terms. The first hexagon relation then says that in the following, which is a generalisation of (\ref{fig:SmallerMassiveDiagram}),
\begin{equation}\label{fig:MassiveDiagramBetaHexagon}
  \begin{tikzcd}[row sep = 10pt, column sep = 30pt, crossing over clearance=3ex]
    (\Cl\boxtimes\Cl\boxtimes \Cl)\vert_{((\Ran X)^3)_\circ}\ar[d,"\wr"]\ar[r,"(\otimes^{ch}\otimes\id)"{xshift=20pt}]\ar[rrd,""{name=U,inner sep=1pt,below,xshift=-80pt,yshift=-3pt}, bend left = 35] &[30pt] (\Cl\boxtimes\Cl)\vert_{((\Ran X)^3)_\circ}\ar[rd,"\otimes^{ch}",""{name=A,inner sep=1pt,below,xshift=-3pt}]& \\ 
    \tilde{\sigma}_{23}^*(\Cl\boxtimes\Cl\boxtimes \Cl)\vert_{((\Ran X)^3)_\circ} \ar[d,"\wr"]  \ar[rr,"\tilde{\sigma}_{12}^*(\otimes^{ch}\otimes\id)"{xshift=-30pt}, ""{name=M,inner sep=1pt,xshift=-89pt}]  & & \Cl\vert_{((\Ran X)^3)_\circ}\\
    \tilde{\sigma}_{23}^*\tilde{\sigma}_{12}^*(\Cl\boxtimes\Cl\boxtimes \Cl)\vert_{((\Ran X)^3)_\circ}\ar[r,"\tilde{\sigma}_{12}^*\tilde{\sigma}_{23}^*(\otimes^{ch}\otimes\id)"'{xshift=20pt,yshift=-5pt}"]\ar[rru, ""{name=D,inner sep=1pt,xshift=-80pt,yshift=3pt}, bend right = 35] & \tilde{\sigma}^*(\Cl\boxtimes\Cl)\vert_{((\Ran X)^3)_\circ}\ar[ru,"\tilde{\sigma}^*\otimes^{ch}"',""{name=B,inner sep=1pt,xshift=-0pt}]\ar[from=uu,"\wr",crossing over]&  
    \arrow[Rightarrow, from=U, to=M, "\beta_{23}"{yshift=-15pt}, shorten <= 5pt, shorten >= 5pt,crossing over] 
    \arrow[Rightarrow, from=M, to=D,"\beta_{12}"{yshift=15pt}, shorten <= 5pt, shorten >= 5pt,crossing over]
    \arrow[Rightarrow, from=A, to=B,"\beta", shorten <= 5pt, shorten >= 5pt,crossing over]
    \end{tikzcd} 
\end{equation}
we have $\beta=\beta_{12}\beta_{23}$. Note that the restriction of the above to $(X\times X\times X)_\circ$ is by definition the hexagon relations for $\beta_{1,1}$.

Now, $\beta$ is determined by its restrictions $\beta_{I_1,I_2}$ to $(X^{I_1}\times X^{I_2})_\circ$, and because $\tilde{\alpha}$ is surjective, also by its restriction to $(X^{I_1^-}\times X^{I_1^+}\times X^{I_2})_\circ$ for any partition $I_1=I_1^-\sqcup I_1^+$ into nonempty subsets. By restricting the first hexagon relation (\ref{fig:MassiveDiagramBetaHexagon}) to $(X^{I_1^-}\times X^{I_1^+}\times X^{I_2})_\circ$ and using that $\otimes^{ch}$ is an isomorphism, we have 
  $$\beta_{I_1,I_2}\ \simeq\ \beta_{I_1^-,I_2}\cdot \beta_{I_1^+,I_2}$$
and likewise for $I_2$. Applying this iteratively by partitioning $I_1,I_2$ into singleton sets, implies that every $\beta_{I_1,I_2}$ (and hence $\beta$) is uniquely determined by $\beta_{1,1}$. Conversely, given a $\beta_{1,1}$ we may use this to define $\beta_{I_1,I_2}$, which is independent of the sequence of partitions chosen by the hexagon relations.
\end{proof}

The main example of this is the following. Let $A$ be a bialgebra in a chiral-diagonal braided factorisation category $\El$, i.e. we have lax braided factorisation structures $\otimes$ and $\otimes^{ch}$ and structure maps  
$$\cdot \ :\ A\otimes_\El A\ \to \ A, \hspace{15mm} \Delta\ :\ (\cup\jmath)^*A\ \stackrel{\sim}{\to}\ A\otimes_\El^{ch}A.$$
Also assume that $\otimes^{ch}_\El$ is and $\Delta$ are equivalences, i.e. define strong factorisation structures on $\El$ and $A$, which implies that $\Cl = (A\FactMd(\El),\otimes^{ch})$ is a strong factorisation category. Let us in addition assume that $\otimes^{ch}$ has a braiding $\sigma_\El$.
\begin{cor} \label{cor:RanSpaceRMatrix}
   To endow $A$ with a spectral $R$-matrix $R \in \Gamma((\Ran X\times\Ran X)_\circ,A\otimes^{ch} A)$ is equivalent to giving a section 
   $$R(z,w)\ \in\ \Gamma\left((X\times X)\setminus X,A\otimes^{ch} A\right)$$
   with  
  \begin{equation}
    \begin{aligned} \label{eqn:HexagonRanSpace}
      & (\Delta\otimes^{ch}\id)R(z_{12},z_3)\ =\ R(z_1,z_3)\cdot R(z_2,z_3), \hspace{30mm}\\[3pt]
      & \hspace{30mm}(\id\otimes^{ch}\Delta)R(z_{1},z_{23})\ =\ R(z_1,z_2)\cdot R(z_1,z_3),
     \end{aligned} 
  \end{equation}
   and
  \begin{equation}\label{eqn:AlmostCommutativityRanSpace} 
  \sigma\Delta(-)\ =\ R(w,z)\cdot \Delta(-) \cdot R(w,z)^{-1}.
  \end{equation}
\end{cor}
\begin{proof}
  Apply Proposition \ref{prop:BetaDeterminedBy11} to the braiding $\beta=R\cdot \sigma_\El$ of the factorisation structure $(\Cl,\otimes^{ch})$.
\end{proof}

Note that conditions (\ref{eqn:HexagonRanSpace}) and (\ref{eqn:AlmostCommutativityRanSpace}) are equalities of elements of $A\otimes^{ch}A\otimes^{ch}A\vert_{(X\times X\times X)_\circ}$ and  maps  $A\vert_{X^2\setminus X} \to A\otimes A\vert_{X^2\setminus X}\to A\otimes^{ch}A\vert_{X^2\setminus X}$, respectively. If $\El$ in addition symmetric, the above is symmetric if and only if $\sigma_\El(R(z,w)) = R(w,z)^{-1}$.

\subsubsection{Holomorphic-topological $\Eb_n$-categories and algebras} We remark here that for a prestack $X$, we may define \textit{holomorphic-topological $\Eb_{n,m}$-categories} and \textit{-algebras} as factorsation $\Eb_{n}\otimes\Eb_m$-categories and -algebras over $\Ran X$, with its structure of a  factorisation $\Eb_\infty\otimes\Eb_\infty\simeq\Eb_\infty$-space induced by the chiral and diagonal commtutative factorisation structures, respectively.

\subsection{Vertex quantum groups} \label{sec:QuantumVertexAlgebras} We now consider equivariant factorisation quantum groups over the affine line, which mirrors the case of ordinary vertex algebras summarised in Appendix \ref{sec:RationalMultiplicativeElliptic}. There we show that strong commutative factorisation algebras correspond to vector spaces $V$ equipped with a map 
$$Y(z)\ :\ V\otimes V\ \stackrel{}{\to}\ V((z))$$
with extra data and conditions. This should be thought of as loosely analogous to a product in an associative algebra, and in this section we will be carrying this analogy over to quantum groups. 

\subsubsection{} 
For the remainder of this section, we work with factorisation $\Eb_1\otimes\Eb_\infty\simeq\Eb_\infty$-space and -category 
$$Y\ = \ \Ran^{ch}_a(\Ab^1_{dR})/\Ga, \hspace{15mm}\El\ =\ \QCoh_{\Ran^{ch}(\Ab^1_{dR})/\Ga},$$
or the completion $\El= \QCoh_{\Ran^{ch}(\Ab^1_{dR})/\Ga}^\wedge $ with the symmetric factorisation structure $\hat{\otimes}^{ch}$ defined in section \ref{sec:CompletedDModules}; in the following we will abuse notation and refer to either as $(\El,\otimes^{ch})$. Thus, vertex quantum groups will be certain strong factorisation $\Eb_1^{\otimes^{ch}}$-algebras equipped with a compatible $\Eb_1^{\otimes}$-coalgebra structure. They are analogues of quantum groups in $\Vect$, with the associative algebra structure replaced with an associative vertex algebra structure:
\begin{defn} \label{defn:VertexQuantumGroup}
  A \textit{vertex quantum group} is an (associative) vertex algebra $V$ with compatible\footnote{This means that
$\Delta(Y(\alpha,z)\beta) =\Delta(\alpha)\cdot_{Y(z)}\Delta(\beta)$, or, in a commutative diagram,
\begin{center}
\begin{tikzcd}[row sep = {30pt,between origins}, column sep = {20pt},ampersand replacement = \&]
V\otimes V\ar[d,"\Delta\otimes\Delta"]\ar[rr,"Y(z)"] \&[-2pt] \&[15pt] V((z))\ar[d,"\Delta"] \\ 
 (V\otimes V)\otimes (V\otimes V)\ar[r,"\sim"',"\sigma_{23}"]\&V\otimes V\otimes V\otimes V\ar[r,"\sim"',"Y(z)\otimes Y(z)"]\& V\otimes V((z)) 
\end{tikzcd}
\end{center}} coalgebra structure $\Delta$, and a linear map 
$$R(z,w)\ :\  V\otimes V \ \to\ k((z-w))$$
satisfying
  \begin{equation} \label{eqn:VertexQuantumGroupDerivation}
  \begin{aligned}
   &  R(z,w)(T\otimes \id) \ =\ \partial_z R(z,w) , \hspace{30mm}\\[3pt]
   & \hspace{30mm} R(z,w)(\id\otimes T) \ =\ \partial_w R(z,w),
  \end{aligned}
\end{equation}
  the hexagon relations 
  \begin{equation}\label{eqn:VertexQuantumGroupHexagon}
    \begin{aligned}
     & R(z,w)\ =\ (Y(u)\otimes\id) R_{13}(z+u,w)R_{23}(z,w), \hspace{30mm}\\[3pt]
     & \hspace{30mm} R(z,w)\ =\ (\id\otimes Y(u)) R_{13}(z,w+u)R_{12}(z,w),
    \end{aligned}
  \end{equation}
   almost cocommutativity
  \begin{equation}
    \label{eqn:VertexQuantumGroupAlmostCocommutativity} 
    \sigma\cdot  \Delta(-)\ =\ R(z,w)\Delta(-)R(w,z)^{-1}
  \end{equation} 
  and compatibility with the unit: $R(|0\rangle\otimes \id)=R(\id\otimes|0\rangle)=1$.  An $R$ as above is called a \textit{vertex $R$-matrix}. It is \textit{symmetric} if $\sigma R(z,w)=R(w,z)^{-1}$, where inverses are taken using the coproduct $\Delta$.
\end{defn}

The compatibility (\ref{eqn:VertexQuantumGroupDerivation}) between vertex $R$-matrices and the derivation in the vertex algebra appears as ``$H_D\otimes H_D$ covariance'' of the associated bicharacter $r(z,w)$ in \cite{An}. Here equation (\ref{eqn:VertexQuantumGroupAlmostCocommutativity}) takes place inside $V\otimes V\subseteq V\otimes V((z-w))$. It is easy to show that any vertex $R$-matrix satisfies the \textit{vertex Yang-Baxter equation}:
  \begin{equation}\label{eqn:VertexYangBaxter}
    R_{12}(z,w)R_{13}(z,u)R_{23}(w,u) = R_{23}(w,u)R_{13}(z,u)R_{12}(z,w),
  \end{equation}
an equality of elements in $V\otimes V\otimes V[[z^{\pm 1},w^{\pm 1},u^{\pm 1}]]$. Similar definitions have been made before, e.g. in \cite{EK}. One new angle is that they are examples of quantum factorisation algebras, and hence fit into the general theory of factorisation algebras:

\begin{theorem} \label{thm:VertexQuantumGroupFactorisationQuantumGroup}
  The data of a vertex quantum group $V$ is equivalent to a factorisation op-quantum group $A$ in $\El$.
\end{theorem}
\begin{proof}
  By Theorem \ref{thm:GaGmEVA}, the data of an associative vertex algebra is equivalent to a strong factorisation algebra in $(\El,\otimes^{ch})$ and by Proposition \ref{prop:GaGmEVABialgebra}, to give a compatible coproduct is equivalent to extending this to a bialgebra $A$ in the braided factorisation category $(\El,\otimes^{ch},\otimes)$. It remains to show that endowing $V$ with a vertex $R$-matrix is equivalent to endowing $A$ with a spectral op-$R$-matrix 
  $$R^*\ :\ A\otimes^{ch}A\ \to\ p^*1_\El.$$
  To begin with, given $R^*$ we restrict it to $(\Ab^1\times \Ab^1)_\circ$.  This gives an equivariant map of D-modules, and taking invariant sections gives
  $$R(z,w)\ :\ V\otimes V((z-w))\ \to\ k((z-w)).$$
  Condition (\ref{eqn:VertexQuantumGroupDerivation}) is then equivalent to such a map of vector spaces inducing a map of D-modules over $(\Ab^1\times\Ab^1)_\circ$. Then by Corollary \ref{cor:RanSpaceRMatrix}, after taking duals to get the analogous claim for op-$R$-matrices, this data over $(\Ab^1\times\Ab^1)_\circ$ determines a spectral op-$R$-matrix uniquely. 
\end{proof}

In the above, we have used the notation from section \ref{sec:CompletedDModules} for $\otimes^{ch}$ to refer to the \textit{completed} chiral factorisation structure. 

\subsubsection{} We expect multiplicative and elliptic analogues of vertex quantum groups and of Theorem \ref{thm:VertexQuantumGroupFactorisationQuantumGroup}, although as remarked after Theorem \ref{thm:GaGmEVA} these will be equivalent to vertex quantum groups. 

In Appendix \ref{sec:RationalMultiplicativeElliptic} we recall analogues of vertex algebraic structures living over one dimensional algebraic groups $G$ other than $\Ab^1$. We define analogues of vertex quantum groups, with spectral $R$-matrix $R(z,w)=R(z-_\gk w)\in V\otimes V((z-_\gk w))$ where $z,w$ are coordinates on the formal neighbourhood of the identity in $G$, and conditions (\ref{eqn:VertexQuantumGroupDerivation}) preserved and (\ref{eqn:SpectralRMatrixAlmostCommutativity}) and condition (\ref{eqn:VertexQuantumGroupHexagon})  replaced by 
\begin{equation}
  \begin{aligned}
   & R(u,v)\ =\ (Y(z-_\gk w)\otimes\id)R_{13}(z+_\gk u,v)R_{23}(w+_\gk u,v), \hspace{40mm}\\[3pt]
   & \hspace{20mm} R(u,v)\ =\ (\id\otimes\, Y(z-_\gk w))R_{13}(u,z+_\gk v)R_{12}(u,w+_\gk v).
  \end{aligned}
\end{equation}
For instance, when $G=\Gm$ the multiplicative vertex Yang-Baxter equation is 
$$R_{12}(s)R_{13}(st)R_{23}(t)\ =\ R_{23}(t)R_{13}(st)R_{12}(s)$$
where $z-w=\log(s)$ and $w-u=\log (t)$.

\subsubsection{Remark} One might wonder why we restrict ourselves to $\El=\QCoh^\wedge_{\Ran^{ch}(\Ab^1_{dR})/\Ga}$ or module categories inside it. The reason is that there is no analogue of Theorem \ref{thm:GaGmEVA} classifying factorisation categories over the Ran space as a dg-category along with a vertex algebra type structure map. The main obstactle is the lack of triangulated or stable $\infty$- structure on $\ShvCat$ which would be necessary for chiral Koszul duality as in Appendix \ref{sec:RationalMultiplicativeElliptic} to be true.

\subsubsection{Holomorphic vertex quantum groups} A vertex algebra is called \textit{holomorphic}\footnote{We warn readers that in the vertex algebra literature, the word ``commutative'' is unfortunately sometimes used instead of ``holomorphic''.} if it extends to a factorisation algebra on 
$$\Ran^{ch}(\Ab^1_{dR})/\Ga\ \to\ \Ran^* (\Ab^1_{dR})/\Ga.$$
It is well-known that a holomorphic commutative vertex algebra is equivalent to a commutative algebra with derivation $(A,T)$, where the field maps are 
\begin{equation}
  \label{eqn:HolomorphicVertexAlgebra} 
  Y(a,z)\ =\ (e^{zT}a)\cdot \ .
\end{equation}
The same is true for associative algebras and holomorphic (associative) vertex algebras. Likewise, for vertex quantum groups,
\begin{prop} There is an equivalence of categories between holomorphic vertex quantum groups, and bialgebras with biderivation and an op-$R$-matrix.
\end{prop}
\begin{proof}
  A bialgebra with biderivation $(A,m=\cdot,\Delta,T)$ is equivalent gives a holomorphic vertex bialgebra with vertex algebra structure (\ref{eqn:HolomorphicVertexAlgebra}) and coproduct $\Delta$. It is an associative vertex algebra because 
  \begin{align*}
    Y(Y(-,z)(-),w)&\ =\ m (e^{wT}\otimes \id) ( m\otimes\id)(e^{zT}\otimes \id \otimes\id )\\
    &\ =\ m  ( m\otimes\id)(e^{(z+w)T}\otimes e^{wT} \otimes\id )\\
    &\ =\ m  (\id\otimes m)(e^{(z+w)T}\otimes e^{wT} \otimes\id )\\ 
    &\ =\ Y(-,z+w)Y(-,w)
  \end{align*}
  because $T$ is a derivation and $m$ is associative, and the product is recovered by setting $z=0$. Given an op-$R$-matrix $R^*: A\otimes A\to k$, we get a vertex $R$-matrix 
  $$R(z,w)\ =\ R^*\cdot (e^{(z-w)T}\otimes \id )$$
  and we recover $R^*$ by setting $z-w=0$.  Weak commutativity is inherited from $m$.
\end{proof}

\subsection{Reconstruction and braided commutative vertex algebras} \label{sec:ReconstructionTheorem}

Recall that an algebra $A$ in a background braided monoidal category $\El$ is called \textit{braided commutative} if 
$$a\cdot b\ =\ b\cdot_\beta a,$$
that is the product of two elements is the same as their product after being swapped by the braiding $\beta$, in other words $m=m\cdot\beta$ where $m$ is the product. This is the same thing as an $\Eb_2$-algebra in $\El$. 

Likwise, we can define a \textit{braided commutative factorisation algebra} $A$ in background braided factorisation category $\Cl$ as a factorisation $\Eb_2$-algebra in $\Cl$. If $(Y,\Cl)$ is strong and classical, this is equivalent to the condition $m=m\cdot \beta$, where $m$ is the factorisation product and $\beta$ is the factorisation braiding. Now, let $H$ be a vertex quantum group with vertex $R$-matrix $R(z,w)$. Then

\begin{lemdefn} An $H$-linear vertex algebra $V$ is called (\textit{weakly}) \textit{braided commutative} or \textit{$R$-twisted} if
$$Y(\alpha,z)Y(\beta,w)\ =\ Y(R(w,z)\cdot \beta\otimes\alpha,w,z).$$
It is equivalent to a strong braided-commutative factorisation algebra inside $H\FactMd(\El)$.
\end{lemdefn}
\begin{proof}
  Follows from Proposition \ref{prop:GaGmEVABraidedCommutative}.  
\end{proof}

In the above, we have as usual written 
$$Y(-,z,w)(-)\ =\ Y(-,z)Y(-,w)(-)\ :\ V\otimes V\otimes V\ \to\ V((w))((z))$$
and recall our convention that $\El=\QCoh^\wedge_{\Ran (\Ab^1_{dR})/\Ga}$ is the factorisation category over the Ran space we take. If we impose additional assumtions on $R(z,w)$, we recover the notion of quantum vertex algebra due to Etingof-Kazhdan \cite{EK}, or deformed chiral algebra due to Frenkel-Reshetikhin \cite{FR}.

\subsubsection{Reconstruction Theorem} For the remainder of this section we sketch some of the theory of braided commutative vertex algebras, the commutative case being \cite{FBZ}.

First consider the case for ordinary algebras. We may identify any algebra $A$ as a subalgebra 
$$A\ \hookrightarrow\ \End(A)$$
by letting any $a\in A$ act by left multiplication. The Reconstruction Theorem (Proposition \ref{prop:ReconstructionBraidedCommutativeAlgebra}) says that we may identify (braided) commutative algebra structures on vector space $A$ with (braided) commutative subalgebras of $\End(A)$ satisfying a mild condition. This allows us to work with generators of $A$ without referring to the free algebra generated by them, and so this method generalises well to vertex algebras, where there is no free vertex algebra on a set of generators. We now recall the Reconstruction Theorem for commutative algebras:\footnote{It will be a special case of the Reconstruction Theorem for (braided commutative) vertex algebras, but we prove it in any case. }

\begin{prop} \label{prop:ReconstructionCommutativeAlgebra}
  To endow $(A,1)$ with the structure of a commutative algebra is equivalent to supplying commuting linear maps
  $$a_i\ :\ A\ \to\ A$$
  such that $A$ is spanned by the set of vectors of the form $a_{i_1}^{k_1}\cdots a_{i_r}^{k_r}1$. 
\end{prop}
\begin{proof}
  Given an algebra structure, let $\{a_i\}\subseteq \End(A)$ be endomorphisms given by an element of $A$. Conversely, given such data, form the smallest subalgebra generated by the $a_i$, 
  $$B\ =\ \langle\{a_i\}\rangle\ \subseteq\ \End(A)$$
  which is commutative. There is an isomorphism $B \simeq A$ sending $b\mapsto b\cdot 1$. It is surjective by the assumption on spanning, and injective because if $b\cdot 1=0$ then for any element $a$, which by spanning takes the form $a=b_a\cdot 1$, we have 
  \begin{equation}\label{eqn:UsingCommutativity}
    b\cdot a\ =\ bb_a\cdot 1\ =\ b_ab\cdot 1\ =\ 0
  \end{equation}
  and so $b=0$. 
\end{proof}

Notice that (\ref{eqn:UsingCommutativity}) is the reason we required commutativity. Indeed, the Proposition cannot work for merely associative algebras, because then any vector space $A$ would have a canonical algebra structure by dint of $\End(A)$ being an algebra. Moreoer, (\ref{eqn:UsingCommutativity}) is not even true for general braided commutative algebras unless $\El=H\Md(\Vect)$ is the category of modules for a quantum group with $R$-matrix $R=\sum R_1\otimes R_2$. In this case, we have 
$$b\cdot a\ =\ bb_a\cdot 1\ =\ \sum R_1 b_a R_2 b\cdot 1\ =\ 0$$
and so for $(A,1)$ a pointed object in $\El$,
\begin{prop} \label{prop:ReconstructionBraidedCommutativeAlgebra}
  To endow $(A,1)$ with the structure of a braided commutative algebra in $\El=H\Md$ is equivalent to supplying a subset
  $$S\ =\ \{a_i\}\ \subseteq\ \End(A)$$
  such that the $a_i$ braided commute,\footnote{To be explicit, $\Endl(A)$ is an algebra in $\El$, and we require that the elements $a_i:1_\El\to \Endl(A)$ braided commute.} and
  $A$ is spanned by the set of elements of the form $a_{i_1}^{k_1}\cdots a_{i_r}^{k_r}1$.
\end{prop}

\subsubsection{} \label{ssec:ReconstructionVertexAlgebras} We now turn to the Reconstruction Theorem for vertex algebras. Recall that if $V$ is a vector space, we say that two linear maps 
$$\alpha(z),\beta(z)\ :\ V\ \to\ V((z))$$
are \textit{mutually local} or \textit{weakly commute}, if we have equivalently
$$(z-w)^n\left(\alpha(z)\beta(w)- \beta(w)\alpha(z)\right)\ =\ 0 \hspace{5mm}\textup{for }n\gg 0, \hspace{15mm}\textup{or}\hspace{5mm}\alpha(z)\beta(w)\ =\ \beta(w)\alpha(z)$$
as elements of top left of the following squares, where the outer one is a pullback (intersection) and the inner one does not commute:
\begin{center}
\begin{tikzcd}[row sep = {30pt,between origins}, column sep = {25pt}]
  V[[z,w]][z^{-1},w^{-1}]\ar[rddd, bend right = 25]\ar[drrr,bend left=10]\ar[rd]&[-45pt]&[-30pt]&\\
& V[[z,w]][z^{-1},w^{-1},(z-w)^{-1}]\ar[rr,"\iota_{w/z}"]\ar[dd,"\iota_{z/w}"'] && V((z))((w))\ar[dd]\\[-15pt]
&&\times&\\[-15pt]
& V((w))((z))\ar[rr] && V[[z^{\pm 1},w^{\pm 1}]] 
\end{tikzcd}
\end{center}
where $\iota_{z/w}$ expands $1/(z-w)$ as a power series in $z/w$. Then 
\begin{prop}\emph{\cite[4.4.1]{FBZ}} Let $(V,|0\rangle)$ be a pointed vector space. To endow it with the structure of a (commutative) vertex algebra is equivalent to giving a collection 
$$\alpha_i(z)\ :\ V\ \to\ V((z))$$
of weakly commuting linear maps such that $V$ is spanned by the set of vectors $\alpha_{i_1,-n_1}^{k_1}\cdots \alpha_{i_r,-n_r}^{k_r}|0\rangle$, over all $n_j>0$.\footnote{
Here we have introduced the notation 
$$\alpha(z)\ =\ \sum_{n\in \Zb}\alpha_n z^{-n-1}.$$}
\end{prop}

The main result of this section will be an analogue of the above well-known result for braided factorisation categories 
$$\El\ =\ H\Md(\QCoh^\wedge_{\Ran(\Ab^1_{dR})/\Ga},\otimes)$$  where $H$ is a vertex quantum group with spectral $R$-matrix $R(z,w)$. To set some notation, if vector space $V$ is any $H$-module then we have a maps
$$\left(H\otimes H ((z))\right)\otimes \Hom(V,V((z)))^{\otimes 2}\ \to\ \Hom(V,V((z)))^{\otimes 2}\ \to\ \Hom(V,V((z))((w)))$$
where the first is given by extending the action map of $H$ by $k((z))$-linearity, and the second is composition.

\begin{defn}
  We say that two linear maps 
  $$\alpha(z),\beta(z)\ :\ V\ \to \ V((z))$$
  are \textit{mutually braided local} or \textit{weakly braided commute} if we have equivalently
  $$(z-w)^n\left(\alpha(z)\beta(w)- R(w,z)\beta(w)\alpha(z)\right)\ =\ 0 \hspace{5mm}\textup{for }n\gg 0,\hspace{45mm}$$
  $$ \hspace{45mm}\textup{or}\hspace{5mm}\alpha(z)\beta(w)\ =\ R(w,z)\beta(w)\alpha(z)$$
  where we have used the same notation as in section \ref{ssec:ReconstructionVertexAlgebras}. 
\end{defn}

Finally, we recall the binary operation on $\Hom(V,V((z)))$ called \textit{normally ordered product}, which is neither associative nor commutative. Set  
$$:\alpha_n\beta_m:\ = \ \begin{cases}
\alpha_n\beta_m &m>0 \\
\beta_m\alpha_n  & m\le 0 
\end{cases},$$
and extend by linearity to define $:\alpha(z)\beta(z):$. The crucial fact for us will be an $R$-twisted version of Dong's Lemma
\begin{lem} \label{lem:QuantumDong}
  Let $V$ be an $H$-module. Then if $\alpha(z),\beta(z),\gamma(z):V\to V((z))$ are pairwise mutually $R$-twisted local, then so are $:\alpha(z)\beta(z):$ and $\gamma(z)$.
\end{lem}
\begin{proof}
   Follows exactly as in the proof of Lemma 3.2 in \cite{Ka}, with $(-1)^{p(\alpha)p(\beta)}$ replaced by $R(z,w)=R(z-w)$.
\end{proof}

We now can state the Reconstruction Theorem for $R$-braided commutative vertex algebras: 
\begin{prop} 
  Let $(V,|0\rangle)$ be a pointed element of $H\Md$. To endow it with the structure of a weakly braided commutative vertex algebra in $H\Md$ is equivalent to giving a collection of $H$-linear maps
  $$\alpha_i(z)\ :\ V\ \to\ V((z))$$
  which weakly braided commute and such that $V$ is spanned by the set of vectors $\alpha_{i_1,-n_1}^{k_1}\cdots \alpha_{i_r,-n_r}^{k_r}|0\rangle$, over all $n_j>0$.
\end{prop}
\begin{proof}
  Given a braided commutative vertex algebra structure, we use the set of $H$-linear maps $Y(v_i,z)$ for $v_i$ ranging over all of $V$. Conversely, given such data we define 
  $$Y(\alpha_{i_1,-n_{i_1}}\cdots \alpha_{i_k,-n_{i_k}}|0\rangle,z)\ =\ \frac{1}{(n_{i_1}-1)!\cdots (n_{i_k}-1)! }:\partial_z^{n_{i_1}-1}\alpha_{i_1}(z)\cdots \partial_z^{n_{i_k}-1}\alpha_{i_k}(z):.$$
  The proof of the Reconstruction Theorem \cite[4.4.1]{FBZ} for (commutative) vertex algebras now applies verbatim, with Dong's Lemma replaced by Lemma \ref{lem:QuantumDong}.
\end{proof}

\subsubsection{} Finally we consider the Reconstruction Theorem for factorisation algebras. Let $\El$ be a classical symmetric factorisation category over $Y$ admitting an internal hom functor 
$$\Homl\ :\ q^*(\El\boxtimes \El)\ \to \ p^*\El$$
satisying tensor-hom adjunction with respect to $\otimes_\El$. Note that $\otimes^{ch}$ does not admit an internal hom functor.  

Given any $F\in \Gamma(Y,\El)$, we define the vector space of \textit{sections} of $F$ to be
$$\Gamma(Y,F)\ =\ \Hom_{\Gamma(Y,\El)}(1_\El,F).$$
Since $\Gamma(\El)$ is a Grothendieck abelian category, we may define images of any map and subobjects, and hence the \textit{span} $\langle S\rangle\subseteq F$ of any collection of sections $S$ as the smallest subobject containing their images. We call $F$ \textit{globally generated} if it is spanned by its sections.\footnote{This agrees with the usual terminology if $F$ is a quasicoherent sheaf.}

\begin{lem} \label{lem:GlobalGenerationInjective}
  If $F$ is globally generated, then for any $G$ the tautological map 
  $$\Hom_{\Gamma(\El)}(F,G)\ \hookrightarrow\ \Hom(\Gamma(F),\Gamma(G))$$
  is injective.  
\end{lem}

Let $(A,1)$ be a pointed element of $\Gamma(Y,\El)$.
\begin{prop} \label{prop:ReconstructionFactorisationAlgebras}
   If $A$ is globally generated, to endow $(A,1_A)$ with the structure of a commutative factorisation algebra is equivalent to supplying a commuting set of maps in $\Gamma(Y,\El)$
   $$a_i\ :\ A\ \to\ A$$
   such that $A$ is spanned by the set of sections of the form $a_{i_1}^{k_1}\cdots a_{i_r}^{k_r}1$.
\end{prop}
\begin{proof}
  Given a commutative factortisation algebra structure, consider the map 
  $$A\ \to\ \Endl(A)$$
  adjoint to the multiplication map. Let $a_i$ run over all sections of $A$, which induce commuting sections $\Gamma(\Endl(A))=\End(A)$. Conversely, given such data we form the smallest factorisation subalgebra spanned by the $a_i$:
  $$B\ =\ \langle\{a_i\}\rangle \ \hookrightarrow\ \Endl(A)$$
  which is a commutative factorisation algebra. To finish, it is enough to give an isomorphism $B\simeq A$ as objects of $\El$. Since by definition both objects are generated by their sections, it is enough to give an isomorphism $\Gamma(B)\simeq \Gamma(A)$, which is supplied by sending $b\mapsto b\cdot 1_A$. This is surjective by assumption. To show injectivity, if  $b\cdot 1_A=0$ then as by the spanning assumption every section $a\in \Gamma(A)$ takes the form $a=b_a\cdot 1_A$, we have
  $$b\cdot a\ =\ bb_a\cdot 1_A\ =\ b_ab\cdot 1_A\ =\ 0$$
  hence by Lemma \ref{lem:GlobalGenerationInjective} we have  $b=0$. 
\end{proof}

Analogues for braided commutative globally generated factorisation algebras may also be proven. We do not discuss the relation to Reconstruction for vertex algebras, but expect that it can be derived from Proposition \ref{prop:ReconstructionFactorisationAlgebras}.

\subsection{Completions} \label{sec:CompletedDModules} We now deal with a technical issue. Naively, we might have defined a vertex quantum group to simply be a factorisation quantum group inside $\QCoh_{Y}$, but in that case the spectral $R$-matrices would lie inside
$$R\ \in\ V\otimes V[(z-w)^{\pm 1}]$$
and likewise for spectral op-$R$-matrices. However, most examples we care about it will in fact live in the completion 
$$R\ \in\ V\otimes V((z-w)).$$
Thus, we need to somehow ``complete'' the factorisation category $\QCoh_Y$, but leaving intact the fact that factorisation coalgebras are vertex algebras. 

\subsubsection{} 

To do this, for any scheme $X$ and nonempty finite set $I$ we consider the closed subscheme
$$Z_I\ =\ \underset{\substack{I\twoheadrightarrow J\\ {|I|>|J|}}}{\bigsqcup}X^J\ \subseteq\ X^I$$
given by the union of all proper diagonal subspaces of $X^I$. 

\begin{lem}
  If $X$ is smooth and $\Ml$ is a D-module on $X^I$ we may endow its completion as a quasicoherent sheaf with a D-module structure, giving is a commuting diagram
  \begin{center}
  \begin{tikzcd}[row sep = 20pt, column sep = 20pt]
    \Dl\Md(X^I)\ar[d,"\oblv"]\ar[r,"(-)^\wedge_{Z_I}",dashed] &[5pt]\Dl\Md(X^I)\ar[d,"\oblv"] \\ 
  \QCoh(X^I)\ar[r,"(-)^\wedge_{Z_I}"] & \QCoh(X^I)
  \end{tikzcd}
  \end{center}
These diagrams are compatible as we vary over $I$. 
\end{lem}
\begin{proof}
  To begin, there is a filtration $F^\sbt \Ol_{X^I}$ given by the order of vanishing along $Z_I$, i.e. consider the order of vanishing along $X^J\subseteq Z_I\subseteq X^I$ for all $J$ with $|J|=|I|-1$ then take the maximum. In particular, if $M$ is a D-module this induces a filtration $F^\sbt M=M\otimes F^\sbt \Ol_{X^I}$ as a quasicoherent sheaf, where the action by vector fields sends
  $$\Tl_{X^I}\otimes F^\sbt M\ \to\ F^{\sbt-1}M.$$
  In particular, as vector fields act on the completion
  $$M^\wedge_{Z_I}\ =\ \lim F^\sbt M/F^{\sbt+1}M$$
  and hence we get an action of $\Dl_{X^I}$ on $M^\wedge_{Z_I}$. 
\end{proof}
It follows that the completions along the $Z_I$ arrange into a functor of sheaves of categories
  $$(-)^\wedge \ :\ \QCoh_{\Ran X_{dR}}\ \to\ \QCoh_{\Ran X_{dR}},$$
and we denote its essential image by $\QCoh^\wedge_{\Ran X_{dR}}$. Its sections consist of D-modules on $\Ran X$ whose restriction to $X^I$ are complete along $Z_I$. 

\subsubsection{} We now define the \textit{completed} chiral factorisation structure on $\QCoh_{\Ran X_{dR}}$. It is induced by the functors
$$\Dl\Md(X^{I_1})\otimes \Dl\Md(X^{I_2}) \ \to\ \Dl\Md(X^{I_1\sqcup I_2})$$
sending
$$M\otimes N\ \mapsto \jmath_*\jmath^*(M\hat{\boxtimes}N)\ \defeq\ \jmath_*\jmath^*(M\boxtimes N)^\wedge_{Z_{I_1\sqcup I_2}}.$$
It is easily checked that this gives a symmetric factorisation structure, and that 
\begin{lem}
  $\hat{\otimes}^{ch}$ restricts to a symmetric factorisation structure on $\QCoh^\wedge_{\Ran X_{dR}}$. 
\end{lem}
In particular, it follows that  
$$\QCoh_{\Ran (X_{dR})}^{\otimes^{ch}}\ \stackrel{(-)^\wedge}{\to}\ \QCoh_{\Ran (X_{dR})}^{\wedge,\hat{\otimes}^{ch}}\ \to\ \QCoh_{\Ran (X_{dR})}^{\hat{\otimes}^{ch}}$$
are maps of symmetric factorisation categories.

\subsubsection{Equivariant version}  
If algebraic group $G$ acts on $X$, then $Z^I\subseteq X^I$ is an equivariant subspace for the diagonal action of $G$. Likewise, we have commuting diagram
\begin{center}
  \begin{tikzcd}[row sep = 20pt, column sep = 20pt]
    \Dl\Md(X^I)^{G,w}\ar[d,"\oblv"]\ar[r,dashed,"(-)^\wedge_{Z_I}"] &[5pt]\Dl\Md(X^I)^{G,w}\ar[d,"\oblv"] \\ 
  \QCoh(X^I)^G\ar[r,"(-)^\wedge_{Z_I}"] & \QCoh(X^I)^G
\end{tikzcd}
\end{center}
for the diagonal action of $G$ on $X^I$. We as before denote by $\QCoh_{\Ran(X_{dR})/G}^\wedge$ the essential image of 
  $$(-)^\wedge \ :\ \QCoh_{\Ran(X_{dR})/G}\ \to\ \QCoh_{\Ran(X_{dR})/G}$$
and again we have maps of symmetric factorisation categories
$$\QCoh_{\Ran (X_{dR})/G}^{\otimes^{ch}}\ \stackrel{(-)^\wedge}{\to}\ \QCoh_{\Ran (X_{dR})/G}^{\wedge,\hat{\otimes}^{ch}}\ \to\ \QCoh_{\Ran (X_{dR})/G}^{\hat{\otimes}^{ch}}$$
are maps of symmetric factorisation categories. Moreover, if $X=G$ is a one dimensional algebraic group over the complex numbers, we have the following variant of Theorem \ref{thm:GaGmEVA} and Proposition \ref{prop:GmGaEAssociative}:
\begin{prop} \label{prop:Completion}
  The functor $(-)^\wedge$ induces an equivalence on categories of strong factorisation algebras
  $$ \Eb_1\FactAg^{\textup{st}}(\Dl\Md(\Ran \Ab^1)^{\Ga,w})\ \stackrel{\sim}{\to}\ \Eb_1\FactAg^{\textup{st}}(\Dl\Md(\Ran \Ab^1)^{\Ga,w,\wedge}),$$
  and likewise for strong commutative factorisataion algebras.
\end{prop}

\newpage

\section{Borcherds Twists} \label{sec:BorcherdsTwists}

\noindent
In this section, we show how to produce braided commutative bialgebras from (factorisation) quantum groups, using the Borcherds twist construction, which was originally used in \cite{Bo1} to construct lattice vertex algebras. These structure on these twists should be viewed as analogous to positive half quantum groups $U_q(N)$, which are a braided cocommutative bialgebra inside $\Rep_q T$, as in \cite[$\S$ 3.2]{Ga2}.

\subsection{Borcherds bialgebras} Let $A$ be a quantum group with $R$-matrix $R$ in a background symmetric monoidal category $(\El,\sigma)$, as in section \ref{sec:BraidedMonoidalCats}. 

\begin{lemdefn} \label{lemdefn:BorcherdsTwist}
The \emph{Borcherds twist} $A_R$ is the coalgebra on the same object $A$ and with coproduct
$$\Delta_R\ =\ R\cdot \Delta.$$
\end{lemdefn}
\begin{proof}
To show that $\Delta_R$ is coassociative, use coassociativity of $\Delta$ and the Yang-Baxter equation:   
\begin{align*}
(\Delta_R\otimes \id)\Delta_R\ = \ R_{12}(\Delta\otimes \id)R\Delta&\ =\ R_{12}R_{13}R_{23}(\Delta\otimes\id)\Delta\\
&\ =\ R_{23}R_{13}R_{12}(\id\otimes \Delta)\Delta\ =\ R_{23}(\id\otimes\Delta)R\Delta\ =\ (\id\otimes\Delta_R)\Delta_R.
\end{align*}
The counit is inherited from $A$ because $(\epsilon\otimes\epsilon)(R)=1$.
\end{proof}

Note that $A_R$ is a coalgebra internal to $\textup{Mod-}A$.

In the following, we say that $R$ is (\textit{naive-})\textit{symmetric} if 
\begin{equation}
  \label{eqn:NaiveSymmetric}
  R\ =\ \sigma(R). 
\end{equation}
In particular, this implies that $R/\sigma(R)$ is fixed by $\sigma(-)^{-1}$, a condition which appears in the definition of symmetric $R$-matrix. In the following, we view $\textup{Mod-}A$ as a braided monoidal category with respect to the braiding $\beta=\sigma(-)\cdot R^{-1}$ given using right multiplication by the inverse of the $R$-matrix.

\begin{prop}\label{prop:SymmetricBorcherdsTwist}
  If $R$ is naive-symmetric then 
  $$A_R\ \in\ \Eb_2\CoAg(\textup{Mod-}A)$$
  defines a braided commutative coalgebra. If in addition $A$ is commutative, then $A_R$ is a cocommutative.
   \end{prop}
  \begin{proof}
    Since $A$ is a quantum group, it follows that 
    \begin{equation}\label{eqn:AlmostCommTwist}
    \sigma(\Delta_R)\ =\ \sigma(R\Delta)\ =\ \sigma(R) R \Delta R^{-1}\ =\ \sigma(R)\Delta_R R^{-1}.
   \end{equation}  
   Thus if $R$ is naive-symmetric, the result follows and if $A$ is commutative then $\sigma(R)\Delta_R R^{-1}=\Delta_R$.
  \end{proof}

\subsubsection{} The Borcherds twist $A_R$ is usually not a bialgebra with respect to the original product on $A$. However, in examples it often \textit{does} extend to a bialgebra $A_R=(A,m,\Delta_R)$. First we axiomatise this structure before considering examples in section \ref{sec:Examples}.

Let  $A=(A,\cdot,\Delta,R)$ be a commutative quantum group with with $R$-matrix.

\begin{defn} An algebra structure $m$ on $A$ is called a \textit{Borcherds product} if it is $A$-linear, i.e. 
  \begin{equation}\label{eqn:BorcherdsLinearity}
    m(\Delta(a)\cdot a')\ = \ a\cdot m(a')
  \end{equation}
  for all $a,a'\in A$, and the following so-called \textit{excess intersection} diagram commutes
  \begin{equation}\label{fig:ExcessIntersectionDiagram}
  \begin{tikzcd}[row sep = 15pt, column sep = 20pt]
  p^*A\otimes p^*A\ar[dd,"m"]\ar[r,"\Delta\otimes\Delta"] &A^{\otimes 2}\otimes A^{\otimes 2}\ar[r,"\sigma_{23}"] &A^{\otimes 2}\otimes A^{\otimes 2}\ar[d,"R_{14}^{-1}\sigma(R_{23})^{-1}"]  \\
  & & A^{\otimes 2}\otimes A^{\otimes 2}\ar[d,"m\otimes m"]\\
  A\ar[rr,"\Delta"] &  & A\otimes A
  \end{tikzcd}
  \end{equation}
  where in $\sigma_{23}$ is induced by the background symmetric monoidal structure. 
  \end{defn}

  If $R=1\otimes 1$ then (\ref{fig:ExcessIntersectionDiagram}) says that $(A,m,\Delta)$ forms a bialgebra. More generally,
  \begin{prop} \label{prop:BorcherdsBialgebra}
    A Borcherds-twisted coalgebra $A_R$ forms a bialgebra in $\textup{Mod-}A$ with respect to any Borcherds product $m$, where $\textup{Mod-}A$ has braided monoidal structure 
    $$M\otimes N\ \mapsto\ \sigma(M\otimes N) \cdot \sigma(R)R^{-1}.$$
    Moreover, the coalgebra structure on $A_R$ is braided cocommutative.
  \end{prop}
\begin{proof}
  Follows from sewing (\ref{fig:ExcessIntersectionDiagram}) to the diagram induced by (\ref{eqn:BorcherdsLinearity}):
\begin{center}
\begin{tikzcd}[row sep = 15pt, column sep = 65pt]
A^{\otimes 2}\otimes A^{\otimes 2}\ar[d,"m\otimes m"]\ar[r,"R_{13}R_{14}R_{23}R_{24}"] &A^{\otimes 2}\otimes A^{\otimes 2}\ar[d,"m\otimes m"] \\
A\otimes A\ar[r,"R"] & A\otimes A
\end{tikzcd}
\end{center}
We are left with the diagram 
\begin{center}
  \begin{tikzcd}[row sep = 15pt, column sep = 20pt]
    p^*A\otimes p^*A\ar[d,"m"]\ar[r,"R\cdot\Delta\otimes R\cdot \Delta"] &[15pt]A^{\otimes 2}\otimes A^{\otimes 2}\ar[r,"\sigma_{23}"] &A^{\otimes 2}\otimes A^{\otimes 2}\ar[r,"R_{23}\sigma(R_{23})^{-1}"] &[35pt]A^{\otimes 2}\otimes A^{\otimes 2}\ar[d,"m\otimes m"]  \\
    A\ar[rrr,"R\cdot \Delta"] &  && A\otimes A
    \end{tikzcd}
\end{center}
It follows that $(A,\Delta_R,m)$ forms a bialgebra inside the braided monoidal category $\textup{Mod-}A$, where $R_{23}\sigma(R_{23})^{-1}$ is induced by the braiding. If $R$ is naive-symmetric then $\Delta_R$ is almost commutative, 
\end{proof}

\subsubsection{Remark} It is natural to expect that there is a systematic way of constructing Borcherds products $m$ making $A_R$ into a bialgebra. For instance, this process will take in the data of a moduli stack $\Ml_\Cl$ and produce the correspondence $\SES$.

\subsection{Geometric Borcherds twists and Euler classes} We can take the background symmetric monoidal category $\Al=\PreStk^{\textup{corr}}$ to consist of prestacks. In this section we will discuss these ``space level'' $R$-matrices and their Borcherds twists.

\subsubsection{Linear lax $R$-matrices} \label{sec:LinearRMatrices} Let $Y$ be a prestack with a unital monoid structure
$$a\ :\ Y\times Y\ \to\ Y.$$
Then if $\theta\in \Perf(Y\times Y)$ is a multiplicative perfect complex, meaning $(a\times\id)^*\theta=\theta_{13}\oplus\theta_{23}$ and $(\id\times a)^*\theta=\theta_{12}\oplus\theta_{13}$, as well as compatibility with the unit $\theta\vert_{1\times Y}=\theta\vert_{Y\times 1}=0$, then if $a$ is abelian its total space projecting to the base in both directions:
\begin{center}
 \begin{tikzcd}[row sep = 10pt, column sep = 15pt]
 & \theta\ar[rd]\ar[ld] & \\
Y\times Y & & Y\times Y
 \end{tikzcd}
 \end{center} 
defines a Yang-Baxter matrix $R_\theta:Y\times Y\to Y\times Y$ with respect to the coproduct
 \begin{center}
 \begin{tikzcd}[row sep = 10pt, column sep = 15pt]
 & Y\times Y\ar[rd,equals]\ar[ld,swap,"a"] & \\
 Y & & Y\times Y
 \end{tikzcd}
 \end{center} 
These Yang-Baxter matrices compose as $R_\theta\cdot R_{\theta'}= R_{\theta\oplus\theta'}$. In particular, $R_{\theta}$ is usually not invertible, but it does have some sort of ``weak inverse": $R_{\theta}\cdot R_{\theta[1]}$ is a degeneration of trivial correspondence. Indeed, it is the fibre above the origin of a flat family of correspondences $R_{\theta \stackrel{t\id}{\to}\theta}$ over $\Ab^1_t$ whose fibres away from the origin are $R_0=Y\times Y$.

\subsubsection{} \label{sec:SpaceRMatrix} More generally, if $Y$ is as in section \ref{sec:LinearRMatrices} we may ask what does it mean for $R\in Y\times Y$, i.e. a map $R:1\to Y\times Y$,
\begin{center}
  \begin{tikzcd}[row sep = 10pt, column sep = 15pt]
  &[5pt] R\ar[rd,"f"]\ar[ld] & \\
\pt & & Y\times Y
\end{tikzcd}
\end{center}
to define a (lax) $R$-matrix. We will explain the conditions  in terms of multiplication $R\cdot$ by $R$ with respect to the diagonal product on $Y$, which one can show defines the correspondence
\begin{center}
  \begin{tikzcd}[row sep = 10pt, column sep = 15pt]
  & R\ar[rd,"f"]\ar[ld,"f"'] & \\
 Y\times Y & & Y\times Y
  \end{tikzcd}
  \end{center} 
\begin{prop}
$R$ defines a lax $R$-matrix if and only if we have identities for $(a\times \id)^*R$ and $(\id\times a)^*R$: 
\begin{equation}\label{eqn:LinearRMatrixPullback}
  (Y^2\times Y)\times_{Y\times Y}R \ =\ R_{13}\times_{Y\times Y}R_{23},\hspace{15mm} (Y\times Y^2)\times_{Y\times Y}R\ =\ R_{12}\times_{Y\times Y}R_{13}
\end{equation}
as correspondences from $Y^3$ to $Y^2$, and as maps of prestacks 
\begin{equation}\label{eqn:LinearRMatrixQuasitriangularity}
  f \cdot a \ = \ f\cdot \sigma\cdot a,
\end{equation}
which is automatically true if $a$ is abelian. 
\end{prop}

 Two such $R$-matrices $R,R'$ compose as 
 \begin{equation}\label{eqn:LinearRMatrixComposition}
  R\cdot R'\ =\ R\times_{Y\times Y} R'.
 \end{equation}
 In the right hand sides of the equalities in (\ref{eqn:LinearRMatrixPullback}) and (\ref{eqn:LinearRMatrixComposition}), we are taking the fibre products over the right hand copy of $Y\times Y$.

 The Borcherds twist of $(Y,a,\check{\Delta})_R$ is the cofactorisation space
\begin{center}
  \begin{tikzcd}[row sep = 10pt, column sep = 15pt]
  & Y\times_{Y\times Y}R\ar[rd]\ar[ld] & \\
 Y & & Y\times Y
  \end{tikzcd}
  \end{center} 

\subsubsection{} More generally, if $Y$ is a factorisation space and $R$ is as in section \ref{sec:SpaceRMatrix}, then 
\begin{prop}\label{prop:SpaceRMatrix}
$R$ defines a lax $R$-matrix if and only if 
\begin{equation}
  (C_Y\times Y)\times_{Y\times Y}R \ =\ R_{13}\times_{Y\times Y}R_{23},\hspace{15mm} (Y\times C_Y)\times_{Y\times Y}R\ =\ R_{12}\times_{Y\times Y}R_{13}
\end{equation}
as correspondences from $Y^3$ to $Y^2$, and
\begin{equation}
  R\times_{Y\times Y,f}C_Y\ =\ R\times_{Y\times Y,\sigma f}C_Y 
\end{equation}
as correspondences from $Y^2$ to $Y$, where we have taken the fibre product of $C_Y\to Y^2$ and $f:R\to Y^2$ or $\sigma f:R\to Y^2$, respectively. 
\end{prop}

\subsubsection{Euler classes} Given a (lax) $R$-matrix space as in section \ref{sec:SpaceRMatrix}, temporarily assume that $f$ and $a$ are smooth, so we can take the relative Euler class
$$e(\Tb_f)\ \in\ \Ht^\sbt(R).$$
In what follows, we will assume that this is the pullback of a class $e(\Tb_f)=f^*e_R$ of a class $e_R\in \Ht^\sbt(Y\times Y)$ and that $f^*$ is injective, which is true at least in the linear case of section \ref{sec:LinearRMatrices}. Second, we assume that $e(\Tb_a)=1$.

If $a$ is commutative,

\begin{lem}
 $e_R$ defines a lax $R$-matrix for the bialgebra $(\Ht^\sbt(Y),\cup, a^*)$. 
\end{lem}
\begin{proof}
  To prove $(a\times \id)^*e_R=e_{R,12}e_{R,23}$, we take diagram (\ref{eqn:LinearRMatrixPullback}) consisting of a cartesian square:
  \begin{center}
    \begin{tikzcd}[row sep = {30pt, between origins}, column sep = {55pt, between origins}]
   & &R_{13}\times_{Y\times Y}R_{23}\ar[ld,swap]\ar[rd]&&\\ 
   &\ar[ld,equals] Y^2\times Y\ar[rd,"a\times\id"]& & R\ar[ld,"f"']\ar[rd,"f"] \\
   Y\times Y&  & Y^2\times Y& & Y\times Y
    \end{tikzcd}
    \end{center} 
We apply the distinguished triangle for tangent complexes of compositions to compute $e(\Tb_{R_{13}\times_{Y\times Y}R_{23}/Y\times Y})$ in two different ways. First, it is the pullback of $e_{R,12}e_{R,23}$. Second, is the the product of the pullbacks of $e(\Tb_f)$ and $e(\Tb_{a\times\id})=1$. The identity $(\id\times a)^*e_R=e_{R,12}e_{R,13}$ is proven similarly. 

Finally, almost cocommutativity follows because $\Ht^\sbt(Y)$ is commutative and cocommutative. 
\end{proof}

\subsubsection{} 
We now relax the assumptions that $f$ and $a$ are smooth to $\Tb_f=(\Tb_f^i)$ and $\Tb_a=(\Tb_a^i)$ being globally complexes of vector bundles. Work in the category $\PreStk^{\textup{corr}}_{/B}$ of prestacks with correspondences over base space $B$, and take a multiplicative subset 
$$S\ \subseteq\ \Ht^\sbt(B)$$
such that $e(\Tb_f^i),e(\Tb_a^i)$  are invertible in the localised cohomology groups $\Ht^\sbt(-)_{\loc}=\Ht^\sbt(-)[S^{-1}]$ for odd $i$.  We can then define $e(\Tb_f)$ as the alternating product of the $e(\Tb_f^i)$, likewise $e(\Tb_a)$, and since this is also compatible with pullback and is multiplicative on distinguished triangles, we have as before

\begin{prop}
  If $a$ is commutative, $e_R$ defines a lax $R$-matrix in $\Ht^\sbt(Y\times Y)_{\loc}$.
\end{prop}

\subsection{Cohomological Hall algebras and localised coproducts} \label{sec:CohaBialg} Consider a lax bialgebra $Y$ in $\PreStk^{\textup{corr}}$, whose coproduct is given by a commutative monoid structure
$$a\ :\ Y\times Y\ \to\ Y$$
Suggestively denote the algebra structure on $Y$ by
\begin{center}
\begin{tikzcd}[row sep = 10pt, column sep = 15pt]
 &\SES\ar[rd,"p"]\ar[ld,"q",swap] & \\
Y\times Y & & Y
\end{tikzcd}
\end{center}
The lax compatibility between the product and coproduct on $Y$ is the existence of a map $i$ in 
\begin{equation}\label{fig:CohaDiagram}
\begin{tikzcd}[row sep = 10pt, column sep = 15pt]
&&\SES^2\arrow[d,"i"]\arrow[rddd,bend right = -30,"p\times p"]\arrow[lldd,bend right = 20,swap,"q\times q"]&&\\[5pt]
&&\SES \times_{Y} Y^2\arrow[rdd]\arrow[ldd,swap]&&\\[-10pt]
Y^2\times Y^2\ar[dd,"a\times a\cdot \sigma_{23}",swap]&&&&\\[-20pt]
&\SES\arrow[ld,swap]\arrow[rd]&& Y^2\arrow[ld,"a",swap]\\
\textcolor{white}{a}Y\times Y\textcolor{white}{a}&&Y&
\end{tikzcd}
\end{equation}
compatible with the (co)associativity isomorphisms. $Y$ is an honest bialgebra if $i$ is in addition an isomorphism. The above diagram is just 
\begin{center}
  \begin{tikzcd}[row sep = 20pt, column sep = 45pt]
   Y\times Y\ar[r,"\SES"]\ar[d,"\sigma_{23}\cdot (a\times a)"',""{name=U,inner sep=2pt,xshift=20pt}] &Y\ar[d,"a",""{name=D,inner sep=2pt,xshift=-20pt}] \\
   Y^2\times Y^2 \ar[r,"\SES\times \SES"]&Y^2
   \arrow[Rightarrow, from=U, to=D, "i", shorten <= 5pt, shorten >= 5pt] 
  \end{tikzcd}
  \end{center}
Of course, the main example of the above is when $Y=\Ml_\Cl$ is the moduli stack of objects in a dg or abelian category.

\subsubsection{} In the following we will assume that $\SES$ and $a$ are compatible in addition by the pullbacks on cohomology being
$$p^*\ =\ q^*a^*,$$
and as before we also assume that $q^*$ is injective. For instance, this is the case when $Y=\Ml_\Cl$ is a moduli stack. 

\subsubsection{} 

Assume $p$ is proper and for simplicity that all spaces in (\ref{fig:CohaDiagram}) and $p$ are smooth. We then get an algebra structure $m=p_*q^*$ on cohomology $\Ht^{\sbt}(Y)$, in addition to the bialgebra structure $A=(\Ht^\sbt(Y),\cup,a^*)$ induced by the diagonal map and $a$. Assume that $i$ is a closed embedding, so that 
$$i^*i_*\ =\ e(\Nb_i)\cdot$$
where $e(\Nb_i)$ denotes the euler class of the normal complex $\Nb_i$, which is a vector bundle since all relevant spaces are smooth.

In the following, we will write $e_R=e(\Tb_p)$ for the unique element of $\Ht^\sbt(Y\times Y)$ pulling back to $e(\Tb_p)\in \Ht^\sbt(\SES)$.

\begin{lem}\label{lem:eNi}
  $e(\Nb_i)=(q\times q)^*e_{R,14}e_{R,32}$ as elements of $\Ht^\sbt(\SES^2)$. 
  \end{lem}
\begin{proof}
  Apply the distinguished triangle for tangent complexes to 
\begin{center}
  \begin{tikzcd}[row sep = 10pt, column sep = 15pt]
    &&\SES^2\arrow[d,"i"]\arrow[rddd,bend right = -30,"p\times p"]&&\\[5pt]
    &&\SES \times_{Y} Y^2\arrow[rdd]&&\\[-10pt]
    &&&&\\
    &&& Y^2
    \end{tikzcd} 
\end{center}
to give that $e(\Tb_{p\times p})=e(\Tb_i)e(\tilde{a}^*\Tb_p)$, where $\tilde{a}:\SES^2\to \SES$ is the map in (\ref{fig:CohaDiagram}). We then compute
$$e(\Tb_{p\times p})\ =\ (q\times q)^*e_{R,12}e_{R,34}$$
and
$$e(\tilde{a}^*\Tb_p)\ =\ \tilde{a}^*q^*e_R\ =\ (q\times q)^*\sigma_{23}^*(a\times a)^*e_R\ =\ (q\times q)^*e_{R,12}e_{R,14}e_{R,32}e_{R,34}$$
where in the last equality we used that $e_R$ was compatible with $a$. This finishes the proof, since by assumption $q^*$ was injective. 
\end{proof}

Finally, assume that we are in a situation such that the maps
$$\Ht^\sbt(\SES\times_Y Y^2)\ \stackrel{i_*}{\to}\ \Ht^\sbt(\SES^2)\ \stackrel{i^*}{\to}\ \Ht^\sbt(\SES\times_Y Y^2)$$
are both isomorphisms. This is often the case for instance if we are in the torus-equivariant setting and replace cohomology with localised cohomology. In particular, 
\begin{equation}
  \label{eqn:TorusLocalisation}
  \id\ =\ i_*\frac{i^*(-)}{e(\Nb_i)}. 
\end{equation}

\begin{prop} The cohomological Hall algebra product and localised coproduct form a bialgebra
$$(\Ht^\sbt(Y),p_*q^*,e(\Tb_p)\cdot a^*)$$
in the braided monoidal category of modules over the algebra $(\Ht^\sbt(Y),\cup,a^*)$, with braiding induced by the $R$-matrix $e(\Tb_p)\cdot \sigma^*(e(\Tb_p))^{-1}$.
\end{prop}
\begin{proof} It is enough to show that $p_*q^*$ defines a Borcherds product for $(\Ht^\sbt(Y),\cup,a^*,e(\Tb_p))$. We have two different maps $\Ht^\sbt(Y\times Y)\to \Ht^\sbt(Y^2)$, by pushing and pulling along the outer arrows of (\ref{fig:CohaDiagram}). Then applying (\ref{eqn:TorusLocalisation}) the diagram
  \begin{center}
    \begin{tikzcd}[row sep = 15pt, column sep = 20pt]
      \Ht^\sbt(Y\times Y)\ar[dd,"p_*q^*"]\ar[r,"(a\times a)^*"] &\Ht^\sbt(Y^2\times Y^2)\ar[r,"\sigma_{23}"] &\Ht^\sbt(Y^2\times Y^2) \ar[dd,"(p\times p)_*\frac{(q\times q)^*}{e(\Nb_i)}"]  \\
      & & \\
      \Ht^\sbt(Y)\ar[rr,"a^*"] &  & \Ht^\sbt(Y\times Y)
      \end{tikzcd} 
  \end{center}
  commutes, and by Lemma \ref{lem:eNi} if we take $R=e(\Tb_p)$ this is the excess intersection diagram (\ref{fig:ExcessIntersectionDiagram}). Finally, the linearity condition holds by the projection formula and the fact that $p^*=q^*a^*$.
\end{proof}

The simplying assumption that $p$ and all spaces in (\ref{fig:CohaDiagram}) are smooth is very restrictive, for instance when $Y=\Ml_{\Al}$ is the moduli stack of objects in an abelian category, we only expect this is only true when $\Al$ has dimension zero. It is possible to relax this assumption, but one needs to carefully consider how to define $e(\Nb_i)$ and $i_*$--for technical details are worked out for dimension one categories, see \cite{La}.

\subsection{Factorisation Borcherds twists}  \label{sec:FactorisationBorcherdsTwists} 

Let $A$ be a factorisation quantum group in background lax classical braided factorisation category $(\El,\sigma)$ on classical braided factorisation space $Y$, and with spectral $R$-matrix $R\in \Gamma(C,A\otimes_\El A)$ a spectral $R$-matrix. As before, we can define from this a factorisation coalgebra

\begin{defn}
The \textit{factorisation Borcherds twist} $A_R$ is the factorisation coalgebra with underlying section $A$ and factorisation coproduct $\Delta_R=R\cdot\Delta$.
\end{defn} 

In the vertex algebra case, this reproduces the Borcherds twists in \cite{Bo2}. We have as before 
\begin{lem}
$A_R$ is a braided commutative factorisation bialgebra in $\textup{Mod-}A$ with respect to the classical factorisation braiding induced by multiplication by $R^{-1}$ so long as $R$ is \emph{(naive-)symmetric} in the sense that $R=\sigma(R)$. 
\end{lem}

If $A$ is a factorisation quantum group, we define 
\begin{defn} A factorisation algebra structure $m:A\otimes_\El' A\to p^*A$ on the section underlying $A$ is called a \textit{factorisation Borcherds product} if it is $A$-linear and the following so-called \textit{excess intersection} diagram
  \begin{equation}\label{fig:FactorisationExcessIntersectionDiagram}
    \begin{tikzcd}[row sep = 15pt, column sep = 20pt]
    {p'}^*A\otimes_\El {p'}^*A\ar[dd,"m"]\ar[r,"\Delta\otimes_\El\Delta"] &A^{\otimes_\El' 2}\otimes_\El A^{\otimes_\El' 2}\ar[r,"\sigma_{23}"] &A^{\otimes_\El 2}\otimes'_\El A^{\otimes_\El 2}\ar[d,"R_{14}^{-1}\sigma(R_{23})^{-1}"]  \\
    & & A^{\otimes_\El 2}\otimes_\El' A^{\otimes_\El 2}\ar[d,"m\otimes m"]\\
    \tilde{p}^*A\ar[rr,"\Delta"] &  & {p}^*A\otimes'_\El {p}^*A
    \end{tikzcd}
    \end{equation}
    in $\Gamma((C\times C)\times_{Y\times Y}C',\tilde{p}^*\El)$ commutes, where $\tilde{p}$ denotes the projection to $Y$.
  \end{defn}
 
One consequence of $A$-linearity is the following commuting
\begin{equation}\label{eqn:SpectralRMatrixTwisted}
  \begin{tikzcd}[row sep = 25pt, column sep = 65pt]
  A^{\otimes_\El' 2}\otimes_\El A^{\otimes_\El' 2}\ar[d,"m\otimes m"]\ar[r,"R_{13}R_{14}R_{23}R_{24}"] &A^{\otimes_\El' 2}\otimes_\El A^{\otimes_\El' 2}\ar[d,"m\otimes m"] \\
  A\otimes_\El A\ar[r,"R"] & A\otimes_\El A
  \end{tikzcd}   
\end{equation}
Thus as is the case for classical Borherds products, we have the following way to produce new factorisation bialgebras:

\begin{prop}
  A Borcherds-twisted factorisation coalgebra $A_R$ forms a braided cocommutative bialgebra in $\textup{Mod-}A$ with respect to any factorisation Borcherds product $m$, where $\textup{Mod-}A$ has braided factorisation structure induced by spectral $R$-matrix $\sigma(R)R^{-1}$. 
\end{prop}
\begin{proof}
 Proceeds exactly as in the proof of Proposition \ref{prop:BorcherdsBialgebra} by sewing (\ref{fig:FactorisationExcessIntersectionDiagram}) to the diagram resulting from (\ref{eqn:SpectralRMatrixTwisted}).
\end{proof}

One can write down analogues of the above results for op factorisation quantum groups, i.e. factorisation bialgebras equipped with a spectral op-$R$-matrix. This is the structure appearing in \cite{La,Li} in the compatiblity between Joyce vertex algebras and cohomological Hall algebras, which gives a Borcherds coproduct.

\newpage

\section{Examples} \label{sec:Examples}

\noindent In this section we give examples of factorisation spaces, categories, and vertex quantum groups.

For motivation we first give examples of quantum groups; for more see for instance \cite{Maj,ES}. The prototypical example of a quantum group is the Drinfeld-Jimbo quantum group $U_q(\gk)$ attached to $\gk$ is a finite-dimensional simple Lie algebra over the complex numbers. But for a simpler example, consider the category $\Vect_{\Zb/2}$ of $\Zb/2$-graded vector spaces, with its \textit{sign} symmetric braiding 
$$ V\otimes W\ \to\ W\otimes V\hspace{15mm} v\otimes w\ \mapsto \ (-1)^{|v|\cdot |w|}w\otimes v$$
for $v,w$ homogeneous vectors. Note that this is the category of modules for the quantum group $A=\Ol(\Zb/2)$, where for any module the idempotents $e_0,e_1\in A$ project onto the even and odd parts respectively, and $e_0+e_1=1$. This braiding is induced by the $R$-matrix 
$$R\ =\ (e_0-e_1)\otimes (e_0-e_1).$$
 More generally, for a discrete group $G$, an $R$-matrix on $A=\Ol(G)$ is determined by a function $R : G\times G \to k.$ The hexagon and unit conditions say that this is multiplicative in both variables, and almost cocommutativty says that for any function $f\in \Ol(G)$,
$$R(g_1,g_2)f(g_1g_2)\ =\ f(g_2g_1)R(g_1,g_2),$$
or in other words that $R(g_1,g_2)=0$ unless $g_1,g_2$ commute.

We can also consider the dual example $A=\Ol(G)^*$, which for any algebraic group $G$ forms a cocommutative bialgebra with coproduct induced by multiplication of functions, and hence is a quantum group. Likewise $U(\gk)$ for any Lie algebra $
\gk$. This is the origin of the term ``quantum group''.

\subsection{Factorisation spaces and categories}
\subsubsection{Groups} If $G$ is an algebraic group, then it has factorisation space structure
\begin{center}
\begin{tikzcd}[row sep = {30pt,between origins}, column sep = {45pt,between origins}]
 &G\times G\ar[ld,equals]\ar[rd,"m"] & \\ 
G\times G & & G
\end{tikzcd}
\end{center}
More generally the same is true for any monoid object in $\PreStk$, for instance the Ran space with its $*$-factorisation structure. If $G$ is commutative then the factorisation structure is commutative. 

A strong classical factorisation category over $G$ is then the same as a \textit{multiplicative} classical sheaf of categories $\Cl$, i.e. a classical sheaf of categories with an equivalence
$$\tau\ :\ \Cl\boxtimes \Cl\ \simeq \  m^*\Cl$$
satisfying the cocycle condition. Thus we have equivalences of categories
$$\tau_{g_1,g_2}\ :\ \Cl_{g_1}\otimes \Cl_{g_2}\ \simeq\ \Cl_{g_1\cdot g_2}$$
for all $g_1,g_2\in G$. If $G$ is commutative, we can ask that $\tau_{g_1,g_2}$ and $\tau_{g_2,g_1}$ induce the same equivalence, in which case this is a symmetric factorisation category.

A factorisation algebra in $\Cl$ is then a section $A\in \Gamma(G,\Cl)$, with an associative multiplication map 
$$m_A\ :\ \tau(A\boxtimes A)\ \to\ m^*A$$
and a compatible unit $1_A\in \Gamma(G,A)$. Over fibres this gives maps $m_{A,g_1,g_2} :\tau_{g_1,g_2}(A_{g_1}\otimes A_{g_2}) \to A_{g_1\cdot g_2}$ in the category $\Cl_{g_1g_2}$, satisfying the cocycle condition. Let now $G$ be commutative and $A$ have an additional compatible factorisation coalgebra structure $\Delta$. In this setting, a spectral $R$-matrix is an invertible section  
$$R\ \in\ \Gamma(G\times G,\tau(A\boxtimes A))\ \simeq \ \Gamma(G\times G,A\boxtimes A)$$
satisfying conditions, in particular on fibres we have a collection of $R_{g_1,g_2}\in A_{g_1}\otimes A_{g_2}$. The hexagon relations say that this is multiplicative in both variables
$$R(g_1,g_{2}\cdot g_{3})\ =\ R(g_1,g_2)R(g_1,g_3), \hspace{15mm} R(g_{1}\cdot g_2,g_3)\ =\ R(g_1,g_3)R(g_2,g_3),$$ 
 compatible with the unit as $R(1,g)=R(g,1)=1$, and making the coproduct of $A$ almost cocommutative, which on fibres says
 $$\Delta_{g_2,g_1}\ =\ R_{g_1,g_2}\Delta_{g_1,g_2}R_{g_1,g_2}.$$

\subsubsection{Ran spaces} The Ran space and associative Ran space of a prestack $X$ were introduced in section \ref{sec:RanSpace}, and we will give a braided version here. 

To begin, recall that a (\textit{$\varnothing$, braided, symmetric}) \textit{operad} valued in background symmetric monoidal category $\Cl$ is a monoid object in $\Fun((-)^{\textup{op}},\Cl)$, for one of the following respective domain categories:
\begin{itemize}
 \item the category $\Nb$ of finite sets with no nontrivial morphisms, 
 \item the \textit{braid} category $\Bb$ of finite sets with morphisms a braid group's $\Bk_n$ worth of automorphisms acting on each object, 
 \item the category $\Pb$ of finite sets with bijections. 
\end{itemize}
Thus, by analogy we finish our trilogy of definitions by saying the \textit{braided Ran space} $\Ran_{br}X$ is the colimit 
\begin{center}
\begin{tikzcd}[row sep = 20pt, column sep = 20pt]
  \Ran_{\textup{br}} X\ =\ \colim_{I \in \FinSet^{\textup{br}}}X^I\ =\ \colim\bigg( X\ar[r,std] & X^2\ar[r,std,"\Bk_{3,2}"]\arrow[loop, distance=2em, in=50, out=130, looseness=5,"\Bk_2"] & X^3  \ar[r,"\Bk_{4,3}",std]\arrow[loop, distance=2em, in=50, out=130, looseness=5,"\Bk_3"] & \cdots\ \ \bigg)
\end{tikzcd}
\end{center} 
where $\FinSet^{\textup{br}}$ is the category of nonempty finite sets and surjective \textit{braided maps} them, and we write $\Bk_{n,m}$ for the set of braided surjections $[n]\to [m]$.  Here a braided map $I\to J$ is a homotopy class of non-intersecting paths in the solid cylinder, where the paths start and end on fixed finite subsets of size $|I|$ and $|J|$ on opposite faces.
 For instance, a braided map $[4]\to [2]$ is 
\begin{center}
  \begin{tikzpicture}
    % Lines
    \draw[->, shorten >=5pt, shorten <=5pt] (0,2*0.7) -- (2,2*0.7);
    \draw[line width = 5pt, white, shorten >=5pt, shorten <=5pt] (0,3*0.7) -- (2,1*0.7);
    \draw[->, shorten >=5pt, shorten <=5pt] (0,3*0.7) -- (2,1*0.7);
    \draw[line width = 5pt, white, shorten >=5pt, shorten <=5pt] (0,1*0.7) -- (2,2*0.7);
    \draw[->, shorten >=5pt, shorten <=5pt] (0,1*0.7) -- (2,2*0.7);
    \draw[->, shorten >=5pt, shorten <=5pt] (0,0*0.7) -- (2,1*0.7);

    % Dots at x position 0
    \foreach \x in {0,1,2,3} {
      \fill (0,\x*0.7) circle (2pt);
    }
    
    % Dots at x position 2
    \foreach \x in {1,2} {
      \fill (2,\x*0.7) circle (2pt);
    }
  \end{tikzpicture}
\end{center}
The braided Ran space has two braided factorisation space structures as in Definition \ref{defn:RanFactStructure}, the $*$ and the chiral. The chiral factorisation structure is given by union of disjoint subsets as before, and the braiding 
\begin{center}
  \begin{tikzcd}[row sep = {30pt,between origins}, column sep = {85pt, between origins}]
    &(\Ran_{\textup{br}}X\times \Ran_{\textup{br}}X)_\circ\ar[rdd, bend left = 30]\ar[ldd, bend right = 30]\ar[d,swap,"\wr","\tilde{\sigma}"'] & \\[5pt]
    &(\Ran_{\textup{br}}X\times \Ran_{\textup{br}}X)_\circ\ar[rd]\ar[ld] & \\
    \Ran_{\textup{br}}X\times \Ran_{\textup{br}}X & & \Ran_{\textup{br}}X
  \end{tikzcd}
  \end{center}
  where $\tilde{\sigma}$ swaps both factors, and the two-morphism making the right triangle commute is induced by the following braided automorphism of $I_1\sqcup I_2$, for all nonempty finite sets $I_1,I_2$:
\begin{center}
  \begin{tikzpicture} 
    % Lines
    \draw[line width = 5pt, white, shorten >=5pt, shorten <=5pt] (0,3*0.7) -- (2,1*0.7);
    \draw[->, shorten >=5pt, shorten <=5pt] (0, 0) -- (2, 0.9*0.8);
    \draw[->, shorten >=5pt, shorten <=5pt] (0, 0.16) -- (2, 1.1*0.8);
    \draw[->, shorten >=5pt, shorten <=5pt] (0, 0.32) -- (2, 1.3*0.8);
    \draw[->, shorten >=5pt, shorten <=5pt] (0, 0.48) -- (2, 1.5*0.8);

    \draw[line width=5pt, white, shorten >=5pt, shorten <=5pt] (0, 1.1*0.8) -- (2, 0);
    \draw[line width=5pt, white, shorten >=5pt, shorten <=5pt] (0, 1.3*0.8) -- (2, 0.16);
    \draw[line width=5pt, white, shorten >=5pt, shorten <=5pt] (0, 1.5*0.8) -- (2, 0.32);
    
    \draw[->, shorten >=5pt, shorten <=5pt] (0, 1.1*0.8) -- (2, 0);
    \draw[->, shorten >=5pt, shorten <=5pt] (0, 1.3*0.8) -- (2, 0.16);
    \draw[->, shorten >=5pt, shorten <=5pt] (0, 1.5*0.8) -- (2, 0.32);

    % Dots at x position 0
    \foreach \x in {0,0.2,0.4,0.6,1.1,1.3,1.5} {
      \fill (0,\x*0.8) circle (2pt);
    }
    
    % Dots at x position 2
    \foreach \x in {0,0.2,0.4,0.9,1.1,1.3,1.5} {
      \fill (2,\x*0.8) circle (2pt);
    }

    % Label for the three-dot group
\node[left] at (0, 0.3*0.8) {$I_1$};
\node[left] at (0, 1.3*0.8) {$I_2$};

% Label for the four-dot group
\node[right] at (2, 1.2*0.8) {$I_1$};
\node[right] at (2, 0.2*0.8) {$I_2$};
  \end{tikzpicture}
\end{center}
As in \cite[3.1.2]{Ta}, there is a second definition when $X$ is a scheme as in section \ref{sec:RanSpace}, by defining $\Ran_{\textup{br}} X(T)$ for any finite type affine scheme $T$ as the $1$-groupoid of ordered finite subsets of $X(T)$ and braided automorphisms. We expect these two definitions are equivalent.

\subsubsection{Ravioli BD Grassmannian} If $X$ is a scheme and $G$ an algebraic group, the \textit{Beilinson Drinfeld Grassmannian} $\Gr_{G,X}$ is a commutative factorisation space over $\Ran^{ch}X$. It parametrises $G$ bundles with trivialisation away from a specified finite subset of $X$: its $T$-points are
$$\Gr_{G,X}(T)\ =\ \{(P,\varphi,S)\ :\ P\in \Bun_G X_T,\ S\in \Ran X(T),\ \varphi\ :\ P\vert_{X_T\setminus S} \simeq\triv\}$$
for a scheme $T$. The factorisation structure is inherited from $\Ran X$.

One may define a variant, the \textit{ravioli Beilinson Drinfeld Grassmannian} $\tilde{\Gr}_{G,X}$, parametrising pairs of $G$ bundles with trivialisations away from \textit{the same} finite subset of $X$:
$$\tilde{\Gr}_{G,X}(T)\ =\ \{(P_0,P_1,\varphi_0,\varphi_1,S)\ :\ P_i\in \Bun_G X_T,\ S\in \Ran X(T),\ \varphi_i\ :\ P_i\vert_{X_T\setminus S} \simeq\triv\}.$$
Again this has the \textit{chiral} factorisation structure inherited from $\Ran X$. However, it has an additional compatible \textit{ravioli} factorisation structure, given by the correspondence
\begin{center}
\begin{tikzcd}[row sep = 10pt, column sep = 15pt]
 & \tilde{\tilde{\Gr}}_{G,X}\ar[rd]\ar[ld] & \\ 
\tilde{\Gr}_{G,X}\times\tilde{\Gr}_{G,X}& &\tilde{\Gr}_{G,X} 
\end{tikzcd}
\end{center}
parametrising triples of $G$ bundles with trivialisation away from the same finite subset of $X$. When $X$ is a curve, the fibre of this correspondence over a point $x\in X$ is the classical correspondence structure on the equivariant affine Grassmannian $G(\Ol)\backslash \Gr_G$.

\subsubsection{Quasicoherent sheaves} Since for any map of prestacks $f:Y\to X$ we have 
$$f^*\QCoh_X\ =\ \QCoh_Y$$
and $\QCoh_Y$ is the unit in $\ShvCat(Y)$, it follows that it has a canonical structure of a factorisation $\Eb_n$-category whenever $X$ is a factorisation $\Eb_n$-space. It follows that for any $\Eb_n$-category $\Al$, the pullback along $p:Y\to \pt$
$$\Al_Y\ =\ p^*\Al$$
inherits the so-called \textit{constant} structure of a factorisation $\Eb_n$-category. For instance, over every point $c$ of the correspondence $C$ defining a factorisation $\Eb_1$ structure on $Y$, we have a monoidal structure 
$$\Al\otimes\Al\ \simeq\ \Al_{Y,y_1}\otimes \Al_{Y,y_2}\ \stackrel{\otimes_{\Al_Y,c}}{\to}\ \Al_{Y,y}\ \simeq\ \Al.$$

\subsection{Vertex quantum groups} 

In this section we give examples of vertex quantum groups. The non-vertex coproduct will mostly be cocommutative, and we will leave the question of deforming them to future work. However, in the following section we will give explicit examples of non-cocommutative vertex quantum groups.

\subsubsection{Jet groups} If $X$ is a scheme over field $k$, we may take its \textit{jet space} $J_\infty X$, an ind-scheme whose $R$-points for algebra $R$ are given by
$$J_\infty X(R)\ =\ \Maps(D,X),$$
 the scheme maps from the formal disk $D=\Spec R[[t]]$ into $X$. See for instance \cite{KV}. The $\Ga$ action on $D$ induces the \textit{loop rotation} action on the jet space:
$$a_X \ :\ \Ga\times J_\infty X\ \to\ J_\infty X.$$
The assignment $X\mapsto (J_\infty X, a_X)$ gives a symmetric monoidal functor from schemes to ind-schemes equipped with a $\Ga$ action. 
Thus when $G$ is an algebraic group, the \textit{jet group} $J_\infty G$ inherits the structure of a group with compatible diagonal comonoid stucture, with $\Ga$ action compatible with both structures. It follows that $\Ol(J_\infty G)$ together with the vector field $\partial_t$ induced by the $\Ga$ action is a bialgebra with biderivation, i.e. a holomorphic vertex bialgebra. Note that we in fact inherit from $D$ the action of the positive half $k\{t^n\partial_t:n\ge 0\}$ of the Witt Lie algebra.

If $G$ is abelian, this is cocommutative and so $\Ol(J_\infty G)$ defines a holomorphic vertex quantum group.

\subsubsection{Universal enveloping vertex algebras} We recall the universal enveloping algebra construction due to Beilinson and Drinfeld, following \cite[$\S$ 12]{Bu}. 

For any variety $X$ is a left adjoint
$$U^{ch}(-)\ :\ \textup{Lie}_{\Delta}(\Dl\Md_{\Ran X},\otimes^*)\ \leftrightarrows\ \Eb_\infty \CoAg^{\textup{st}}(\Dl\Md_{\Ran X},\otimes^{ch})\ :\ \oblv$$
where the left hand side is the category of $\otimes^*$-Lie algebras supported on the main diagonal $X\subseteq \Ran X$. The functor 
$$\oblv \ :\ A\ \mapsto \ \textup{KD}(A)\vert_{\otimes^*}$$
first applies chiral Koszul duality of \cite{FG} to give a $\otimes^{ch}$-Lie algebra, then restricts along $\otimes^*\to \otimes^{ch}$ to give a $\otimes^*$-Lie algebra. 

Now we recall in a long-winded but convenient way why the universal enveloping algebra $U(\gk)$ of a Lie algebra is a bialgebra. We use the adjunction 
$$\Hom_{\textup{Lie}}(\gk,A)\ =\ \Hom_{\textup{Alg}}(U(\gk),A)$$
three times. The first to $A=U(\gk)$ gives a map of Lie algebras $\gk\to U(\gk)$. In particular we get a map of Lie algebras
\begin{equation}\label{eqn:Coproduct}
  \gk\ \to\ U(\gk)\otimes U(\gk)\hspace{15mm} x\ \mapsto \ x\otimes 1+1\otimes x
\end{equation}
and setting $A=U(\gk)\otimes U(\gk)$ gives a comultiplication compatible with the product. Cocommutativity is inherited from (\ref{eqn:Coproduct}). Thirdly, we set $A=0$ to get the counit. 
\begin{prop}
  $U^{ch}(L)$ is a commutative cocommutative factorisation bialgebra in $(\Dl\Md_{\Ran X},\otimes^{ch},\otimes)$, and its $\otimes^{ch}$-algebra structure is strong.  
\end{prop}
\begin{proof}
  Note that if $A$ is a (strong) factorisation $\otimes^{ch}$-algebra, then so is $A\otimes A$. Indeed, there is a strong braiding between $\otimes$ and $\otimes^{ch}$, and so given any factorisation algebra structure $m:A\otimes^{ch}A \stackrel{\sim}{\to} p^*A$, we get 
$$(A\otimes A)\otimes^{ch}(A\otimes A)\ \simeq\ (A\otimes^{ch} A)\otimes (A\otimes^{ch} A)\ \stackrel{\sim}{\to}\ p^*A\otimes p^*A\ \simeq\ p^*(A\otimes A),$$
defining a strong factorisation algebra structure. In particular, the tensor product $U^{ch}(L)\otimes U^{ch}(L)$ is also a strong commutative factorisation coalgebra. We can then use the adjunction  
$$\Hom_{\textup{Lie}_{\Delta}(\Dl\Md_{\Ran X},\otimes^*)}(L,A)\ =\  \Hom_{\Eb_\infty \CoAg^{\textup{st}}(\Dl\Md_{\Ran X},\otimes^{ch})}(U^{ch}(L),A)$$
for $A=0,U^{ch}(L),U^{ch}(L)\otimes U^{ch}(L)$ and apply exactly the same argument as for $U(\gk)$ to give the result.  
\end{proof}

Note that translation equivariant $\otimes^*$-Lie algebras on $X=\Ab^1$ have a explicit equivalent definition as \textit{vertex Lie algebras}, see \cite[10.5.4]{Bu} for a proof. They are called \textit{Lie conformal algebras} by \cite{KDS}, in which they also define ``non-linear'' examples, which as proven in \cite{KL} include the $\Wl_{1+\infty}$ vertex algebras. One might expect that universal enveloping vertex algebras of these nonlinear Lie conformal algebras are also vertex quantum groups. 

For instance, recall the affine vertex algebra  $V_\kappa(\gk)$ attached to a finite dimensional Lie algebra $\gk$ and ad-invariant bilinear form $\kappa$. It is a chiral universal envelope of $L=\gk_\Dl$ as in \cite[10.5.7]{Bu}, and is strongly generated by fields $\alpha(z)$ as $\alpha$ varies over a basis of $\gk$. The coproduct $\Delta$ is thus uniquely determined by specifying 
$$\Delta(\alpha(z))\ =\ \alpha(z)\otimes\id\ +\ \id\otimes\alpha(z)$$
making $V_\kappa(\gk)$ into a cocommutative vertex bialgebra. Choosing $\kappa=0$ returns the jet group construction $\Ol(J_\infty \gk^*)$ where $\gk^*$ is viewed as a group under addition.

\subsubsection{Nonsymmetric lattice vertex algebras} We may get more examples by taking Borcherds twists as in section \ref{sec:BorcherdsTwists}. Let $(\Lambda,\kappa)$ be a lattice with integral bilinear form, which we do not assume is symmetric. Then writing $\hk=\Lambda\otimes_\Zb k$ for the vector space on $\Lambda$, we set 
$$A\ =\ k[\Lambda]\otimes \Ol(J_\infty \hk )$$
with its natural cocommutative bialgebra structure. Extending the derivation to the first factor by $T(e^\lambda) =\lambda_{-1}e^\lambda$ makes it into a cocommutative holomorphic vertex quantum group.\footnote{One would like to say that when $k$ is the complex numbers, $A$ is functions on the loop space of $\Cb^n/\Lambda\simeq(\hk)/\Lambda$, or rather the algebraic torus having this as a maximal compact, which would thus automatically inherit this algebraic structure, but such a notion of loop space does not exist currently. The data $(\kappa,\epsilon)$ below should presumably parametrise its central extensions.} To define a vertex $R$-matrix, we begin with a bicharacter
$$k[\Lambda]\otimes k[\Lambda]\ \to \ k(z)\hspace{15mm}e^\lambda\otimes e^\mu\ \mapsto\ \epsilon_{\lambda,\mu} z^{\kappa(\lambda,\mu)}$$
where $\epsilon_{\lambda,\mu}$ are choices of signs as in \cite[$\S$ 5.4.2]{FBZ}. Then taking \cite[Lem 2.15]{Bo2} applied to $M=k[\Lambda]$ and $H=k[\partial]$ we see that this extends uniquely to a vertex $R$-matrix 
$$R(z)\ :\ A\otimes A\ \to \ k(z).$$
It is symmetric in the sense of Proposition \ref{prop:SymmetricBorcherdsTwist} if and only if 
$$\epsilon_{\lambda,\mu}z^{\kappa(\lambda,\mu)}\ =\ \epsilon_{\mu,\lambda}(-z)^{\kappa(\mu,\lambda)}$$
which is equivalent to $\kappa$ being symmetric. The resulting Borcherds twist is called the \textit{lattice vertex algebra} $V_{(\Lambda,\kappa)}$, and when $\kappa$ is symmetric this is a commutative vertex algebra. This construction is due to Borcherds, and it inherits the cocommutative coproduct from $A$.\footnote{To be concrete, the coproducts on $k[\Lambda]$ and $\Ol(J_\infty\hk)$ are
$$\Delta(e^{\lambda(z)})\ =\ e^{\lambda(z)}\otimes e^{\lambda(z)}, \hspace{15mm} \Delta(\lambda(z))\ =\ \lambda(z)\otimes \id \ +\ \id\otimes\lambda(z).$$
or more explicitly, $\Delta(e^\alpha)\ =\ e^\alpha\otimes e^\alpha$ and $\Delta(\lambda_{1,-n_1}\cdots \lambda_{k,-n_k}|0\rangle)\ =\ \prod (\lambda_{i,-n_i}\otimes \id\ +\ \id\otimes \lambda_{i,-n_i})\, |0\rangle\otimes|0\rangle$.}

We may repeat this construction, replacing $k$ with any op quantum group $H$ with op-$R$-matrix $R$. Then we get a bicharacter 
$$H[\Lambda]\otimes H[\Lambda]\ \to\ k(z)\hspace{15mm} h_\lambda e^{\lambda}\otimes h_\mu e^{\mu}\ \mapsto\ R(h_\lambda,h_\mu)\epsilon_{\lambda,\mu}z^{\kappa(\lambda,\mu)}$$
which by the same argument extends to a vertex $R$-matrix on the $H$-linear vertex algebra $A=H[\Lambda]\otimes\Ol(J_\infty\hk_H)$, where $\hk_H=\hk\otimes_kH$.

\subsection{Moduli stacks}  \label{sec:ModuliSpaces}

For a dg category $\Al$, its \textit{moduli stack of objects} $\Ml=\Ml_\Al$ is the prestack defined by
$$\Ml(R)\ \defeq\ \Hom_{\dgCat}(\Al,\Perf R)$$
for any ring $R$. If $\Al$ has finite type, \cite{TV} showed that this is a derived stack which is a union of open $n$-geometric (in particular, $n$-Artin) substacks which are locally of finite presentation. In particular, we can apply the theory \cite{LZ} of $\ell$-adic sheaves and hence define its (co)homology. We will from now on assume this. 

\begin{lem}
  $\Ml$ is a monoid object with an action of $\Perf=\Ml_{\Vect}$. 
\end{lem}
\begin{proof}
  The assignment $\Al\mapsto \Ml_\Al$ is symmetric monoidal and functorial in $\Al$. In particular, the direct sum map on $\Al$ gives $\Ml$ a monoid structure. Likewise, the direct sum and tensor structure on $\Vect$ gives $\Perf$ two compatible monoid structures, and the action $\Vect\otimes\Al\to \Al$ gives the result.
\end{proof}
Note that we have $\BGm\subseteq \Perf$ as group stacks, and writing $\Ht_\sbt(\BGm)\simeq k[\tau]$ as algebras, 

\begin{cor}
  $\Ht_\sbt(\Ml)$ is a cocommutative bialgebra with with biderivation $\tau_*$. 
\end{cor}
 
It follows that this is a (cocommutative) vertex quantum group. 

\subsubsection{$R$-matrices from chern classes} We can consider the forgetful map of dg-categories $\Ext_\Al\to \Al\otimes \Al$ where $\Ext_\Al$ classifies pairs of objects in $\Al$ and a map between them. Thus we get a map of prestacks 
$$\Ext\ \to\ \Ml\times\Ml$$
which satisfies 
\begin{lem}
  $\Ext$ is a multiplicative perfect complex with respect to the monoid structure on $\Ml\times\Ml$, and has weights $(-1,1)$ with respect to the action of $\Perf$ (or $\BGm$) on the two factors of $\Ml$. 
\end{lem}

In particular, we get the \textit{Joyce} vertex $R$-matrix
\begin{equation}
  \label{eqn:JoyceRMatrix}
  R(z)\ =\ \sum_{k\ge 0} z^{\rank \Ext -k}c_k(\Ext). 
\end{equation}
It is an element of cohomology, so defines a $k((z^{-1}))$-valued functional on homology. The hexagon relations are satisfied because $\Ext$ is multiplicative, and the compatiblity with $T$ because of the weight conditions. In the case $\Ext$ is a vector bundle it is 
$$(\oplus\times\id)^*(e(\Ext)) \ \in \ \Ht^\sbt(\BGm\times\Ml^2)\ \simeq\ \Ht^\sbt(\Ml^2)[z],$$
 where $e$ denotes the Euler class. In particular, we get the associative vertex algebra structure on $\Ht_\sbt(\Ml)$ discovered by Joyce in \cite{Jo}\footnote{The dual map on cohomology is more natural:
 $$\Delta(\alpha,z)\ =\ R(z) \cdot \act_{1}^*\oplus^*\alpha$$
 where $\act_1:\BGm\times\Ml^2\to\Ml^2$ is induced by the action on the first factor and we identify $\Ht^\sbt(\BGm)\simeq k[z]$.}
 \begin{equation}
   \label{eqn:JoyceBorcherdsTwist}
   Y(\alpha,z)\beta\ =\ \oplus_*((e^{z\tau_*}\otimes\id)\cdot R(z)\cap(\alpha\otimes\beta)).
 \end{equation}

 \subsubsection{Extension factorisation structure} If $\Al$ is a dg-category, we may consider the dg-category $\SES(\Al)$ of distinguished triangles in $\Al$, giving a correspondence 
 \begin{center}
 \begin{tikzcd}[row sep = {30pt,between origins}, column sep = {45pt,between origins}]
  &\Ml_{\textup{SES}(\Al)}\ar[rd]\ar[ld] & &[30pt] & a_1\to e\to a_2\ar[ld,|->]\ar[rd,|->]& \\ 
 \Ml_\Al\times \Ml_\Al & & \Ml_\Al & (a_1,a_2)&& e
 \end{tikzcd}
 \end{center}
 Indeed, in the notation of section \ref{sec:ClassicalFactorisationSpaces}, the head  $(C\times C)\times_{Y\times Y}C'$ of the correspondence (\ref{fig:OrdBraid}) parametrises diagrams of short exact sequences in $\Al$
\begin{center}
\begin{tikzcd}[row sep = {30pt,between origins}, column sep = {40pt,between origins}]
a\ar[r] &e\ar[d]\ar[r] &b\\ 
 &E\ar[d] &\\
a'\ar[r]& e'\ar[r]& b'
\end{tikzcd}
\end{center}
The second head $(C'\times C')\times_{Y\times Y}Y$ parametrises the same but with $b$ and $a'$ swapped, so in general these two correspondences are \textit{not} the same, nor is there a map of correspondences from one space to the other.

\subsubsection{Relative version} \label{sec:RelativeVersion} Let $B$ be a fixed prestack, and $\Al$ a sheaf of categories over it. We define the relative moduli stack of sections $\Ml=\Ml_\Al$ by 
$$\Ml(R)\ =\ \Maps_{\ShvCat(B)}(\Perf R,\Al)$$
for $\Spec_B R\to B$ any affine map. As before, we will from now on assume that as in the absolute case that $\Ml$ and $B$ are derived stacks which are the union of open $n$-geometric (in particular, $n$-Artin) substacks which are locally of finite presentation.\footnote{We expect it is enough to assume that this holds for $B$ and that $\Al$ satisfies a finite type condition.} We get as before

\begin{lem}
  $\Ml$ is a monoid object in $\PreStk_{/B}$ with an action of $\Perf B=\Ml_{\QCoh_B}$. 
\end{lem}

Note that we have $B\times\BGm\subseteq \Perf B$ as group prestacks over $B$. From now on we assume that the tautological map
\begin{equation}
  \Ht^\sbt(\Ml)\otimes_{\Ht^\sbt(B)}\Ht^\sbt(\Ml)\ \stackrel{\sim}{\to}\ \Ht^\sbt(\Ml\times_B\Ml)
\end{equation}
is an isomorphism, and so as before
\begin{cor}
  $\Ht^\sbt(\Ml)$, and hence also its graded $\Ht^\sbt(B)$-linear dual\footnote{To be concrete, $\Ht_{B,\sbt}(\Ml)$ consists of $\Ht^\sbt(B)$-linear $\Zb$-graded linear maps $\Ht^\sbt(\Ml)\to \Ht^\sbt(B)$.} $\Ht_{B,\sbt}(\Ml)$, is a cocommutative bialgebra in $\Ht^\sbt(B)\Md$ with biderivation.
\end{cor}

In particular, $\Ht_{B,\sbt}(\Ml)$ is a (cocommutative) vertex group in $\Ht^\sbt(B)\Md$ viewed as a symmetric monoidal category with $\otimes_{\Ht^\sbt(B)}$ defining its product. As before, we let $\Ext_\Al\to \Al\otimes\Al$ parametrise pairs of sections of $\Al$ and a map between them, so that we get a map of prestacks over $B$
$$\Ext\ \to\ \Ml\times\Ml$$
which as before is a multiplicative perfect complex with weights $(-1,1)$ with respect to the $\Perf_B$ (or $B\times\BGm$) action on the two factors of $\Ml$. We use the same formula (\ref{eqn:JoyceRMatrix}) to define 
$$R(z)\ \in\ \Ht^\sbt(\Ml\times_B\Ml)((z^{-1}))\ \simeq\ \Ht^\sbt(\Ml)\otimes_{\Ht^\sbt(B)}\Ht^\sbt(\Ml)((z^{-1}))$$
taking cap product with which gives an $\Ht^\sbt(B)$-linear vertex $R$-matrix. Thus just as before, 
\begin{prop}
  The Borcherds twist (\ref{eqn:JoyceBorcherdsTwist}) defines an associative vertex algebra structure on $\Ht^\sbt_{B,\sbt}(\Ml)$. 
\end{prop}

This structure is commutative if the vertex $R$-matrix is symmetric, $\sigma(R(z))=R(-z)$. However, the reason we introduce the relative version is that $(\Ht^\sbt(B)\Md,\otimes_{\Ht^\sbt(B)})\simeq \QCoh(\Spec \Ht^\sbt(B))$ potentially has other braided monoidal structures, e.g. arising from pull-push along a braided factorisation space structure on $\Spec \Ht^\sbt(B)$. Using the construction of this section one can for instance get $q$-deformed vertex algebras, as in for instance \cite{Li}, where one takes the quotient of moduli stacks by a torus action.

\subsubsection{Remark} We now explain what the vertex quantum group and related structures are which appear in \cite{La,Li} when considering cohomological Hall algebras. Consider an braided commutative bialgebra
$$\Hl\ \in\ \Eb_2\FactAg(\Eb_1\FactCoAg(\El))$$
with respect to $\otimes,\otimes^{ch}$ as follows:
$$\Hl\ \stackrel{\cup^*}{\to}\ \Hl\otimes\Hl, \hspace{10mm}\Hl\otimes\Hl\ \stackrel{\oplus_*}{\to}\ \Hl, \hspace{10mm} m\ :\ \Hl\otimes^{ch}\Hl\ \stackrel{\sim}{\to}\ \Hl.$$
Thus, $(\Hl,\cup^*)$ forms a bialgebra with respect to $\oplus_*$ and $m$, which are compatible themselves as they form an $\Eb_2$-algebra. It follows that 
\begin{lem}
  $\Cl=(\Hl,\cup^*)\CoMd(\El)$ is a braided factorisation category.  
\end{lem}
Then by a Borcherds twist construction as in section \ref{sec:BorcherdsTwists} one constructs a bialgebra $\Vl$ in $\Cl$, with structure maps
$$Y\ :\ \Vl\otimes^{ch}\Vl\ \stackrel{\sim}{\to}\ p^*\Vl, \hspace{15mm} \Delta\ :\ \Vl\ \to\ \Vl\otimes\Vl.$$
The fact that these structure maps need to be linear under the coaction $\Vl \stackrel{\cup^*}{\to} \Hl\otimes\Vl$ of $\Hl$ says that the diagrams
\begin{center}
  \begin{tikzcd}[row sep = {30pt,between origins}, column sep = {30pt}]
  \Vl\otimes^{ch}\Vl\ar[r,"Y"]\ar[d,"\cup^*\otimes^{ch}\cup^*"]&p^*\Vl\ar[dd,"\cup^*"]\\
  (\Hl\otimes\Vl)\otimes^{ch}(\Hl\otimes\Vl)\ar[d,"\wr"]& \\
  (\Hl\otimes^{ch}\Hl)\otimes(\Vl\otimes^{ch}\Vl)\ar[r,"m\otimes Y"]&p^*(\Hl\otimes\Vl) 
  \end{tikzcd}
  \end{center}
and
\begin{center}
  \begin{tikzcd}[row sep = {30pt,between origins}, column sep = {30pt}]
  \Vl\ar[r,"\Delta"]\ar[dd,"\cup^*"]&\Vl\otimes \Vl\ar[d,"\cup^*\otimes\cup^*"]\\
  & (\Hl\otimes\Vl)\otimes(\Hl\otimes\Vl)\ar[d,"\wr"]\\
  \Hl\otimes\Vl\ar[r,"\cup^*\otimes \Delta"]&(\Hl\otimes\Hl)\otimes(\Vl\otimes \Vl)
  \end{tikzcd}
  \end{center}
commute. If we in addition assume that both $\Hl$ and $\Vl$ are $\Zb$-graded, we may take the contragredient dual vertex coalgebra $V$, and the above diagrams are then equivalent to the commutativity of
\begin{center}
\begin{tikzcd}[row sep = {30pt,between origins}, column sep = {20pt}]
V\otimes V((z)) \ar[r,<-]\ar[d,"m^\vee(\alpha)\cup"] & V \ar[d,"\alpha\cup"]\\ 
V\otimes V((z)) \ar[r,<-] & V \ 
\end{tikzcd}
\end{center}
and
\begin{center}
\begin{tikzcd}[row sep = {30pt,between origins}, column sep = {20pt}]
V \ar[r,<-]\ar[d,"\alpha\cup"] & V \ar[d,"(\oplus^*\alpha)\cup"]\\ 
V \ar[r,<-] & V\otimes V  
\end{tikzcd}
\end{center}
commutes, for all $\alpha\in H^\vee$, where $m^\vee\alpha\in H\otimes H((z))$ and $\oplus^*\alpha\in H\otimes H$.

 \newpage

\appendix

\section{Vertex Algebras} \label{sec:RationalMultiplicativeElliptic}

\subsubsection{}  \label{sec:VAFact} In this section, we recall the equivalence--also covered in \cite{FBZ}--between vertex algebras and translation equivariant factorisation coalgebras on $\Ab^1$. We also write down analogues replacing $\Ab^1$ with $\Gm$ or an elliptic curve $E$, which is known to experts but we could not find in the literature.

\subsection{Recollections on chiral Koszul duality} 

\subsubsection{} If $A$ is a cocommutative coalgebra in $\Vect$, the subspace $\Prim A$ of primitive elements forms a Lie algebra. Conversely, if $\gk$ is a Lie algebra then its Chevalley-Eilenberg complex $\Ct(\gk)$ is a cocommtative coalgebra. These functors give an adjunction
$$\Ct_\sbt\ :\ \Lie(\Vect) \ \leftrightarrows\ \Eb_\infty\CoAg(\Vect)\ :\ \Prim[-1]$$
This Koszul duality between $\Lie$ and $\textup{CoCommCoAlg}$ was extended in \cite{FG} to the chiral setting, where they gave an adjoint equivalence
$$\Lie(\Dl\Md(\Ran^{ch} X)) \ \stackrel{\sim}{\leftrightarrows}\ \Eb_\infty\CoAg(\Dl\Md(\Ran^{ch}X))$$
which restricts to 
$$\Lie(\Dl\Md(\Ran^{ch} X)_X) \ \stackrel{\sim}{\leftrightarrows}\ \Eb_\infty\FactCoAg^{\textup{st}}(\Dl\Md(\Ran^{ch}X))$$
where $(-)_X$ denotes the category of $\Dl$-modules with support inside $X\subseteq \Ran X$, and $(-)^{\textup{st}}$ denotes  factorisation coalgebras whose coproduct is an equivalence. Here $X$ is a separated scheme of finite type over $k$, a field of characteristic zero. 

\subsubsection{Explicit description} \label{sec:ExplicitChiralKoszul} Let $\Fl$ be a strong cocommutative factorisation coalgebra over $\Ran^{ch}X$, as a D module
$$\Prim \Fl\ =\ \Fl_X$$
where we have implicitly pushed forward along $X\to \Ran X$. The map defining the Lie algebra structure
$$\mu \ :\ \Fl_X[-1]\otimes^{ch}\Fl_X[-1]\ \to\ \Fl_X[-1]$$
is thus supported on $\Ran^{\le 2}X$, and thus is uniquely determined by its restriction to $X^2$. To be explicit, this is a map
$$\mu\vert_{X^2}[1]\ =\ (\nu,-\nu)\ :\ \jmath_*\jmath^*(\Fl_X\boxtimes \Fl_X)^{\oplus 2}\ \to\ \Fl_X[1]$$
where $\iota:X\to X^2$ is the diagonal and $\jmath$ its open complement. The antisymmetry in the two factors is forced by it being a map of D modules on the Ran space. To finish: the map $\nu$ is the boundary map in the Mayer-Vietoris sequence
$$\nu\ :\ \jmath_*\jmath^!\Fl_{X^2}\ \to\ \iota_*\iota^!\Fl_{X^2}[+1],$$
where we use that $\Fl$ is a strong factorisation coalgebra to give $\jmath^*\Fl_{X^2}=\jmath^*(\Fl_X\boxtimes \Fl_X)$, and $\iota^!\Fl_{X^2}=\Fl_X$. To summarise:
\begin{prop}
  Let $\Fl$ be a strong factorisation coalgebra over $\Ran^{ch}X$ for $X$ a smooth variety over an algebraically closed field $k$ of characteristic zero. Then it is uniquely determined by the boundary map
  $$\jmath_*\jmath^*(\Fl_X\boxtimes \Fl_X)\ \stackrel{}{\to}\ \Fl_{X^2\setminus X}\ \to\ \Fl_{X}[1]$$ 
  where the first map is induced by the factorisation coproduct. 
\end{prop}  

\subsubsection{Remark} Note that we consider $\Eb_\infty$-coalgebras in chiral Koszul duality. However, we will define commutative vertex algebras as strong factorisation $\Eb_\infty$-\textup{algebras}, because vertex algebras are unital. Because the categories of strong nonunital $\Eb_\infty$-algebras and -coalgebras are identified, nonunital vertex algebras may be viewed as elements of either category. 

\subsection{Rational, Multiplicative and Elliptic Vertex algebras}

We now write down Theorems describing strong $\Eb_n$-algebras in over the Ran space of one dimensional algebraic groups $G$ in terms of very explicit data. When 
\begin{itemize}
  \item $n=1$ these are (\textit{associative}) \textit{vertex algebras},
  \item $n=2$ these are \textit{braided commutative vertex algebras},
  \item $n=\infty$ these are \textit{commutative vertex algebras}, which are elsewhere usually just called \textit{vertex algebras}.
\end{itemize}
As a background factorisation $\Eb_n$-category we take eithe the category of weakly equivariant D-modules or the category of comodules $H\CoMd$ of a vertex quantum group, where in the latter case $n\le 2$ unless its spectral $R$-matrix is symmetric. Thus let  
$$G\ =\ \Ga,\Gm,E$$
be the affine line, multiplicative group or an elliptic curve over $k$. Choosing a formal coordinate $z$ of the origin allows us to identify the formal completion $\hat{G}$ as a formal group scheme as $\Spf k[[z]]$ with the group addition law denoted $+_\gk$. We may make a change of coordinates $u=u(z)$ which induces an isomorphism $\Spf k[[z]]\simeq \Spf k[[u]]=\hat{\Gb}_a$ as formal group schemes, see section \ref{ssec:ExamplesOfTheoremGaGmEVA}.\footnote{$u$ is called the \textit{logarithm}.} The associated vertex algebras will be called \textit{rational} (or $\varnothing$), \textit{multiplicative} or \textit{elliptic vertex algebras}. 

\begin{theorem} \label{thm:GaGmEVA} \emph{\cite{BD}}
  The category of strong commutative factorisation algebras
  $$\Eb_\infty\FactAg^{\textup{st}}(\Dl\Md(\Ran G)^{G,w})$$
  is equivalent to the category of vector spaces with endomorphism $(V,T)$ with a map 
  $$Y\ :\ V\otimes V\ \to\ V((z))$$
  satisfying $[T,Y(\alpha,z)]=\partial_z Y(\alpha, z)$, skew commutativity:
  \begin{equation} \label{eqn:SkewCommutativity}
    Y(\alpha,z)\beta\ =\ e^{u(z)T}Y(\beta,-_\gk z)\alpha,
  \end{equation}
  and  associativity: 
  \begin{equation} \label{eqn:VertexAssociativity}
    Y(Y(\alpha,z_1)\beta,z_2)\gamma\ =\ Y(\alpha,z_1+_\gk z_2)Y(\beta,z_2)\gamma.
  \end{equation}
\end{theorem}

In particular, it follows that the above categories for various $G$ are equivalent. We first prove

\begin{prop}\label{prop:GmGaEAssociative}
  The category of strong (associative) factorisation algebras
  $$\Eb_1\FactAg^{\textup{st}}(\Dl\Md(\Ran G)^{G,w})\ \simeq\ \Eb_1\FactAg^{\textup{st}}(\Dl\Md(\Ran_a G)^{G,w})$$
  is equivalent to the data in Theorem \ref{thm:GaGmEVA} but omitting the skew commutativity condition (\ref{eqn:SkewCommutativity}).
\end{prop}
\begin{proof}
  Given a strong factorsation algebra structure $m:A\otimes^{ch}A \stackrel{\sim}{\to} \jmath^*\cup^*A$, we get from the Mayer Vieotris boundary map
  \begin{center}
  \begin{tikzcd}[row sep = {30pt,between origins}, column sep = {85pt,between origins}]
  j_*j^*(A\otimes^{ch}A)_2\ar[rd,dashed,"Y"]\ar[d,"\wr"',"m"] & \\ 
  j_*j^*A_2\ar[r,"\delta"] & i_*i^!A_2[1]&[-32pt]\ \simeq\ i_*A_1[1]
  \end{tikzcd}
  \end{center} 
  where $(-)_n$ denotes the restriction to $G^n$, and $i,j$ are the embedding of the diagonal and its complement inside $G\times G$. As discussed above, the above restriction uniquely determines $m$. Rewriting a little, we have map 
  $$Y\ :\ j_*j^*(B\boxtimes B)\ \to\ i_*B$$
  where $B=A_1[-1]$. By Lemma \ref{lem:BDLem} below, this map of D-modules is equivalent to a map 
  $$Y\ :\ B\boxtimes B\ \to \ (\hat{\iota}_*B)\otimes j_*\Ol.$$
  Now, we use that $B$ is weakly $G$-equivariant to give that 
  $$B\ =\ V\otimes\Ol_G$$
  as O-modules, and the D-module structure is determined by the action of the generating left invariant vector field $\xi$ on $G$, which acts as $T\otimes\id+\id\otimes\xi$ for $T$ some endomorphism of $V$. Thus if we take diagonal-antidiagonal coordinates on $G^2$:
  \begin{equation} \label{eqn:DiagonalAntidiagonalCoordinates}
   G_\Delta\times G_{-\Delta}\ \stackrel{\sim}{\to}\ G^2\hspace{15mm} (t,s)\ \mapsto\ (ts,ts^{-1}) 
  \end{equation}
  the data of $Y$ becomes equivalent to a $\xi$-linear map of vector spaces $Y : V\otimes V \to V((z))$. The $\xi$-linearity is equivalent to the condition $[T,Y(\alpha,z)]=\partial_z Y(\alpha,z)$. It thus remains to show that the associativity condition (\ref{eqn:VertexAssociativity}) is equivalent to the associativity of $m$. The left hand side of (\ref{eqn:VertexAssociativity}) fits into the diagram 
\begin{center}
\begin{tikzcd}[row sep = {35pt,between origins}, column sep = {120pt,between origins}]
j_3^*(A_1\boxtimes A_1\boxtimes A_1)\ar[d,"m\otimes \id","\wr"']\ar[rd,dashed,"Y\boxtimes \id"] & & \\ 
j_3^*(j^*A_2\boxtimes A_1)\ar[r,"\delta\boxtimes\id"] \ar[d,"m","\wr"']&j_3^*(j^*i_*A_1\boxtimes A_1)[1]\ar[rd,dashed,"Y"] \ar[d,"m","\wr"']& \\
j_3^*A_3\ar[r,"j_3^*(\delta\boxtimes \id)"]&j_3^*(i\times\id)_*A_2[1]\ar[r,"\delta"]&j_3^*((i\times\id)_*i_*A_1)[2]
\end{tikzcd}
\end{center}
and likewise for the right side of (\ref{eqn:VertexAssociativity}). These isomorphisms take place in the domain of 
$$j_3\ :\ G^3\setminus(\Delta_{12}\cup \Delta_{23})\ \to\ G^3.$$
We are then done from $m(\id\times m) =m(m\times\id)$, and because $\delta (\delta\boxtimes\id)  = \delta (\id\boxtimes\delta)$ from the iterative Mayer-Vietoris Lemma \ref{lem:IterativeMayerVietoris}, below. 
\end{proof}

\begin{proof}[Proof of Theorem \ref{thm:GaGmEVA}] It remains to show that skew commutativity (\ref{eqn:SkewCommutativity}) is equivalent to commutativity of $m$. Thus, assume that $m$ is commutative, i.e. that the following commutes
  \begin{equation}\label{fig:VertexCommutativity}
    \begin{tikzcd}[row sep = {30pt,between origins}, column sep = {85pt,between origins}]
      j_*j^*(A_1\boxtimes A_1)\ar[dr,"\sim","m"'] \ar[rr,"\sigma","\sim"']&[-35pt]&[-35pt]j_*j^*(A_1\boxtimes A_1)\ar[rd,dashed,"Y"]\ar[dl,"\sim"',"m"] & \\ 
    &j_*j^*A_2\ar[rr,"\delta"] && i_*i^!A_2[1]&[-32pt]\ \simeq\ i_*A_1[1]
    \end{tikzcd}
  \end{equation}
where $\sigma$ is induced by the commutative factorisation category structure on $\Dl\Md_{\Ran G}$, i.e. it exhibits that $A_1\boxtimes A_1$ is a $\Sk_2$-equivariant D-module on $(G\times G)_\circ$. 

To understand this diagram, we take $G_{-\Delta}$-invariant  sections of (\ref{fig:VertexCommutativity}) as quasicoherent sheaves to get 
\begin{center}
  \begin{tikzcd}[row sep = {20pt,between origins}, column sep = 20pt]
    V^{\otimes 2}\otimes \Gamma(G_{-\Delta},\jmath_*\Ol) \ar[r,"\sim"',"\sigma"]&V^{\otimes 2}\otimes \Gamma(G_{-\Delta},\jmath_*\Ol)\ar[r,"Y"]& V\otimes \Gamma(G_{-\Delta},\delta_1)
  \end{tikzcd}
\end{center}
where here $\jmath:G_{-\Delta}-1\to G_{-\Delta}$. It remains to understand the map $\sigma$. Let us restrict to $G=\Ga$, the other cases following similarly using the isomorphism $\hat{G}\simeq\hat{\Gb}_a$. Choosing a coordinate $z-w$ on the formal neighbourhood of the identity in $G_{-\Delta}$, and writing $\partial_{z-w}$ for the generating left invariant vector field on $G_{-\Delta}$ with $\partial_{z-w}(z-w)=2$, we have commuting diagram
\begin{center}
\begin{tikzcd}[row sep = {35pt, between origins}, column sep = {30pt}]
  V^{\otimes 2}((z-w))\ar[d,"\sigma_{V,V}"] \ar[ddd,"\sigma"',bend right = 20,{xshift=-40pt}] &\\
 V^{\otimes 2}((z-w))\ar[r,"Y"] \ar[d,":e^{(w-z)\partial_{z-w}}:"]&V((z-w))\ar[d,":e^{(w-z)\partial_{z-w}}:"] \\ 
 V^{\otimes 2}((z-w))\ar[r,"Y"]\ar[d,"e^{\frac{1}{2}(z-w)\partial_{z+w}}"] &V((z-w))\ar[d,"e^{\frac{1}{2}(z-w)\partial_{z+w}}"]\\
 V^{\otimes 2}((z-w))\ar[r,"Y"] &V((z-w))
\end{tikzcd}
\end{center}
where $:e^{(w-z)\partial_{z-w}}:=\sum_{k\ge 0}\frac{1}{k!}(w-z)^k\partial_{z-w}^k$ by Taylor's Theorem sends $(z,w)\mapsto (w,z)$. Both squares commute because $Y$ is a map of D-modules, and cell on the left commutes because 
$$\sigma\ =\ (e^{\frac{1}{2}(z-w)T}\otimes e^{\frac{1}{2}(z-w)T}) \cdot (z\leftrightarrow w)\cdot\sigma_{V,V}.$$
Finally, since on the right $\partial_z,\partial_w$ act as 
$$\partial_z\ =\ T\otimes \id + \id\otimes\partial_{z-w}, \hspace{15mm}\partial_w\ =\ T\otimes\id-\id\otimes\partial_{z-w},$$
it follows that the composite of the two vertical maps on the right is $e^{(z-w)T}\cdot(z\leftrightarrow w)$, and so 
$$Y(\alpha,z-w)\beta\ =\ Y(z-w)\cdot\sigma(\alpha\otimes\beta)\ =\ e^{(z-w)T}Y(\beta,-(z-w))\alpha.$$
\end{proof}

We now state the classification Theorem for bradided commutative vertex algebras. Let $H$ be a vertex quantum group with vertex $R$-matrix $R(z,w)=R(z-w)$, corresponding to a factorisation quantum group as in Theorem \ref{thm:VertexQuantumGroupFactorisationQuantumGroup}.  
\begin{prop} \label{prop:GaGmEVABraidedCommutative}
  The category of strong braided commutative factorisation algebras 
  $$\Eb_2\FactAg^{\textup{st}}(H\Md_{\Dl\Md}(\Ran G)^{G,w})$$
  is equivalent to the data in Theorem \ref{thm:GaGmEVA}, but all data being $H$-linear and with (\ref{eqn:SkewCommutativity}) replaced by braided skew commutativity:
  $$Y(\alpha,z)\beta\ =\  e^{u(z)T}Y(R(z)\cdot \beta\otimes\alpha,-_\gk z).$$
\end{prop}
\begin{proof}
  Proceeds as in the proof of Theorem \ref{thm:GaGmEVA}, but using the braided commutativity of $m$
  \begin{equation}\label{fig:BraidedVertexCommutativity}
    \begin{tikzcd}[row sep = {30pt,between origins}, column sep = {85pt,between origins}]
      j_*j^*(A_1\boxtimes A_1)\ar[dr,"\sim","m"'] \ar[rr,"\sigma \cdot R","\sim"']&[-35pt]&[-35pt]j_*j^*(A_1\boxtimes A_1)\ar[rd,dashed,"Y"]\ar[dl,"\sim"',"m"] & \\ 
    &j_*j^*A_2\ar[rr,"\delta"] && i_*i^!A_2[1]&[-32pt]\ \simeq\ i_*A_1[1]
    \end{tikzcd}
  \end{equation}
  instead of commutativity (\ref{fig:VertexCommutativity}). 
\end{proof}

Finally, we write down a classification Theorem for vertex bialgebras. Remember that we view D-modules on the equivariant Ran space as a classical braided factorisation category with respect to its factorisation products $\otimes^{ch}$ and $\otimes$.

\begin{prop} \label{prop:GaGmEVABialgebra}
  The category of factorisation bialgebras 
  $$\Eb_1^{\otimes^{ch}}\FactAg^{\textup{st}}\left(\Eb_1^{\otimes}\FactCoAg(\Dl\Md(\Ran G)^{G,w})\right)\hspace{50mm}$$
  $$\hspace{50mm}\ \simeq\ \Eb_1^{\otimes^{ch}}\FactAg^{\textup{st}}\left(\Eb_1^{\otimes}\FactCoAg(\Dl\Md(\Ran_a G)^{G,w})\right)$$ 
  is equivalent to an associative vertex algebra $V$ (as in Proposition \ref{prop:GmGaEAssociative}) with a coassociative coalgebra structure 
  $$c\ :\ V\ \to\ V\otimes V, \hspace{15mm} \epsilon\ :\ V\ \to\ k,$$
  compatible as the following commuting 
  \begin{center}
  \begin{tikzcd}[row sep = {30pt,between origins}, column sep = {20pt}]
  V\otimes V\ar[rr,"Y"] \ar[d,"c\otimes c"]&[-10pt]&V((z))\ar[d,"c"] \\ 
  (V\otimes V)\otimes( V\otimes V)\ar[r,"\sim"',"\sigma_{23}"] &V\otimes V\otimes V\otimes V\ar[r,"Y\otimes Y"] & V\otimes V((z))
  \end{tikzcd}
  \end{center}
  with $|0\rangle$ grouplike and $T$ a coderivation for $c$, $\epsilon|0\rangle =1$ and $\epsilon\cdot  T = 0$.
\end{prop}
\begin{proof}
  Let $A$ be a factorisation bialgebra with coproduct $c:A\to A\otimes A$ with respect to the diagonal factorisation structure, which is thus determined by its restrictions $c_I:A_I\otimes A_I\to A_I$ to $X^I$. It is easy to show that $c_1$ together with the compatible associative factorisation algebra structure gives the data in the Proposition, so it remains to show that $c_1$ determines the remaining $c_I$. Thus, consider 
  \begin{equation}\label{eqn:VBialgDiag}
    \begin{tikzcd}[row sep = {30pt,between origins}, column sep = {20pt}]
    A_I\otimes A_I\ar[r] &\jmath_*\jmath^*(A_I\otimes A_I)\ar[r] &\iota_*\iota^!(A_I\otimes A_I)[1] \\ 
    A_I\otimes A_I\ar[r]\ar[u,"c_I"] &\jmath_*\jmath^*(A_I\otimes A_I)\ar[r]\ar[u,"\jmath_*\jmath^!c_I"] &\iota_*\iota^!(A_I\otimes A_I)\ar[u,"\iota_*\iota^!c_I"][1]
    \end{tikzcd}   
  \end{equation}
  where $\jmath:(X^{I_1}\times X^{I_2})_\circ\to X^{I_1\sqcup I_2}$ is the open inclusion corresponding to partition $I=I_1\sqcup I_2$ into two nonempty finite sets, and $\iota$ is its closed complement. 
  
  Using the strong factorisation algebra structure, one can show that $\jmath_*\jmath^! c_I$ and $\iota_*\iota^!c_I$ are determined by $c_{J}$ for $|J|<|I|$. Moreover, the horizontal maps in (\ref{eqn:VBialgDiag}) are determined by $A$ and its D-module structure, i.e. are independent of $c_I$. It follows that $c_I$, as the fibre of the natural transformation $\jmath_*\jmath^!\to\iota_*\iota^![1]$ applied to $c_I:A_I\otimes A_I\to A_I$, is uniquely determined by the $c_J$ with $|J|<|I|$. Thus by induction $c$ is determined by $c_1$.
\end{proof}

\subsubsection{Examples of Theorem \ref{thm:GaGmEVA}} \label{ssec:ExamplesOfTheoremGaGmEVA}
When $G=\Ga$ we have $\pm_\gk=\pm$ and the above reduces to the  ordinary notion of (commutative) vertex algebra from cite \cite{FBZ}. When $G=\Gm$ we have $u(z)=\log(1+z)$, and so writing $s=1+z$ we have
 $$Y(\alpha,s)\beta\ =\ s^T Y(\beta,s^{-1})\alpha, \hspace{10mm} Y(Y(\alpha,s_1)\beta,s_2)\gamma\ =\ Y(\alpha,s_1s_2)Y(\beta,s_2)\gamma,$$
 which matches the definition of (non-equivariant commutative) multiplicative vertex algebra as defined in \cite{Li}. When $G=E$ is an elliptic curve, an explicit formulas for $+_\gk$ may be found in section IV.1 of \cite{Si} and the logarithm in section IV.4.

\subsubsection{Completions} 
We need a technical Lemma due to \cite{BD}. Let $Z\subseteq X$ be a smooth divisor inside smooth variety. We then have adjoint functors
\begin{center}
\begin{tikzcd}[row sep = 20pt, column sep = 20pt]
\Dl\Md(Z_{\textup{et}})\ar[r,"\hat{\iota}_*",shift left] \ar[d,"\oblv"]&\ar[l,"\iota^*",shift left] \Dl\Md(X_{\textup{et}})\ar[d,"\oblv"]\\ 
\Dl\Md(Z) \ar[r,"\iota_*",shift left] &\ar[l,"\iota^*",shift left]   \Dl\Md(X)
\end{tikzcd}
\end{center}
and the rightwards functors are fully faithful, with essential images the category $\Dl\Md(X_{\textup{et}})_Z^\wedge$ of D-modules on $X_{et}$ which are complete along $Z$, and $\Dl\Md(X)_Z$ of D-modules with support inside $Z$, see section 9 of \cite{Bu}. The diagram involving $\iota^*$ commutes, but the diagram involving $\iota_*$ commutes only up to non-invertible $2$-morphism:
$$\oblv \tilde{\iota}_*\Fl\ \to\ \oblv \tilde{\iota}_*\Fl\, \otimes \iota_*k\ \simeq\ \iota_*\oblv \Fl $$
where $k$ denotes the unit D-module. Then we have as in \cite[9.1.4]{Bu},

\begin{lem}\label{lem:BDLem}
  For any $\Al_X\in \Dl\Md(X_{\textup{et}})$ and $\Bl_Z\in \Dl\Md(Z_{\textup{et}})$, the tautological maps
  $$\Hom(\Al_X\otimes\jmath_*\Ol,\iota_*\Bl_Z) \ \stackrel{\sim}{\leftarrow}\ \Hom(\Al_X\otimes\jmath_*\Ol,\hat{\iota}_*\Bl_Z\otimes\jmath_*\Ol)\ \stackrel{\sim}{\to}\ \Hom(\Al_X,\hat{\iota}_*\Bl_Z\otimes\jmath_*\Ol)$$
  are isomorphisms, where the tautological maps are induced by
  \begin{center}
  \begin{tikzcd}[row sep = {30pt,between origins}, column sep = {70pt,between origins}]
   \Al_X\otimes \jmath_*\Ol\ar[r,dashed]\ar[rd,dashed]& \iota_*\Bl_Z\\ 
  \Al_X\ar[r,dashed] \ar[u]&\hat{\iota}_*\Bl_Z\otimes\jmath_*\Ol \ar[u]
  \end{tikzcd}
  \end{center}
  All Hom spaces are taken in $\Dl\Md(X)$.  
\end{lem}

\subsubsection{Iterative Mayer Vietoris} 
If we have a complementary open and closed embedding, we obtain the boundary map of the associated Mayer Vietoris sequence
\begin{center}
\begin{tikzcd}[row sep = {30pt,between origins}, column sep = {45pt,between origins}]
Z\ar[r,"\iota"] & X & &[-25pt] &[15pt]\\[-17pt]
& & &\rightsquigarrow &  \jmath_*\jmath^!\stackrel{\delta}{\to} \iota_*\iota^![1]\\[-17pt] 
 \varnothing \ar[uu]\ar[r]& U\ar[uu,"\jmath"'] & & &
\end{tikzcd}
\end{center}
We now consider a version of this where we have sequences of closed and open subspaces 
$$Z_1\ \subseteq \ Z_2\ \subseteq\ X \hspace{15mm} U_1\ \subseteq\ U_2\ \subseteq\ X.$$
We do not assume that these are complementary. Instead, if the top left and bottom right of the following pullback squares are \textit{complementary} closed and open embeddings, we get 
\begin{center}
\begin{tikzcd}[row sep = {30pt,between origins}, column sep = {45pt,between origins}]
Z_1\ar[r] &Z_2\ar[r] & X &[10pt]&[100pt]\\ 
\varnothing\ar[r]\ar[u] & Z_2\cap U_2\ar[r,"\iota"]\ar[u,"\jmath"] & U_2\ar[u] & \rightsquigarrow & \jmath_{1*}\jmath_1^!\ \to \ \jmath_{2*}\iota_*\iota^!\jmath_2^![1]\ =\ \iota_{2*}\jmath_*\jmath^!\iota_2^![1]\ \to\ \iota_{1*}\iota_1^![2]\\
\varnothing\ar[r]\ar[u]& \varnothing\ar[r]\ar[u] & U_1\ar[u] &&     
\end{tikzcd}
\end{center}
where $\iota_i:Z_i\to X$ and $\jmath_i: U_i\to X$ are the closed and open embeddings. 

\begin{lem} \label{lem:IterativeMayerVietoris}Let $Z,Z'\subseteq X$ be two closed subspaces, with complementary opens $U,U'$. Then the functors for 
  $$Z\cap Z'\ \subseteq \ Z, \hspace{4mm} U\cap U'\ \subseteq\ U' \hspace{10mm} \text{and}\hspace{10mm}Z\cap Z'\ \subseteq \ Z', \hspace{5mm} U\cap U\ \subseteq\ U$$
  are the same.
\end{lem}
\begin{proof}
   We have the following diagram whose rows and columns are complementary open and closed embeddings
\begin{center}
  \begin{tikzcd}[row sep = {35pt,between origins}, column sep = {55pt,between origins}]
  U\cap Z\ar[r,"\overbar{\jmath}_1'"]\ar[d,"\overbar{\iota}_2'"]&Z'\ar[d,"i_2'"]&Z\cap Z'\ar[l,"i_1'"']\ar[d,"i_1"]\\
  U\ar[r,"j_2'"]& X& Z\ar[l,"i_2"']\\
  U\cap U'\ar[u,"j_1'"']\ar[r,"j_1"]& U'\ar[u,"j_2"']& Z\cap U\ar[l,"\overbar{\iota}_2"']\ar[u,"\overbar{\jmath}_2"']
  \end{tikzcd}
  \end{center}
It follows that we have commuting diagram 
\begin{center}
\begin{tikzcd}[row sep = {35pt,between origins}, column sep = {20pt}]
   {j_2}_*{j_2}^!{i_2}_*{i_2}^![1]\ar[r] &{i_2'}_*{i_2'}^!{i_2}_*{i_2}^![2] \\ 
j_{2*}j_2^!{j_2'}_*{j_2'}^!\ar[r]\ar[u] & {i_2'}_*{i_2'}^! {j_2'}_*{j_2'}^![1]\ar[u]
\end{tikzcd}
\end{center}
Thusthe two functors under consideration are the same since this square commutes. 
\end{proof}

Likewise,
\begin{lem}
  If $Z_1,...,Z_n$ are closed subspaces with open complements $U_1,...,U_n$ then the functor 
  $$\jmath_*\jmath^!\ \to\ \iota_*\iota^![n],$$
  induced by iteratedly applying the boundary map of the Mayer Vietoris sequence, is independent of the order of the $Z_i$ chosen. Here, $\iota$ and $\jmath$ are the embeddings of $Z_1\cup...\cup Z_n$ and $U_1\cap...\cap U_n$ respectively.
\end{lem}

  The relevant example for Theorem \ref{thm:GaGmEVA} is 
  \begin{center}
  \begin{tikzcd}[row sep = {30pt,between origins}, column sep = 10pt]
  X^3\setminus (\Delta_{12}\cup \Delta_{23})\ar[r] &(X^3\setminus \Delta_{23})\ar[r] &X^3 \\ 
  \varnothing \ar[u]\ar[r] & \Delta_{12}\setminus \Delta_{123}\ar[r]\ar[u] & \Delta_{12}\ar[u]\\
  &\varnothing\ar[r]\ar[u]& \Delta_{123}\ar[u]
  \end{tikzcd}
  \end{center}
  where the two closed subspaces $\Delta_{ij}\subseteq X^3$ are the closed locus of points $(x_1,x_2,x_3)$ with $x_i = x_j$. The outer square is clearly unchanged if we swap $\Delta_{12}$ and $\Delta_{23}$, so we may apply Lemma \ref{lem:IterativeMayerVietoris}.

\end{document}